\documentclass[10pt,bezier,amstex]{article}  
\usepackage{color}              
\usepackage{graphicx}
\usepackage[]{amsmath}
\RequirePackage{amsthm,amsmath,amsfonts,amssymb}
\usepackage{amssymb}
\usepackage{hyperref,bbm}
\definecolor{MyDarkBlue}{rgb}{0,0.08,0.50}
\definecolor{BrickRed}{rgb}{0.65,0.08,0}
\usepackage{amsmath}

\hypersetup{
colorlinks=true,       
    linkcolor=MyDarkBlue,          
    citecolor=BrickRed,        
    filecolor=red,      
    urlcolor=cyan           
}

\usepackage{enumerate}
\usepackage[mathscr]{eucal}
\usepackage{graphicx}
\usepackage{latexsym}
\usepackage{amssymb}
\usepackage{psfrag}
\usepackage{epsfig}
\usepackage{tikz}
\usetikzlibrary{automata,positioning}

\usepackage{mhchem}

\include{psbox}

\usepackage{amsfonts}

\usepackage{graphicx}

\usepackage{amsfonts}

\usepackage{color}

\newtheorem{Theorem}{Theorem}
\newtheorem{Assumption}{Assumption}
\newtheorem{Lemma}{Lemma}[section]
\newtheorem{lemma}{Lemma}[section]
\newtheorem{Proposition}[Lemma]{Proposition}

\newtheorem{Corollary}[Lemma]{Corollary}

\newtheorem{remark}[Lemma]{Remark}

\DeclareMathOperator*{\esssup}{ess\,sup}

\newcommand{\s}{\quad}

\newcommand{\eps}{\epsilon}

\newcommand{\R}{\mathbb{R}}

\newcommand{\normal}{\operatorname{N}}

\newcommand{\non}{\nonumber}

\newcommand{\beq}{\begin{eqnarray*}}
\newcommand{\eeq}{\end{eqnarray*}}
\newcommand{\beqn}{\begin{eqnarray}}
\newcommand{\eeqn}{\end{eqnarray}}

\newcommand{\bt}{\begin{Theorem}}
\newcommand{\et}{\end{Theorem}}

\newcommand{\bas}{\begin{Assumption}}
\newcommand{\eas}{\end{Assumption}}

\newcommand{\be}{\begin{equation}}
\newcommand{\ee}{\end{equation}}

\newcommand{\sign}{\operatorname{sign}}

\newcommand{\mb}[1]{{\mathbf #1}}

\newcommand\ci{\perp\!\!\!\!\perp}

\newcommand{\Var}{{\rm Var}}

\numberwithin{equation}{section}

\begin{document}

\tikzset{every node/.style={auto}}
 \tikzset{every state/.style={rectangle, minimum size=0pt, draw=none, font=\normalsize}}
 \tikzset{bend angle=20}

	\author{Abhishek Pal Majumder
\thanks{University of Reading}
	}

\title{Long time behavior of semi-Markov modulated perpetuity and some related processes}
\maketitle
\begin{abstract}
Examples of stochastic processes whose state space representations involve functions of an integral type structure $$I_{t}^{(a,b)}:=\int_{0}^{t}b(Y_{s})e^{-\int_{s}^{t}a(Y_{r})dr}ds, \s t\ge 0$$ are studied under an ergodic semi-Markovian environment described by an $S$ valued jump type process $Y:=(Y_{s}:s\in\R^{+})$ that is ergodic with a limiting distribution $\pi\in\mathcal{P}(S)$. Under different assumptions on signs of $E_{\pi}a(\cdot):=\sum_{j\in S}\pi_{j}a(j)$ and tail properties of the sojourn times of $Y$ we obtain different long time limit results for $I^{(a,b)}_{}:=(I^{(a,b)}_{t}:t\ge 0).$ In all cases mixture type of laws emerge which are naturally represented through an affine stochastic recurrence equation (SRE) $X\stackrel{d}{=}AX+B,\,\,$ $X\ci (A, B)$. Examples include explicit long-time representations of pitchfork bifurcation, and regime-switching diffusions under semi-Markov modulated environments, etc.
\end{abstract}



\section{Introduction and the integral process} We study some dynamical systems (both stochastic and deterministic) whose parameters characterizing their overall stability, are not fixed constants, instead, they switch their values driven by some hidden jump-type process mechanisms. These models are abundantly seen in directions of econometrics, finance, biological systems, statistical physics etc. We focus on regime-switching diffusion processes, pitchfork bifurcation examples where the key parameter is stochastically varying through different regimes. Common theme of all these processes is that they have an explicit state-space representation in terms of a following integral type of  bivariate process $(\Phi^{(a)},I^{(a,b)}):=\big((\Phi^{(a)}_{t},I^{(a,b)}_{t}):t\ge 0\big)$ where

\beqn
(\Phi^{(a)}_{t},I^{(a,b)}_{t}):= \Big(e^{-\int_{0}^{t}a(Y_{r})dr},\int_{0}^{t}b(Y_{s})e^{-\int_{s}^{t}a(Y_{r})dr}ds\Big)\,\,\,\,\,\forall t>0,\label{modelint}
\eeqn

for some given functions $a,b:S\to\R$ and some $S$-valued process $Y:=(Y_{t}: t\ge 0)$ where $S$ is a countable set. Let the common filtration of the regime-switching dynamics be defined by $\mathcal{F}^{Y}_{t}:=\sigma\{Y_{s}:s\le t\}.$ Main content of this article is to understand how various parametrizations of the dynamics $Y$ influence the time-asymptotic behaviors of $(\Phi^{(a)},I^{(a,b)})$ and their applications in different examples.

The integral process $I^{(a,b)}$ is known as perpetuity in the actuarial science. Several previous works, 
e.g.~Gjessing and Paulsen \cite{gjessing1997present}, 
Bertoin and Yor \cite{bertoin2005exponential}, 
Maulik and Zwart \cite{maulik2006tail}, 
Behme and Lindner \cite{behme2015exponential}, 
Chapter $2$ of Iksanov \cite{iksanov2016renewal}, 
Zhang et al. \cite{zhang2016long} and 
Feng et al. \cite{feng2019exponential} studied this perpetuity in line with exponential functionals of Lévy processes defined by  $J_{t} := \int_{0}^{t} e^{-P_{s}} dQ_{s},$
where \( (P_{t}, Q_{t}: t \ge 0) \) jointly constitutes a Lévy process in an appropriate state space. These works addressed distributional properties, including moments and asymptotic behaviors of tail events, etc. In our context, this implies if we let $\int_{0}^{t}a(Y_{s})ds$ be a Levy process instead of an integral of a semi-Markov process, subsequent analysis gets relatively simpler due to the stationary, independent increment properties. In a recent work (Behme \& Sideris \cite{behme2019exponential}), authors considered $(P_{t}, Q_{t})$ as a Markov additive process and give sharp existence uniqueness conditions motivated by Alsmeyer \cite{alsmeyer2017stability}, that entails the foundation of stability for affine stochastic recurrence equation in a Markovian environment. Contrary to Markov-switching, semi-Markovian switching framework driven by $Y$ will exhibit a natural auto-regulatory control on the law of sojourn time that depends on the future regime state as well as the current one, which is extremely useful for implementing control-related formulations into the latent state dynamics. The theoretical foundation of the Semi-Markov process and its time asymptotic behaviors were explored long back in Cinlar \cite{cinlar1969markov} along with Cinlar \cite{ccinlar1967queues},\cite{ccinlar1969semi} etc. \cite{korolyuk1975semi} reviewed subsequent developments along with other references therein. Still, for ``perpetuity" type of path-based functionals $I^{(a,b)},$ there were no exact long-term results available in the literature for a semi-Markovian $Y,$ to the best of our knowledge.

Regime-switching processes have a rich history of applications in econometrics, actuarial mathematics, mathematical biology, and quantitative finance. Examples of such uses include Ang and Timmermann \cite{ang2012regime}, BenSaida \cite{BenSaida15}, Fink et al. \cite{fink2017regime}, and Genon-Catalot et al. \cite{GCJL00}, who have used it in the context of stochastic volatility modeling in financial markets. Additionally, Zhang et al. \cite{zhang2016long} have considered stochastic interest rate models with Markov switching, while Hardy \cite{Hardy01}, Lin et al. \cite{LinKenHailiang09}, and Shen and Kuen Siu \cite{ShenSiu13} have studied the long-term behavior of stock returns and bond pricing. Similar to quantitative finance, such processes are commonly used in actuarial science for solvency investigations (e.g., Abourashchi et al. \cite{abourashchi2016pension}), mortality modeling (e.g., Gao et al. \cite{GMLT15}), and in the context of disability insurance (e.g., Djehiche and L\"ofdahl \cite{DjehicheLofdahl18}). In all of the above references, the sojourn times for the regime-switching process have been considered to be Exponential (to make $Y$ Markovian), although it may not be ideal. 
\begin{figure}[h]
\centering
\includegraphics[width=14cm, height=5.5cm]{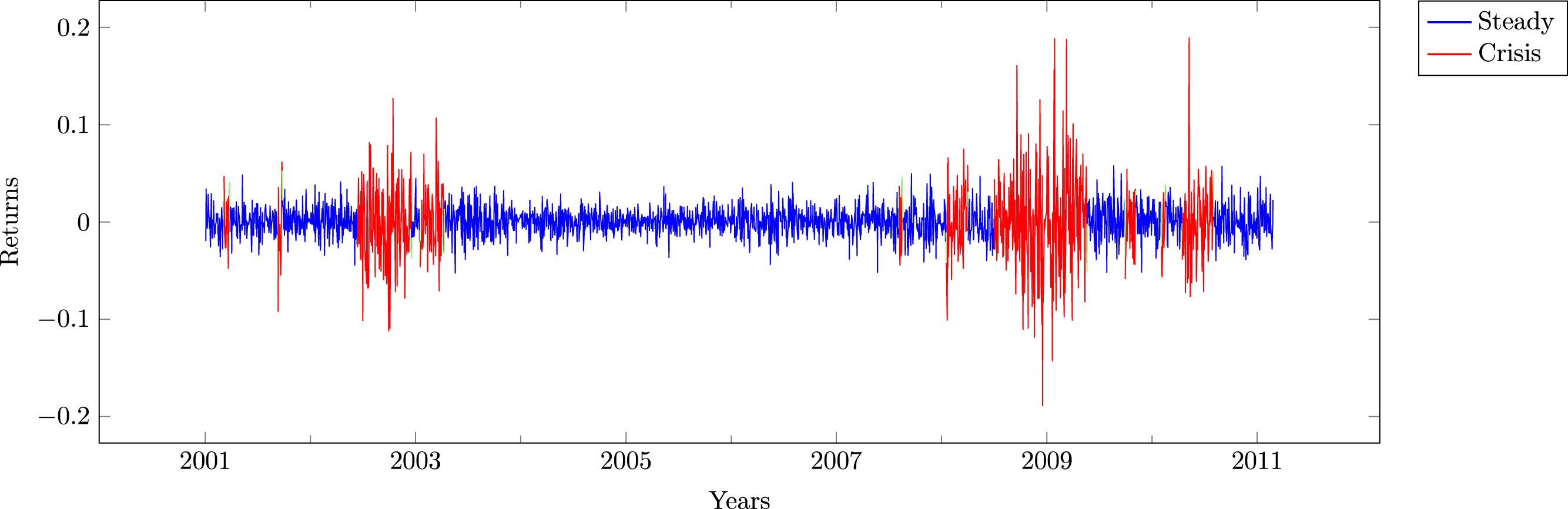}
\caption{BNP Paribas stock price daily returns and Markov-switching environments predicted in \cite{salhi2016regime} assuming  $|S|=2$.  
}\label{fig:regime}
\end{figure}

For example, in Salhi et al. \cite{salhi2016regime}, assuming that there is a Markovian regime-switching with only two latent states, the underlying environments were predicted in the time evolution of BNP stock returns (as shown in the Figure \ref{fig:regime}). It is observed that the sojourn time of the ``steady" period from 2003 Q1 to 2007 Q2 is more than four years, which is around seven times higher than the average length of other ``steady" state sojourn times. This event is highly unlikely when the underlying regime switching dynamics is Markovian, as the sojourn times of a particular state are iid exponentials (hence has an exponentially decaying tail making that event significantly rare). However, the aforementioned event may not be rare if the distribution of sojourn times is modeled as Pareto or any other heavy-tailed distributions, as it fits the notion of having a few large jumps in a handful sample size. This motivates the use of a general semi-Markovian regime-switching process, as the sojourn time distributions can be flexible with any laws. Several recent works such as Apergis et al. \cite{Apergis2019Decoding}, Qin et al. \cite{qin2024robust}, Wang et al. \cite{Wang2021Bear} have demonstrated the superior performance of semi-Markovian regime-switching models compared to their Markov-switching counterparts in various applications including energy consumption forecasting, stock price modeling, and historical S\&P 500 analysis. Availability of computational tools, such as the R packages hsmm \cite{bulla2010hsmm} and mhsmm \cite{o2011hidden}, has facilitated the practical implementation of semi-Markov switching models. Despite empirical success, a rigorous theoretical understanding of the long-term behavior of semi-Markov modulated processes remains an open problem. While foundational properties, such as existence \& uniqueness, recurrence \& transience, ergodicity are well-established in the works of Yin \& Zhu \cite{yin2010switching}, Mao \cite{mao2006stochastic}, Shao \cite{shao2015ergodicity}, and Zhang \cite{zhang2016long} in contexts of general hybrid regime switching models including long-term asymptotics of switching diffusions, jump processes etc., explicit asymptotic characterizations are lacking, making it difficult to assess the sensitivity of limiting distributions to latent switching dynamics $Y$.  Beyond finance, semi-Markovian models have been applied in disability insurance \cite{stenberg2007algorithmic} to relax Markovian assumptions and in interest rate modeling \cite{d2013semi} to allow sojourn times to depend on both current and future states. These applications highlight the growing importance of processes with semi-Markovian latent structures, and a need for a deeper theoretical exploration of their time-asymptotic properties.

Beyond providing long-term characterizations of $(\Phi^{(a)}_{}, I^{(a,b)}_{})$, we consider two examples where the processes of interest after some transformations, admit an exact state-space representation either in terms of a linear combination of $\{\Phi^{(a)}_{t}, I^{(a,b)}_{t}\}$ or $\{\Phi^{(a)}_{t}, |I^{(ma,b)}_{t}|^{1/m}\}$. Corollary \ref{conn1} aids in establishing the limit results of such functionals, within the divergent framework. First application where we use the long-term behavior of $(\Phi, I),$ is the deterministic pitchfork bifurcation model in polar coordinates $(\rho,\phi):=((\rho_{t},\phi_{t}):t\ge 0)$ satisfying 
\beqn
\frac{d\rho_{t}}{dt} &=& \rho_{t} \left( a(Y_{t}) - b(Y_{t}) \rho_{t}^{2} \right), \quad \frac{d\phi_{t}}{dt} = 1, \label{model2}
\eeqn
where $a(\cdot)$ and $b(\cdot)$ are functions of the underlying semi-Markov process $Y$, rather than fixed parameters. Note that if $a$ and $b$ are constants with $b>0$, then for $a \leq 0$, the state $0$ would be a stable steady state, while for $a > 0$, it becomes an unstable one. If $a(\cdot)$ switches its values changing its signs over time, it is interesting to see under varying stochastic environment $Y$, how the limiting measure of $\rho$ behaves around $0,$ if that exists. For $|S| > 2$ and semi-Markovian $Y$, the long-term behavior of $\rho$ is also not known in the literature, which we will address explicitly in Theorem \ref{Ex4}, in different cases. 

Another major application, discussed in the Section \ref{secApplications}, involves the long-term behavior of an $\R$-valued regime-switching Ornstein-Uhlenbeck process, denoted by $X = (X_{t})_{t \geq 0}$, under a semi-Markovian $Y$. It is defined as the weak solution to the SDE
\beqn
dX_{t} = \big(c(Y_{t}) - a(Y_{t}) X_{t}\big) dt + b(Y_{t}) dL_{t}, \quad Y_{0} = y_{0} \in \R, \label{model01}
\eeqn
where $a(\cdot), b(\cdot), c(\cdot): S \to \R$ are arbitrary functions, and $(L_{t})_{t \geq 0}$ is either a standard stable-$\alpha$ Lévy process or a Brownian motion, independent of the background $S$-valued process $Y$.

In context of Markov modulated Ornstein-Uhlenbeck process, Zhang \& Wang \cite{ZW2017stationayOU} studies for two-states (i.e when $|S|=2$) regime process. It is possible to compute the stationary distribution under the stable regime since the possible transitions are only of $1\to 2$ and $2\to 1$ type which makes studying $(\Phi, I)$ tractable. However if $|S|> 2,$ there will be many different transitions between other states, and generalizing that approach becomes infeasible. The regime-switching process is also deeply connected with Piecewise Deterministic Markov processes (PDMP) and related applications. In similar contexts, ~Cloez and Hairer \cite{hairer2010convergence} studied a class of general regime switching Markov processes where sufficient conditions ensure exponential ergodicity of the observed process under  Wasserstein distance. Benaim et al. \cite{benaim2016lotka} showed that survival or extinction of competing
species in a Lotka-Volterra model is heavily influenced by switching rate parameters using Lyapunov drift type criteria (Section 4 or Theorem 4.1 of \cite{benaim2016lotka}). That generalized a set of Lyapunov-type approaches for cases where no explicit temporal representation of the overall process is available.  In more recent works, Behme and Sideris \cite{behme2020markov} characterized the stationarity of a Markov-modulated generalized Ornstein-Uhlenbeck process with applications in Risk theory. However, the results on time-limiting behaviour are not available in explicit form.

In contrast to all works mentioned above, \cite{lindskog2018exact} addressed a similar problem to \eqref{model01} in a diffusion setting with Markov switching, providing several long-time limit results. However, when the sojourn times of $Y$ are non-exponential and depend on the future state (determined by an independent state-determining Markov chain), the analysis for non-Markovian $Y$ becomes more complex. This is because the residual time $(t - \tau_{g_{t}^{j}})$ in \eqref{z_to_r} does not factor out as it does in the Markovian case (as noted after display (6.9) in \cite{lindskog2018exact}), creating difficulties. To address this, a modification via a limiting argument is required, exploiting the regenerative structure of $Y$, as performed in step $2$ of Lemma \ref{lem01}. Lemma \ref{lem01} is crucial for our semi-Markovian framework, as it can be used to study time asymptotic behaviors of any functionals of $Y$ satisfying the conditions of $H_{t}$ in \eqref{key}. Lemma \ref{lem01} is then applied to establish all subsequent time asymptotic results in Theorems \ref{T1}- \ref{Ex1}. Unlike the Markovian switching context \cite{lindskog2018exact}, long-term limit results here in the divergent cases (i.e, $E_{\pi}a(\cdot)\le 0$) are found to be further subdivided based on whether \[\sigma_{j}^{2}:=\Var\Big(\int_{\tau_{0}^{j}}^{\tau_{1}^{j}}\big(a(Y_{s})-E_{\pi}a(\cdot)\big)ds\Big),\] is finite, infinite or equal to $0$ (where $\tau_{0}^{j},\tau_{1}^{j}$ are time instants of two successive hitting times of $Y$ to state $j$ (validated under Assumption \ref{As0})), which is mostly due to the general sojourn time distributions of $Y.$ The flexibility in controlling the sojourn time dynamics gives diverse modeling opportunities and will have an additional impact on the scaling limits in all cases, which along with different applications, are the main contributions of this article.

{\color{blue}As semi-Markov modulated processes, all results in these examples are novel, but two surprising takeaways are  following:
\begin{enumerate}[(a)]
\item In the Pitchfork Bifurcation example \eqref{model2}, under stability domain (i.e., when $E_{\pi}a(\cdot)>0$), the limiting distribution $\mathcal{L}(\rho^{2}_{\infty})$ is not known for $|S|>2$. Specifically, if \( \inf_{j \in S} a(j) < 0 \), how \( \mathcal{L}(\rho^{2}_{\infty}) \) behaves under a stable domain and concentrates around \( 0 \) has remained an open question. Along with an explicit characterization of $\mathcal{L}(\rho^{2}_{\infty})$ (see Theorem \ref{Ex4}(A)), we found that there exist constants $\nu^{*}$ and $c > 0$, explicitly determined by the dynamics of $Y$, such that the following small-ball asymptotic limit holds:

\beqn
\lim_{\epsilon \to 0^{+}} \epsilon^{-\nu^{*}} P\big[\rho^{2}_{\infty} < \epsilon\big] = c,\label{GK1}
\eeqn
(as shown by setting $x = 1/\epsilon$ in Remark \ref{rmk}) which is quite surprising, as it shows  how $ P\big[\rho^{2}_{\infty} < \epsilon\big]$ scales when $\epsilon$ is close to $0.$ The constant $c$ is related to the Goldie-Kesten constant \cite{bura2016kestenG}.
\item In diffusion example \eqref{model01} (with $c=0$), or more generally in \eqref{exe3} or in \eqref{SaS}, we observe that the noise process influences weak limit of the long-term behavior of the original process only in following two cases:
\begin{itemize}
    \item in the stable case, i.e., when $E_{\pi}a(\cdot)>0$, corresponding to Theorem \ref{Ex1} (A), and
    \item in a specific divergent case where $\sigma_{j}^{2}=0, \forall j\in S$ (as described in Theorem \ref{Ex1} (D)).
\end{itemize}
In other divergent cases (i.e., where $\sigma_{j}^{2} < \infty$ or $\sigma_{j}^{2} = \infty$ for all $j\in S$, corresponding to Theorem \ref{Ex1} (B) and (C), respectively), given certain integrability conditions, the weak limits of original process after scaling and transformation are found to be universal and they do not depend on the noise processes (see Remark \ref{RR1*}), which is quite fascinating. This entails how impactful features of the latent dynamics are, in governing different time-asymptotic behaviours of the overall process $X.$
\end{enumerate}
 }

This paper is organized as follows. Section \ref{sec2} sets the structure of the regime process $Y$ along with preliminary assumptions and necessary definitions. In Section \ref{secStable}, exact long-time characterizations for the process $(\Phi^{(a)},I^{(a,b)})$, and different generalizations, are presented under assumptions corresponding to the stable regime. Section \ref{sec_unstable} presents long time results for a scaled version of $(\Phi^{(a)},I^{(a,b)})$ in the divergent and critical  regimes i.e, when $E_{\pi}a(\cdot)\le 0.$ Section \ref{secApplications} contains applications of the findings in Sections \ref{secStable} and \ref{sec_unstable} to \eqref{model2} \& \eqref{model01}. In Section \ref{Lem2} we present the Lemma \ref{key} regarding asymptotic independence of the residual time and a specific functional $(H_{t}:t\ge 0)$ satisfying \eqref{key} which is the heart of this work as it is used in proving all the key results. We end the article through Section \ref{secConclusion} by giving some conclusions, future directions, and some open questions. The proofs will be found in Section \ref{proof} and in the subsequent sections.

\section{Preliminaries and Assumptions}\label{sec2}
 $Y$ is an $S$-valued, where $S$ is a countable set, jump-type semi-Markov process that will be introduced here along with some underlying structures and relevant assumptions. Consider an $S\times \R_{+}$ valued time-homogeneous Markov chain $(J,T):=(J_{n},T_{n})_{n\in\mathbb{N}}$ following the transition kernel
\beqn
Q_{ij}(t):= P\big[J_{n+1}=j, T_{n+1}\le t| J_{n}=i, T_{n}=s\big]\s\forall\,\,\, n\in\mathbb{N}, \s\forall\,\,\, s,t\ge 0,\label{Q}
\eeqn
with initial distribution $P[J_{0}=i,T_{0}=0]=p_{i}\,\,\,\forall\,\, i\in S$, with $\sum_{i\in S}p_{i}=1,$ and further structures of $\big(Q_{ij}(\cdot):i,j\in S\big)$ are following. As in definition of $Q$ in \eqref{Q} suggests that  $ Q_{ij}(t)$ does not depend on $T_{n}=s$ in the conditioning part so one can write $$Q_{ij}(t)=P\big[J_{n+1}=j, T_{n+1}\le t| J_{n}=i\big]$$ without any problem. Let $\sum_{i=0}^{n}T_{i}$ be denoted by $S_{n}$ for $n\in\mathbb{N}.$

We define the the $S$-valued regime switching dynamics $Y:=(Y_{t}:t\ge 0)$ as \beqn
Y_{t}:= J_{N_{t}}\s\text{where}\s N_{t}:=\esssup\Big\{n:\,\,S_{n}\le t\Big\}.\label{Ydef}
\eeqn
Clearly \eqref{Ydef} suggests that $T_{n}$ denotes the sojourn time of $Y$ at state $J_{n-1}$ for $n\ge 1,$ and $T_{0}:=0.$ Let $J:=(J_{n}:n\in\mathbb{N})$ marginally be a discrete-time $S$-valued Markov chain (independent of everything else) with transition kernel $P:=\big(P_{ij}: i,j\in S\big)$ such that 
$$P_{ij}:=P\big[J_{n+1}=j| J_{n}=i\big]\s \text{and}\s P_{ii}:=0,\,\,\forall i\in S,$$
representing the marginal dynamics of the embedded chain $J.$ Note that $Y$ can be seen as the marginal of the Markov renewal process $(S_{n},J_{n})$ which is alternatively known as the semi-Markov process in the literature \cite{cinlar1969markov} (which is not necessarily Markovian).

 Let 
\beqn
\mathcal{G}:=\big\{F_{ij}\in\mathcal{P}(\R_{+}): i,j\in S,\,\, i\neq j,\,\,\text{and}\,\, P_{ij}>0,\,\,F_{ij}(0)=0, F_{ij}(\infty)=1 \big\}\label{G}
\eeqn
 be the set of probability distributions (having no atoms at $0$) of sojourn times of $Y$ where particularly $F_{ij}$ denotes the law of the sojourn time of $Y$ at state $i$ when the future state will be $j\in S,$ or more precisely 
$$F_{ij}(t):=P\big[T_{n+1}\le t| J_{n}=i,J_{n+1}=j\big].$$
Clearly with above definitions one can decompose $Q_{ij}(t)$ in following fashion
\beqn
Q_{ij}(t)&=&P\big[J_{n+1}=j| J_{n}=i\big] P\big[T_{n+1}\le t| J_{n+1}=j, J_{n}=i\big]\non\\
&=& P_{ij}F_{ij}(t).\label{multiplic}
\eeqn
Given the transition kernel $P,$ the set $\mathcal{G}$ and the initial distribution $\{p_{i}:i\in S\}$ the dynamics of $(J,T)$ is completely specified and well defined as well as the process $Y$ defined in \eqref{Ydef}. Observe that semi-Markov process $Y$ exhibits much more dependence structure than continuous time Markov chain which is obtained as a special case $F_{ij}(t):= 1-e^{-\lambda_{i}t}$ for some $\lambda_{i}>0$ for all $i,j\in S.$ The semi-Markov formulation gives much more flexibilities and control on the sojourn time distributions that may depend on the next state with more general than just Exponential distribution. This opens up a huge number of modelling opportunities for underlying regime switching mechanisms.  

The semi-Markov process $Y$ will be called conservative at state $j\in S,$ if one has $$P[N_{t}<\infty\mid Y_{0}=j]=1,\s\forall t\in[0,\infty),$$
which is a regularity condition (Page 150 of \cite{cinlar1969markov}) implying that $Y$ is non-explosive or that the number of transitions in a finite time is finite with probability $1$ starting from a specific state $j\in S.$ 

\bas\label{As0} Following assumptions are about the dynamics of the process \((J, T)\):

\begin{enumerate}[(a)]
    \item The embedded chain \(J\) is a time-homogeneous, irreducible, aperiodic, and positive recurrent Markov chain, taking values in \(S\) and evolving in discrete time. It is independent of all other variables. The transition matrix \(P\) satisfies the properties \(P_{ii} = 0\), \(P_{ij} \in [0,1]\), and \(\sum_{j \in S} P_{ij} = 1\) for all \(i, j \in S\). Additionally, the chain admits a unique stationary distribution \(\mu \in \mathcal{P}(S)\), which satisfies the measure-valued equation \(\mu = \mu P\).

    \item The set \(\mathcal{G}\) is defined such that, for all \(i,j \in S\), the distribution functions \(F_{ij}\) satisfy \(F_{ij}(0) = 0\) and contain no atoms in \(\mathbb{R}_{\geq 0}\). In other words, each \(F_{ij}\) is absolutely continuous with respect to the Lebesgue measure on \(\mathbb{R}_{\geq 0}\). For each pair \(i,j \in S\), define the mean of the random variables distributed according to \(F_{ij}(\cdot)\) as
    \[
    m_{ij} := \int_{0}^{\infty} \big(1 - F_{ij}(t)\big) dt.
    \]
    We assume that \(\sup_{i,j \in S} m_{ij} < \infty\), ensuring that the expected sojourn times are uniformly bounded.

    \item The conditional probability 
    \[
    P\big[J_{n+1} = j, T_{n+1} \leq t \mid J_n = i, T_n = s\big]
    \]
    does not depend on \(s\). As a result, \(s\) is omitted in the notation for the distribution function \(Q_{ij}(t)\) used in the definition \eqref{Q}.

    \item The semi-Markov process \(Y\) is conservative for all \(j \in S\), meaning that it does not explode in finite time.
\end{enumerate}

\eas

For any $i,j\in S,$ let $W_{ij}$ be defined as the $F_{ij}$ distributed random variable denoting the sojourn time of $Y$ at state $i$ conditioned on the future state $j$. Then we define the distribution of sojourn time at state $i$ unconditional on the future value as
$$P[T_{n+1}<t| J_{n}=i]= \sum_{j\in S}P_{ij}F_{ij}(t)=:F_{i}(t)$$ denoted by $F_{i}(\cdot)$ and denote the corresponding random variable by $W_{i}.$ For each $i\in S,$ by $m_{i}$ we denote the mean of $W_{i}$ as
$$m_{i}:=EW_{i}=\int_{0}^{\infty}(1-F_{i}(t))dt=\int_{0}^{\infty}(1-\sum_{j\in S}P_{ij}F_{ij}(t))dt=\sum_{j\in S}P_{ij}m_{ij}.$$ Assumption \ref{As0}(a) suggests that for each $j\in S,$ $\mu_{j}>0$ and positive recurrence for $J$ implies that there exist hitting times at state $j$ defined as
$$N_{0}^{j}:=\inf\{n\ge 0: J_{n}=j\},\s N_{k}^{j}:= \inf\{n>N_{k-1}^{j}: J_{n}=j\}\s $$ for $\,\, k\ge 1.$ Denote the set $\{N_{k-1}^{j},N_{k-1}^{j}+1, \ldots, N_{k}^{j}-2,N_{k}^{j}-1\}$ (which is a successive array of integers)  by $A_{k}^{j}$  (the $k$-th recursion from state $j$ to itself by the process $J$) for each $k\ge 1.$ Positive recurrence of $J$ suggests that $E|A_{k}^{j}|=E[N_{k}^{j}-N_{k-1}^{j}]<\infty.$  By $(i,j)$ we denote a successive transition from state $i$ to $j$ (or $i\to j$) for the $J$ chain if $P_{ij}>0.$ Set of all possible transitions can be described as $\{(i,j)\in S\times S: P_{ij}>0\}$. 

The recursion path $A_{k}^{j}$ can be alternatively represented as the set $A_{k}^{j}=\{j,i_{1},i_{2},\ldots,i_{l_{1}}\}$ where $l_{1}=[N_{k}^{j}-N_{k-1}^{j}]-1,$ representing the sequence of transitions $j\to i_{1}\to i_{2}\to\ldots\to i_{l_{1}}\to j.$ By $\widetilde{A}_{k}^{j}$ we denote the set of tuples denoting the successive state transitions of  $A_{k}^{j}$ expressed as
\beqn
\s\s\widetilde{A}_{k}^{j}=\{(j,i_{1}),(i_{1},i_{2}),(i_{2},i_{3}),\ldots,(i_{|A_{k}^{j}|-1},j): i_{l}\in S,\&\,\, i_{l}\neq j\,\,\,\forall l=1,\ldots,|A_{k}^{j}|-1\}.\label{Atilde}
\eeqn
Observe that $\widetilde{A}_{k}^{j}$ has a one to one correspondence with $A_{k}^{j}$ as both represent the 
 $k-$th recursion cycle from state $j$ to itself.

Note that $P[T_{0}=0]=1$ and for any arbitrary $u_{1},\ldots,u_{n}>0$ and any arbitrary $i_{1},\ldots,i_{n+1}\in S$ one has the following identity for any $n\in \mathbb{N}$
\begin{align}
&P\big[T_{1}\le u_{1},\ldots,T_{n}\le u_{n} |J_{0}=i_{1},J_{1}=i_{2},\ldots, J_{n}=i_{n+1}\big]\non\\ &\s\s\stackrel{(a)}{=}\frac{\prod_{k=1}^{n}Q_{_{i_{k}i_{k+1}}}(u_{k})}{\prod_{k=1}^{n} P_{i_{k}i_{k+1}}}\s\non\\
&\s\s\stackrel{(b)}{=}\prod_{k=1}^{n}F_{_{i_{k}i_{k+1}}}(u_{k})\label{Prod}
\end{align}
where $(a)$ and $(b)$ in above equalities hold due to \eqref{Q} and \eqref{multiplic} (as a result of Assumption \ref{As0} (c)) respectively. Display \eqref{Prod} suggests that conditioned on path $\{J_{0}=i_{1},J_{1}=i_{2},\ldots,J_{n}=i_{n+1}\}$ the sojourn times $\{T_{1},\ldots T_{n}\}$ in corresponding states are independent, which also prompts us to consider the regenerating renewal intervals for $Y$ as well. We define the hitting times of $Y$ at state $j\in S,$ denoted by $\{\tau_{k}^{j}:k\ge 0\}$ as
$$\tau_{0}^{j}:=\sum_{i=0}^{N_{0}^{j}}T_{i},\s \tau_{k}^{j}:=\sum_{i=0}^{N_{k}^{j}}T_{i}\s\forall k\ge 1,$$
and denote the last hitting time to state $j$ before $t$ by 
\beqn
g_{t}^{j}:=\max(\sup\{n: \tau_{n}^{j}\le t\},0),\s \sup\emptyset=-\infty.\label{g_{t}}
\eeqn
  It is clear that any functionals of Y in $\{[\tau_{i}^{j},\tau_{i+1}^{j}):i\ge 0\}$ behave identically and independently as $\Big\{\sum_{i=N_{k-1}^{j}+1}^{N_{k}^{j}}T_{i}:k\ge 1\Big\}$ are iid due to the fact that $\mathcal{G}$ is a fixed set and $\{A_{k}^{j}\}_{k\ge 1}$ are regenerating sets for $J$ at state $j\in S.$

For $Y$ denote the $k$-th recursion interval at state $j\in S$ by $\mathfrak{I}_{k}^{j}:=[\tau_{k-1}^{j},\tau_{k}^{j}).$  From Assumption \ref{As0}(b) it follows that
\beqn
E|\mathfrak{I}_{k}^{j}|&=&E\sum_{n=1}^{\infty}T_{n} 1_{\{(J_{n-1},J_{n})\in \widetilde{A}_{k}^{j}\}}\non\\
&=&E \Big[\sum_{n=1}^{\infty}E\big[T_{n}| J_{n}=i_{n},J_{n-1}=i_{n-1}\big]1_{\{(J_{n-1},J_{n})\in \widetilde{A}_{k}^{j}\}}\Big]\non\\
&\le& \bigg(\sup_{i,j\in S}m_{ij}\bigg) E|\widetilde{A}_{k}^{j}|<\infty,\s\forall j\in S,\s \forall k\in \mathbb{N},\non
\eeqn
as a consequence of $E\big[T_{n}| J_{n}=i_{n},J_{n-1}=i_{n-1}\big]=m_{i_{n-1},i_{n}}\le \sup_{i,j\in S}m_{ij}$, uniformly for all $(i_{n-1},i_{n})\in S^{2}.$

For each tuple $(i_{k}, i_{k+1})$, the sojourn time of the semi-Markov process $Y$ during the transition from state $i_{k}$ to state $i_{k+1}$ is denoted by $\widetilde{F}_{i_{k}i_{k+1}}$ and follows the distribution $F_{i_{k}i_{k+1}}$. It is possible for two tuples, say $(i_{k}, i_{k+1})$ and $(i_{l}, i_{l+1})$, representing different transitions in $\widetilde{A}_{k}^{j}$, to be identical; that is, for $k \neq l$, we may have $i_{k} = i_{l}$ and $i_{k+1} = i_{l+1}$. Even in this case, we treat the two tuples $(i_{k}, i_{k+1})$ and $(i_{l}, i_{l+1})$ as distinct elements of $\widetilde{A}_{k}^{j}$. Correspondingly, the sojourn times $\widetilde{F}_{i_{k}i_{k+1}}$ and $\widetilde{F}_{i_{l}i_{l+1}}$ are considered independent random variables, each distributed according to $F_{i_{k}i_{k+1}}$.

 For any event $A$ by $P_{_{i}}[A]$ we denote $P[A\mid Y_{0}=i].$

\begin{remark}\label{R0}
Assumption \ref{As0}(c) is essential for applying the regenerating-renewal argument to the process $Y$. As a result, any functionals of $Y$ within the intervals $\{[\tau_{i}^{j},\tau_{i+1}^{j}):i\ge 0\}$ behave identically and independently for each $j\in S.$ If Assumption \ref{As0}(c) does not hold, the process $Y$ would fail to regenerate independently after reaching state 
$j$, as the sojourn time at state $j$ would be influenced by the time spent in the previous state before transitioning to $j$. This would break the independence required for the regeneration of $Y$ in the intervals $\{[\tau_{i}^{j},\tau_{i+1}^{j}):i\ge 0\}$.
\end{remark}

Theorem 7.14 of \cite{cinlar1969markov} (Page 160) suggests that if $Y$ is an irreducible, recurrent (The term `Persistent' is used in \cite{cinlar1969markov} for recurrent), aperiodic semi-Markov process along with $m_{i}:=\sum_{j\in S}P_{ij}m_{ij}<\infty,$ then
\beqn
\lim_{t\to\infty}P_{i}[Y_{t}=j]=\frac{\mu_{j}m_{j}}{\sum_{k\in S}\mu_{k}m_{k}},\s\forall i\in S\s\label{lim1}
\eeqn
which follows from conditions of Assumption \ref{As0}. By $\pi$ denote the probability vector $(\pi_{i}:i\in S)$ such that 
\beqn
\pi_{j}:= \frac{\mu_{j}m_{j}}{\sum_{k\in S}\mu_{k}m_{k}}\,\,\forall j\in S.\label{pi}
\eeqn
 From Assumption \ref{As0} it follows that $\pi_{j}>0, \forall j\in S.$ We loosely describe a process to be ergodic when it converges to a limiting distribution that does not depend on the initial distribution and hence a semi-Markov process $Y$ having all properties above (irreducible, recurrent, aperiodic, $m_{k}<\infty,\,\,\forall k\in S$) is ergodic.

Regenerative property of $Y$ suggests that any specific functional of $Y$ in $\{[\tau_{i-1}^{j},\tau_{i}^{j}): i\ge 1\}$ behave identically independent random variables. Denote the sigma algebra $\sigma\{Y_{t}: t\in [\tau_{i-1}^{j},\tau_{i}^{j})\}$ by $\mathcal{H}_{i}$ for each $i\ge 1.$ In this setting we have the following result. For any time $t>0,$ denote $t-\tau^{j}_{g_{t}^{j}}, \tau_{g_{t}^{j}+1}^{j}-t$ by $A_{j}(t),B_{j}(t)$ as respectively backward residual time and forward residual times at state $j\in S$. Clearly $A_{j}(t)+B_{j}(t)= \tau_{g_{t}^{j}+1}^{j}-\tau^{j}_{g_{t}^{j}},$  length of the regenerating interval containing $t.$ Results from \cite{cinlar1969markov},\cite{asmussen2008applied}  suggest if $E[\tau^{j}_{2}-\tau_{1}^{j}]<\infty,$ (which follows from $E|\mathfrak{I}_{1}^{j}|<\infty,$ ensured under Assumption \ref{As0}) then both $A_{j}(t),B_{j}(t)$ are $O_{_{P}}(1)$ as both quantities
\beqn
P[A_{j}(t)>x]\to \frac{\int_{x}^{\infty}P\big[|\mathfrak{I}_{k}^{j}|>y\big]dy}{E|\mathfrak{I}_{k}^{j}|},\s P[B_{j}(t)>x]\to\frac{\int_{x}^{\infty}P\big[|\mathfrak{I}_{k}^{j}|>y\big]dy}{E|\mathfrak{I}_{k}^{j}|},\label{eRes}
\eeqn
as $t\to\infty$. 

The following proposition is a consequence of $Y$ exhibiting renewal regenerative property.

\begin{Proposition}\label{Pr1}
Under Assumption \ref{As0}, for any $t>\tau_{0}^{j},$
\begin{enumerate}[(i)]
\item the distribution of $\tau_{g^{j}_{t}+2}^{j}-\tau_{g^{j}_{t}+1}^{j}$ is identical as $\tau^{j}_{1}-\tau_{0}^{j}.$
\item any functional of $Y$ in $\{s: s\ge \tau^{j}_{g_{t}^{j}+1}\}$ is independent of $g_{t}^{j}$ and identically distrubuted as the same functional of $Y$ on $\{s:s\ge \tau_{0}^{j}\}.$ 
\item Let $Y^{(1)}:=\big(Y^{(1)}_{t}:t\ge 0\big), Y^{(2)}:=\big(Y^{(2)}_t:t\ge 0\big)$ be two arbitrary stochastic processes. Conditioned on any event $K_{t},$ suppose marginals of $Y^{(1)}_{t},Y^{(2)}_{t}$ are respectively $\sigma\{\mathcal{H}_{i}:i\le g^{j}_{t}\}$ and $\sigma\{\mathcal{H}_{i}: i\ge g^{j}_{t}+2\}$ measurable for any $j\in S$. Then following holds conditioned on $\{Y_{t}=j'\}$ 
$$P[Y^{(1)}_{t}\in A ,Y^{(2)}_{t} \in B\mid K_{t},Y_{t}=j']=P[Y^{(1)}_{t}\in A \mid K_{t},Y_{t}=j'] P[Y^{(2)}_{t} \in B\mid  K_{t}],$$
for arbitrary events $A,B$ for any $j'\in S.$
\end{enumerate}
\end{Proposition}

Consider a probability measure $\pi$ on a countable set $S$ such that $\pi(A)=\sum_{i\in A}\pi_{i}$ for any $A\in \mathcal{B}(S)$ and a set of probability measures $\{\mu_j:j\in S\}$. Then  
$$
\sum_{j\in S} \delta_{U}(\{j\})Z_{j}, \quad\text{where}\quad U \perp (Z_j)_{j\in S}, \quad U\sim \pi, \quad Z_j\sim \mu_j,
$$
is a random variable whose distribution is the mixture distribution $\sum_{j\in S} \pi_{j}\mu_{j}$. 

Let $\{R_{n}\}_{n\ge 1}$ be an arbitrary sequence of random variables and $N^{*}_{t}$ is a $\mathbb{N}$ valued stopped random time. Suppose there exists a random variable $R_{\infty},$ such that $R_{n}\stackrel{d}{\to}R_{\infty}$  as $n\to\infty$. Let $N^{*}_{t}$ be any ``stopped random time" (defined in the common probability space where $\{R_{n}\}_{n\ge 1}$ is defined) such that there exists an increasing function $c(\cdot)$ such that $c(t)\to\infty$ and $\frac{N^{*}_{t}}{c(t)}\stackrel{P}{\to}1$ as $t\to\infty.$ In this context \emph{Anscombe’s contiguity condition} (Gut \cite{gut2009stopped}, p.~16) is useful for establishing weak convergence of the process $(R_{N^{*}_{t}}:t\ge 0)$ as $t\to\infty$  which is following:

Given $\epsilon>0$ and $\eta>0$ there exists $\delta>0$ and $n_{0},$ such that 
\beqn
P\bigg(\max_{\{m: |m-n|<n\delta\}}|R_{m}-R_{n}|>\epsilon\bigg)<\eta,\s\forall\s n> n_{0}.\label{Anscombe}
\eeqn 
If \eqref{Anscombe} is satisfied by $\{R_{n}\}_{n\ge 1}$, then it is sufficient to conclude $$R_{N^{*}_{t}}\stackrel{d}{\to}R_{\infty}\s\text{as}\s t\to\infty.$$
 The notion of ``stopped random time" is not same as the ``stopping time" in the literature.

For a given bivariate random variable $(A,B)$, the following time series is referred to as a stochastic recurrence equation (in short SRE, also referred to as random coefficient AR$(1)$) 
\begin{align*}
Z_{n+1}=A_{n+1}Z_{n}+B_{n+1}\quad\text{with}\quad (A_{i},B_{i})\stackrel{\text{i.i.d}}{\sim} \mathcal{L}(A,B),\quad Z_{n}\ci (A_{n+1},B_{n+1}) 
\end{align*} 
for an arbitrary initial $Z_{0}=z_{0}\in \R$. Let $\log^{+}|a|:=\log(\max(|a|,1))$.  
If 
\beqn
P[A=0]=0, \quad E\log|A| <0,\s\text{and}\s E\log^{+}|B|<\infty, \label{sre_conv_conds}
\eeqn
then $(Z_{n})$ has a unique causal ergodic strictly stationary solution solving the following fixed-point equation in law: 
\beqn
Z\mathop{=}^{d}AZ+B\s\text{ with} \s Z\ci (A,B).\label{x=ax+b}
\eeqn
The condition $P[Ax+B=x]<1,$ for all $x\in \mathbb{R}$, rules out degenerate solutions $Z=x$ a.s. We refer to Corollary 2.1.2 and Theorem 2.1.3 in Buraczewski et al. \cite{buraczewski2016stochastic} for further details. For multivariate $d$-dimensional case the existence and uniqueness \cite{erhardsson2014conditions} of the distributional solution \eqref{x=ax+b} holds if
$$\prod_{i=1}^{n}A_{i}\,\,\stackrel{a.s}{\to}\,\, 0,\s\&\s \bigg(\prod_{i=1}^{n}A_{i}\bigg) B_{n+1}\,\,\stackrel{a.s}{\to}\,\, 0.$$

For any two functions $f,g$ the notation $f\sim g,$  implies that $\lim_{x\to\infty}\frac{f(x)}{g(x)}\to 1.$ Furthermore the notation $f\sim o(g),$ will imply $\lim_{x\to\infty}\frac{f(x)}{g(x)}\to 0.$ A random variable $X$ will be called regularly varying at $\infty$ with index $\alpha>0,$ if $P[X>x]\sim x^{-\alpha}L(x)$ holds where $L(\cdot)$ is any slowly varying function at $\infty$, i.e $\lim_{t\to\infty}\frac{L(xt)}{L(t)}=1,\forall x\in\R_{+}.$ If $X$ is a random variable regularly varying at $\infty$ with index $\alpha>0,$ then we have $E|X|^{\beta}<\infty$ for all $\beta<\alpha,$ and $E|X|^{\beta}=\infty$ for all $\beta\ge \alpha,$ for which we alternatively describe $X$ as heavy tailed with index $\alpha\ge 0$" (Proposition 1.4.6 of \cite{kulik2020heavy}).

 Let $\widetilde{X}$ be a non-negative regularly varying random variable with index $\alpha>0.$ For each $\epsilon>0$ there exists a finite constant $c_{\epsilon}$ such that for any $y\ge 0$ one has the following upper bound 
\beqn
\frac{P[y \widetilde{X}>x]}{P[\widetilde{X}>x]}\le c_{\epsilon}(1\vee y)^{\alpha}\label{bound}
\eeqn
holds that follows directly from Potter's bound (see Proposition 1.4.2 of \cite{kulik2020heavy}).

Define $\mathcal{S}_{\alpha}(\sigma,\beta,\mu)$ as the Stable random variable with $(\alpha,\sigma,\beta,\mu)\in (0,2]\times \R^{+}\times[-1,1]\times \R$ respectively denoting the index, scale, skewness and the shift parameter whose characteristic function is 
\beqn
\log Ee^{i\theta \mathcal{S}_{\alpha}(\sigma,\beta,\mu)}&=& -\sigma^{\alpha}|\theta|^{\alpha}\Big(1-i\beta(\sign(\theta))\tan\frac{\pi \alpha}{2}\Big)+i\mu\theta,\s\alpha\neq 1,\non\\
&=&-\sigma^{}|\theta|^{}\Big(1+i\beta(\sign(\theta))\frac{2}{\pi}\log|\theta|\Big)+i\mu \theta,\s\,\,\,\,\alpha=1\non
\eeqn
where $\sign(\theta)=1_{\{\theta>0\}}-1_{\{\theta<0\}}.$ In our work we are only concerned for $\alpha\in (1,2).$ Under above definition given $C_{\alpha}:=\bigg(\int_{0}^{\infty}x^{-\alpha}\sin xdx\bigg)^{-1},$ significance of $\beta\in(-1,1)$ can be represented through the following limits
\beqn
x^{\alpha}P\big[\mathcal{S}_{\alpha}(\sigma,\beta,\mu)>x\big]\to C_{\alpha}\frac{1+\beta}{2}\sigma^{\alpha},\s x^{\alpha}P\big[\mathcal{S}_{\alpha}(\sigma,\beta,\mu)<-x\big]\to C_{\alpha}\frac{1-\beta}{2}\sigma^{\alpha}\non
\eeqn
as $x\to\infty.$ If $\alpha\in (1,2),$ it follows that $\frac{1}{\sigma}\mathcal{S}_{\alpha}(\sigma,\beta,\mu)\stackrel{d}{=}\mathcal{S}_{\alpha}(1,\beta,\frac{\mu}{\sigma}).$ We recall Theorem 4.5.1 of \cite{whitt2002stochastic} for a version of stable central limit theorem for a sequence of iid random variables $\{X_{i}:i\ge 1\}$ with each of them is regularly varying RV$^{-\alpha},$ i.e there exists a slowly varying function $L(\cdot)$ (at $\infty$) such that $$P[|X|>x]\sim x^{-\alpha}L(x)\s \text{with} \s\lim_{x\to\infty}\frac{P[X>x]}{P[|X|>x]}\to \frac{1+\beta}{2}$$ for some $-1<\beta<1.$ Given this set up, define a scaling function $c_{n}$ such that 
\beqn
\lim_{n\to\infty}\frac{nL(c_{n})}{c_{n}^{\alpha}}=C_{\alpha},\s \text{which implies that}\s c_{n}= n^{\frac{1}{\alpha}}L_{0}(n)\,\,\,\text{(Thm 4.5.1 \cite{whitt2002stochastic})},\label{cn}
\eeqn  for some slowly varying function $L_{0}(\cdot)$ different from $L(\cdot)$ but it  depends on $L$ through \eqref{cn}. 
By $\mathcal{\bf{S}}_{\alpha,\beta}$ denote the $\alpha$-stable Levy process $\mathcal{\bf{S}}_{\alpha,\beta}:=\big(\mathcal{\bf{S}}_{\alpha,\beta}(t):t\ge 0\big)$  such that the increments follow Stable law (see display (5.30) of \cite{whitt2002stochastic}) as
$$\mathcal{\bf{S}}_{\alpha,\beta}(t+s)-\mathcal{\bf{S}}_{\alpha,\beta}(s)\stackrel{d}{=}\mathcal{S}_{\alpha}(t^{\frac{1}{\alpha}},\beta,0)\stackrel{d}{=}t^{\frac{1}{\alpha}}\mathcal{S}_{\alpha}(1,\beta,0)\s\forall s,t\ge 0.$$

\section{Main results on stable domain}\label{secStable}

In this section we study long-time behavior of the joint process $(\Phi^{(a)},I^{(a,b)},Y):=(\Phi^{(a)}_{t},I^{(a,b)}_{t},Y_{t})_{t\ge 0}$ ensuring convergence in distribution as $t\to\infty$. Following assumption will ensure the existence of a limiting distribution for $(\Phi^{(a)},I^{(a,b)},Y)$:  

\bas\label{as2} The $S$-valued process $Y$ and the functions $a,b:S\to\R$ satisfy 
\begin{itemize}
\item[(a)] $a(\cdot)$ is integrable with respect to $\pi,$ and $E_{\pi} a(\cdot)>0.$
\item[(b)] For every $j\in S$, $E\Big[\log^{+}\big|\int_{\tau^{j}_{0}}^{\tau^{j}_{1}}b^{}(Y_{s})e^{-\int_{s}^{\tau^{j}_{1}}a(Y_{r})dr}ds\big|\Big]<\infty.$
\end{itemize}
\eas

For each $j\in S,$ let $\pi^{*}_{j}$ be the distribution on $\R^{+}$ such that for any $x>0,$ 
\beqn
\pi^{*}_{j}[x,\infty):=\frac{\sum_{k\in S}P_{jk}\int_{x}^{\infty}(1-F_{jk}(y))dy}{m_{j}}.\label{pi_j^*}
\eeqn 

\begin{remark}\label{rem1} 
\begin{enumerate}[(i)]
\item Assumption \ref{as2} corresponds to the general condition \eqref{sre_conv_conds} for existence of a stationary solution to the stochastic recurrence equation
$$
Z_{j,n+1}=e^{-\int_{\tau^{j}_{n}}^{\tau^{j}_{n+1}}a(Y_{s})ds}Z_{j,n}+\int_{\tau^{j}_{n}}^{\tau^{j}_{n+1}}b^{}(Y_{s})e^{-\int_{s}^{\tau^{j}_{n+1}}a(Y_{r})dr}ds
$$ 
with affine solution of the form 
$$
Z_{j}\stackrel{d}{=}e^{-\int_{\tau^{j}_{0}}^{\tau^{j}_{1}}a(Y_{s})ds}Z_{j}+\int_{\tau^{j}_{0}}^{\tau^{j}_{1}}b^{}(Y_{s})e^{-\int_{s}^{\tau^{j}_{1}}a(Y_{r})dr}ds.
$$ 
Assumption \ref{as2}(b) defines integrability criteria for $b(\cdot).$ Observe that if $\sup_{j\in S} |a(j)|<\infty$ and $E_{\pi}|b^{}(\cdot)|<\infty$, then Assumption \ref{as2}(b) follows immediately from  Assumption \ref{As0} as a consequence of the inequalities $\log^{+}|ab|\le \log^{+}|a|+\log^{+}|b|$ and $\log^{+}|a|\le |a|$ for any $a,b$.
\item $\pi_{j}^{*}$ in \eqref{pi_j^*} is a size-biased distribution of averaged sojourn time at state $j\in S$. Note that if $F_{jk}(x)=1-e^{-\lambda_{j}x}\,\,\forall k,$ (corresponds to the case of CTMC) then $\pi^{*}_{j}[x,\infty)=e^{-\lambda_{j}x}$ for $x>0,$ as a consequence of the fact that Exponential distribution is the unique distributional fixed point of the size-biased transformation. If we consider a variation, such that $F_{jk}(x):=1-e^{-\lambda_{jk}x},$ for some $\lambda_{jk}\in(0,\infty)\,\,\forall j,k\in S,$ then $$\pi^{*}_{j}[x,\infty):=\frac{\sum_{k\in S}\frac{P_{jk}}{\lambda_{jk}}e^{-\lambda_{jk}x}}{\sum_{k\in S}\frac{P_{jk}}{\lambda_{jk}}}\s\forall x>0,$$
i.e, $\pi_{j}^{*}$ will be a mixture of Exponentials of aforementioned form.
\end{enumerate}
\end{remark}
An explicit expression for the limiting distribution of the joint process $(\Phi,I, Y)$ is given in following theorem.

\begin{Theorem}\label{T1}
Under Assumptions \ref{As0} and \ref{as2} the limiting distribution of the joint process $(\Phi^{(a)},I^{(a,b)},Y)$ can be expressed as :
\beqn
(\Phi^{(a)}_{t},I^{(a,b)}_{t},Y_{t}) \,\,\stackrel{d}{\to}\,\, \Big(0,\sum_{j\in S}\delta_U(\{j\})Z_j,U\Big)\quad \text{as } t\to\infty, \label{P1e01}
\eeqn
where $U\ci (Z_j)_{j\in S}$, $U\sim \pi$,   and 
\beqn
Z_j\stackrel{d}{=}
b(j)\int_0^{T^j}e^{-a(j)(T^{j}-s)}ds+e^{-a(j)T^{j}}V^{*}_j,\label{Z_{j}}\label{mixture}
\eeqn
where $T^{j}\sim \pi_{j}^{*}$ is independent of $V^{*}_j$, and 
$\mathcal{L}(V^{*}_j)$ is the unique solution to \eqref{x=ax+b} with $(A,B)$ having the distribution of 
\begin{eqnarray}
\Big(e^{-\int_{\tau^{j}_{0}}^{\tau_{1}^{j}}a(Y_{s})ds},
\int_{\tau^{j}_{0}}^{\tau_{1}^{j}}b^{}(Y_{s})e^{-\int_{s}^{\tau^{j}_{1}}a(Y_{r})dr}ds\Big).\label{AB}
\end{eqnarray}
\end{Theorem}

A remark on moments of $I$ is made (similar to Remark 3.3 in \cite{lindskog2018exact}) below.

\begin{remark}
Moments for the stationary distribution of $I$ in \eqref{P1e01} can be computed recursively using the representation for $(A_j,B_j)$ in \eqref{AB}. From $V^{*}_j\stackrel{d}{=}A_jV^{*}_j+B_j$ follows that, for $m\in\mathbb{N}$,  
$$
(V^{*}_j)^{m}(1-A_j^m)=\sum_{k=0}^{m-1}{m \choose k}A_j^kB_j^{m-k}(V_j^{*})^{k}.
$$
If there exists $n\in\mathbb{N}$ such that $EA_j^{m}<\infty$ and $E[A_j^kB_j^{m-k}]<\infty$ for $0\le k\le m\le n$, then $E|V^{*}_j|^{n}<\infty$ and independence between $V^{*}_j$ and $(A_j,B_j)$ gives following recursive relation of moments
$$
E[(V^{*}_j)^{m}]=\frac{1}{1-EA^{m}_{j}}\sum_{k=0}^{m-1}{m\choose k}E\big[A^{k}_{j}B^{m-k}_{j}\big]E[(V^{*}_{j})^{k}].
$$
From the above representations moments of the limit distribution 
$$
E\Big[\big(I^{(a,b)}_{\infty}\big)^m\Big]=\sum_{j\in S}\pi_jE\Big[\Big(b(j)\int_{0}^{T^{j}}e^{-a(j)(T^{j}-s)}ds+e^{-a(j)T^{j}}V^{*}_j\Big)^{m}\Big]
$$
can be computed using the independence of $T^{j}, V_{j}^{*}$. 
\end{remark}

\section{Divergence and critical domains}\label{sec_unstable}
This section is devoted to the exploration of the long-term behavior of the process $(\Phi^{(a)},I^{(a,b)})$ under conditions that are different from Assumption \ref{as2}. In particular, it will be assumed that the stability condition $E_{\pi}a(\cdot)> 0$ in Assumption \ref{as2}(a) does not hold. We will see that the conditions  $E_{\pi}a(\cdot)=0$ and $E_{\pi}a(\cdot)<0$ correspond to critical case and divergent  (or transient) case respectively which will be further subdivided into cases when $$\sigma_j^2:=\Var\Big(\int_{\tau_{0}^{j}}^{\tau_{1}^{j}}\big(a(Y_{s})-E_{\pi}a(\cdot)\big)ds\Big)$$ is $\in(0,\infty),$ or  $=\infty$, or $=0$ for each $j\in S$ as below
\begin{itemize}
\item \textbf{Case A:} $E_{\pi}a(\cdot)\le 0,$ and $\sigma_{j}^{2}\in (0,\infty)$ for all $j\in S,$
\item \textbf{Case B:} $E_{\pi}a(\cdot)\le 0,$ and $\int_{\tau_{0}^{j}}^{\tau_{1}^{j}}(a(Y_{s})-E_{\pi}a(\cdot))ds$ is regularly varying with index $\alpha\in(1,2)$, for all $j\in S,$
\item \textbf{Case C:} $E_{\pi}a(\cdot)\le 0,$ and $\sigma_{j}^{2}=0$ for all $j\in S.$
\end{itemize}
\subsection{\textbf{Case A:\s $E_{\pi}a(\cdot)\le 0,$ and $\sigma_{j}^{2}\in (0,\infty)$} }

Observe that $I^{(a,b)}_{t}=\Phi^{(a)}_{t}\widetilde{I}^{(a,b)}_{t}$ where $\widetilde{I}^{(a,b)}_{t}=\int_{0}^{t}b(Y_{s})e^{\int_{0}^{s}a(Y_{r})dr}ds.$ When $E_{\pi}a(\cdot)<0,$ then $\widetilde{I}^{(a,b)}_{t}$ behaves similar to $I^{(a,b)}_{t}$ as $t\to\infty$ distributional limit in Theorem \ref{T1}, implies that as $t\to\infty,$ $$\frac{\log |I^{(a,b)}_{t}|}{t}\stackrel{as}{\to}-E_{\pi}a(\cdot)$$ as a consequence of ergodic renewal theory.  To be precise we restate the assumptions under which a scaled fluctuation of $(\Phi^{(a)},I^{(a,b)})$ will obtain the following limit theorem. 
\bas\label{as3}
The $S$-valued process $Y$ and the functions $a,b:S\to\R$ satisfy 
\begin{enumerate}[(a)]
\item $a$ is integrable with respect to $\pi$, and $E_{\pi} a(\cdot)\le 0$.
\item For every $j\in S$, $\sigma_j^2:=\Var\Big(\int_{\tau_{0}^{j}}^{\tau_{1}^{j}}\big(a(Y_{s})-E_{\pi}a(\cdot)\big)ds\Big)\in (0,\infty)$. 
\item Assumption \ref{as2}(b) holds.
\end{enumerate}
\eas
Let $(B,M)$ be a Brownian motion and its supremum jointly defined in the same space, i.e 
\beqn
\big((B_{s},M_{s}): B\,\,\text{is a standard Brownian motion},\,\,M_{s}:=\sup_{0\le t \le s}B_{t},\,\, s\ge 0\big).\label{Brown}
\eeqn
 
\bt\label{T2} Suppose Assumptions \ref{As0} and \ref{as3} hold. 
Let $U\sim \pi,$ and $N\sim \normal(0,1)$  are all mutually independent.
\begin{enumerate}[(a)]
\item If  $E_{\pi}a(\cdot)<0,$ as $t\to\infty$
\beqn
\bigg(\frac{\log\Phi^{(a)}_{t} + tE_{\pi}a(\cdot)}{\sqrt{t}},\frac{\log|I^{(a,b)}_{t}| + tE_{\pi}a(\cdot)}{\sqrt{t}}\bigg)\stackrel{d}{\to}\sum_{j\in S}\delta_{U}(\{j\})\frac{\sigma_{j}}{\sqrt{E|\mathfrak{I}_{1}^{j}|}}(N,N).\label{eT2(a)}
\eeqn
\item Suppose $E_{\pi}a(\cdot)=0, b\neq 0$ and Assumption \ref{as3}(c) is replaced by  \beqn
E\Big[\Big(\log \big|\int_{\tau_{0}^{j}}^{\tau_{1}^{j}}b^{}(Y_{s})e^{-\int_{s}^{\tau_{1}^{j}}a(Y_{r})dr}ds\big|\Big)^{2}\Big]<\infty\s \text{for every}\,\, j\in S,\label{Theorem2bcond*}
\eeqn
\beqn
\bigg(\frac{\log\Phi^{(a)}_{t}}{\sqrt{t}},\frac{\log|I^{(a,b)}_{t}|}{\sqrt{t}}\bigg)\stackrel{d}{\to}\sum_{j\in S}\delta_{U}(\{j\})\frac{\sigma_{j}}{\sqrt{E|\mathfrak{I}_{1}^{j}|}}(F_{1},F_{2})\s\text{as}\,\, t\to\infty,\label{eT2b}
\eeqn
where $U$ is independent of random variables $(F_{1},F_{2})$ having joint density
\beqn
P[F_{1}\in dx, F_{2}\in d\xi]=\sqrt{\frac{2}{\pi}} \big(2\xi - x\big)e^{-\frac{(2\xi - x)^{2}}{2}}1_{[0,\infty)\times (-\infty, \xi)}(\xi, x)d\xi dx .\label{F1F2}
\eeqn
\end{enumerate}
\et

\begin{remark}$(F_{1},F_{2})$ is identical with $(B_{1},M_{1})$ in distribution (using notation from \eqref{Brown}). The difference between $``E_{\pi}a(\cdot)<0"$ and $``E_{\pi}a(\cdot)=0"$ is reflected in the distributional part of RHS  weak limit of $\bigg(\frac{\log\Phi^{(a)}_{t} + tE_{\pi}a(\cdot)}{\sqrt{t}},\frac{\log|I^{(a,b)}_{t}| + tE_{\pi}a(\cdot)}{\sqrt{t}}\bigg),$ as it weakly transits from $(B_{1},B_{1})\stackrel{d}{=}(N,N)$ to $(B_{1},M_{1})$.  Marginally $M_{1}\stackrel{d}{=}|B_{1}|\stackrel{d}{=} |N|\stackrel{d}{=}F_{2},$ (a consequence of  the reflection principle) which is also celebrated as a consequence of the Erdos-Kac theorem \cite{ErdosKac1946}.
\end{remark}

\subsection{\textbf{Case B}: ($E_{\pi}a(\cdot)\le 0,$ and a case of $\sigma_{j}^{2}=\infty$ for all $j\in S$)} By $\widehat{a}(\cdot),$ denote the function $a(\cdot)-E_{\pi}a(\cdot).$  We consider a special case of $\sigma_{j}^{2}=\infty,$ such that the condition $$\int_{\tau_{0}^{j}}^{\tau_{1}^{j}}\widehat{a}(Y_{s})ds$$ is regularly varying with index $\alpha\in(1,2)$ holds for all $j\in S.$ Under additional model assumptions,  a limit result in this context will be established. Recall the definition of $\widetilde{A}_{k}^{j}$ from display \eqref{Atilde}, which represents the set of tuples corresponding to successive transitions during the $k$-th recursion at state $j$ by the underlying process $J$. If we allow some elements of $\mathcal{G}$ to be regularly varying heavy-tailed, under the assumption of divergence (i.e $E_{\pi}a(\cdot)\le 0$),  it is intuitive that the long time limit results will be influenced by the behavior of the process during sojourn times, that have the heaviest tails. As noted in the  Remark \ref{Rinfinity}, even if all sojourn times have an Exponentially light tail, it is still possible for $\int_{\tau_{0}^{j}}^{\tau_{1}^{j}}\widehat{a}(Y_{s})ds$ to exhibit a regularly varying heavy tail, driven purely by the state-interaction dynamics of $J$. Here, we examine a scenario where the heaviest tail of the sojourn times dominates the effect of the $J$ chain (via controlling the moments of $|\widetilde{A}_{1}^{j}\cap \mathbb{S}_{\alpha}|$ in \eqref{bound3}), in explaining the regular variation of $\int_{\tau_{0}^{j}}^{\tau_{1}^{j}}\widehat{a}(Y_{s})ds$. Following assumption, particularly Assumption \ref{as4}(b), formalizes this notion which is why it is somewhat lengthy.

\bas\label{as4}The $S$-valued process $Y$ and the functions $a,b:S\to\R$ satisfy
\begin{enumerate}[(a)]  
\item $a$ is integrable with respect to $\pi$, and $E_{\pi} a(\cdot)\le 0$.
\item  Suppose the set $\mathcal{G}$ contains at least one element that is regularly varying at $\infty.$ Let the heaviest tail index of the sojourn times (which are elements of $\mathcal{G}$) be $\alpha \in(1,2),$ and let $L(\cdot)$ represent the largest slowly varying function among them (the function that diverges to $\infty$ the fastest as the argument tends to $\infty$). Define the set 
\beqn
\mathbb{S}_{\alpha}:= &\big\{(i,j)\in S^{2}: P_{ij}>0, \,\, F_{ij}\in \mathcal{G}, \,\,\widehat{a}(i)\neq 0, \,\,\text{and}\,\, \overline{F}_{ij}(x)=c_{ij}L(x)x^{-\alpha}\non\\ &\s \text{such that}\s \forall (i_{k},i_{k+1})\notin \mathbb{S}_{\alpha},\,\,\overline{F}_{i_{k}i_{k+1}}(x)=o(\overline{F}_{ij}(x)) \big\}.\non
\eeqn
This set $\mathbb{S}_{\alpha}$ consists of all pairs $(i,j)$ (representing transitions of the $J$-chain) whose corresponding sojourn time distributions have the same heaviest regularly varying tails (upto multiplicative constants), characterized by the smallest index $\alpha$ and the largest slowly varying function $L(\cdot)$) at $+\infty$. We assume $\mathbb{S}_{\alpha}$ is non-empty.

In line with \eqref{bound}, we assume the existence of $\epsilon>0,$ and a universal constant $\widetilde{c}_{\epsilon}<\infty$  such that for all $y\ge 0,$
\begin{align}
&\sup_{(i_{k},i_{k+1})\in \mathbb{S}_{\alpha}}\frac{P[y \widetilde{F}_{i_{k}i_{k+1}}>x]}{P[\widetilde{F}_{i_{k}i_{k+1}}>x]}\le \widetilde{c}_{\epsilon}(1\vee y)^{\alpha+\epsilon},\s\text{and}\label{bound2}\\ &E\Big[\sum_{(i_{l},i_{l+1})\in \widetilde{A}_{1}^{j}\cap \mathbb{S}_{\alpha}}c_{i_{l}i_{l+1}}(|a(i_{l})||\widetilde{A}_{1}^{j}\cap \mathbb{S}_{\alpha}|)^{\alpha+\epsilon}\Big]<\infty, \,\,\forall j\in S,\label{bound3}
\end{align}
 where $\widetilde{F}_{i_{k}i_{k+1}}$ is a random variable distributed as $F_{i_{k}i_{k+1}}.$

Furthermore, there exist constants $x^{*}, a^{*}<\infty,$ such that following hold: 
\beqn
\sup_{x\ge x^{*}}\sup_{(j_{k},j_{k+1})\notin \mathbb{S}_{\alpha}}\frac{\overline{F}_{j_{k}j_{k+1}}(x)}{x^{-\alpha}L(x)}\le 1,\,\,\label{bound4}\\ \sup_{(j_{k},j_{k+1})\notin \mathbb{S}_{\alpha}}a(j_{k})=a^{*},\,\,\,\&\,\,\,E|\widetilde{A}_{1}^{j}\cap \mathbb{S}^{c}_{\alpha}|^{1+\alpha}<\infty \non
\eeqn
where $\mathbb{S}^{c}_{\alpha}=\mathbb{S}^{c,+}_{\alpha}\cup \mathbb{S}^{c,-}_{\alpha},$ with  
\begin{align}
&\mathbb{S}^{c,+}_{\alpha}:=\{(i_{k},i_{k+1})\notin \mathbb{S}^{}_{\alpha}:P_{i_{k}i_{k+1}}>0, \widehat{a}(i_{k})>0\},\s\text{and}\nonumber\\ &\mathbb{S}^{c,-}_{\alpha}:=\{(i_{k},i_{k+1})\notin \mathbb{S}^{}_{\alpha}: P_{i_{k}i_{k+1}}>0, \widehat{a}(i_{k})<0\}.
\end{align}
\item[(c)] Assumption \ref{as2}(b) holds.
\item[(d)] Define $\mathbb{S}_{\alpha}^{+}:=\{(i,j)\in \mathbb{S}_{\alpha}:\widehat{a}(i)>0\},\s \mathbb{S}_{\alpha}^{-}:=\{(i,j)\in \mathbb{S}_{\alpha}:\widehat{a}(i)<0\}.$ For any $j\in S,$ define further 
\beqn
\widetilde{\alpha}_{j}^{(+)}&:=&E\Big[\sum_{(i_{l},i_{l+1})\in \widetilde{A}_{1}^{j}\cap \mathbb{S}^{+}_{\alpha}}c_{i_{l}i_{l+1}}\big[a(i_{l})-E_{\pi}a(\cdot)\big]^{\alpha}\Big],\label{alpha+}\\
\widetilde{\alpha}_{j}^{(-)}&:=& E\Big[\sum_{(i_{l},i_{l+1})\in \widetilde{A}_{1}^{j}\cap  \mathbb{S}^{-}_{\alpha}}c_{i_{l}i_{l+1}}|a(i_{l})-E_{\pi}a(\cdot)|^{\alpha}\Big],\label{alpha-}\\
\sigma_{(j,\alpha)}&:=& \bigg(\widetilde{\alpha}_{j}^{(+)}+\widetilde{\alpha}_{j}^{(-)}\bigg)^{1/\alpha},\s \beta_{j}:=\frac{\widetilde{\alpha}_{j}^{(-)} - \widetilde{\alpha}_{j}^{(+)}}{\widetilde{\alpha}_{j}^{(+)}+\widetilde{\alpha}_{j}^{(-)}}\in [-1,1]\label{sig}
\eeqn
given that the RHS of \eqref{alpha+},\eqref{alpha-} and \eqref{sig} are all finite, specified by the marginal dynamics of $J$ chain and $a(\cdot).$
\end{enumerate}
\eas
\begin{remark}\label{verifycond}
\begin{enumerate}[(a)]
\item Under Assumptions \ref{As0} and \ref{as4} the quantities $\widetilde{\alpha}_{j}^{(+)},\widetilde{\alpha}_{j}^{(-)}$ can be simplified in closed form, given the information on the marginal Markovian dynamics $J$. For all $j\in S,$
\[
\widetilde{\alpha}_{j}^{(+)}=E|A_{k}^{j}| \sum_{(i_{l},i_{l+1})\in \mathbb{S}^{+}_{\alpha}}\mu_{i_{l}}P_{i_{l}i_{l+1}}c_{i_{l}i_{l+1}}[\widehat{a}(\cdot)(i_{l})]^{\alpha},\]
\[\widetilde{\alpha}_{j}^{(-)}=E|A_{k}^{j}| \sum_{(i_{l},i_{l+1})\in \mathbb{S}^{-}_{\alpha}}\mu_{i_{l}}P_{i_{l}i_{l+1}}c_{i_{l}i_{l+1}}|\widehat{a}(\cdot)(i_{l})|^{\alpha},\]
where $\mu$ is the stationary distribution of the marginal $S$-valued Markovian dynamics $J.$ 
\item In line of \eqref{bound}, \eqref{bound2} holds trivially by setting $\widetilde{c}_{\epsilon}=\sup_{(i_{k},i_{k+1})\in\mathbb{S}_{\alpha}}c_{i_{k}i_{k+1}}<\infty.$ 
\item Display \eqref{bound3} indicates that the tail of the heaviest sojourn time distribution will dominate the effect of the interaction dynamics of the $J$ chain, and it holds if we have $\sup_{x\in S}|a(x)| <\infty,$ and $E[|A_{1}^{j}|^{1+\alpha+\epsilon}]<\infty.$
\item[(d)] Since for any $(j_{k},j_{k+1})\notin \mathbb{S}_{\alpha},$ $\frac{\overline{F}_{j_{k}j_{k+1}}(x)}{x^{-\alpha}L(x)}\to 0$ as $x\to\infty,$ so does $\sup_{(j_{k},j_{k+1})\notin \mathbb{S}_{\alpha}}\frac{\overline{F}_{j_{k}j_{k+1}}(x)}{x^{-\alpha}L(x)}$ and \eqref{bound4} holds trivially.
\end{enumerate}
\end{remark}

\bt\label{T3} Suppose Assumptions \ref{As0} and \ref{as4} hold for some $\alpha\in (1,2)$ and slowly varying function $L(\cdot)$. 
Let $U\sim \pi,$ and $\alpha$-Stable processes $(\mathcal{\bf{S}}_{\alpha,\beta_{j}}:j\in S)$  are all mutually independent. 
\begin{enumerate}[(a)]
\item If  $E_{\pi}a(\cdot)<0,$  there exists a slowly varying function $L_{0}(\cdot)$ such that as $t\to\infty,$
\beqn
&&\bigg(\frac{\log\Phi^{(a)}_{t} + tE_{\pi}a(\cdot)}{t^{\frac{1}{\alpha}}L_{0}(t)},\frac{\log|I^{(a,b)}_{t}| + tE_{\pi}a(\cdot)}{t^{\frac{1}{\alpha}}L_{0}(t)}\bigg)\non\\ &&\s\stackrel{d}{\to}\,\,\sum_{j\in S}\delta_{U}(\{j\})\bigg(\frac{1}{E|\mathfrak{I}^{j}_{1}|}\bigg)^{\frac{1}{\alpha}}\sigma_{(j,\alpha)}\Big(\mathcal{S}_{\alpha}(1,\beta_{j},0),\mathcal{S}_{\alpha}(1,\beta_{j},0)\Big).\label{eT2a}
\eeqn
\item Suppose $E_{\pi}a(\cdot)=0, b\neq 0$ and Assumption \ref{as4}(c) is replaced by the following condition that there exists $\eps>0,$ such that following holds for all $j\in S,$ \beqn
E\Big[\Big(\log \big|\int_{\tau_{0}^{j}}^{\tau_{1}^{j}}b^{}(Y_{s})e^{-\int_{s}^{\tau_{1}^{j}}a(Y_{r})dr}ds\big|\Big)^{\alpha+\eps}\Big]<\infty.\label{Theorem2bcond**}
\eeqn
Then there exists a slowly varying function $L_{0}(\cdot)$ such that as $t\to\infty,$
\beqn
\bigg(\frac{\log\Phi^{(a)}_{t}}{t^{\frac{1}{\alpha}}L_{0}(t)},\frac{\log|I^{(a,b)}_{t}|}{t^{\frac{1}{\alpha}}L_{0}(t)}\bigg)\stackrel{d}{\to}\sum_{j\in S}\delta_{U}(\{j\})\bigg(\frac{1}{E|\mathfrak{I}^{j}_{1}|}\bigg)^{\frac{1}{\alpha}}\sigma_{(j,\alpha)}\big(\mathcal{\bf{S}}_{\alpha,\beta_{j}}(1),\sup_{0\le u \le 1}\mathcal{\bf{S}}_{\alpha,\beta_{j}}(u)\big).\non
\eeqn
\end{enumerate}
\et

\begin{remark}\label{Rinfinity}A major component of Theorem \ref{T3} is Lemma \ref{regV}, which states that under Assumption \ref{As0}, Assumption \ref{as4}(b) implies that the quantity $\int_{\tau_{0}^{j}}^{\tau_{1}^{j}}\widehat{a}(Y_{s})ds$ is regularly varying with index $\alpha \in (1, 2)$. We conjecture that in a finite state space ($|S| < \infty$), the reverse also holds, i.e under Assumption \ref{As0}, if the quantity $\int_{\tau_{0}^{j}}^{\tau_{1}^{j}}\widehat{a}(Y_{s})ds$ is regularly varying with index $\alpha \in (1, 2)$, then at least one element of $\mathcal{G}$ must be regularly varying at infinity with tail index $\alpha \in (1, 2)$. However, in the case of an infinite state space ($|S| = \infty$), this implication may not hold. For instance, it is possible to construct a continuous-time Markov chain (i.e, all sojourn times are Exponentially distributed) where $\int_{\tau_{0}^{j}}^{\tau_{1}^{j}}\widehat{a}(Y_{s})ds$ exhibits heavy-tailed behavior with index $\alpha \in (1, 2)$ purely due to the dynamics of $J$ in the infinite state space $S$. In such pathological cases, results analogous to Theorem \ref{T3}(a) and (b) will still hold, even in the absence of Assumption \ref{as4}(b) (or a modification), a scenario we plan to explore further in future.
\end{remark}

\subsection{\textbf{Case C}: ($E_{\pi}a(\cdot)\le 0,$ and  $\sigma_{j}^{2}=0,$ for all $j\in S$)}
We begin by stating the following proposition.
\begin{Proposition}\label{aconstant} If $S$ is finite, then the Assumption $\sigma^{2}_{j}=0$ implies that $a(\cdot)= $ a constant with probability $1.$ 
\end{Proposition}

We assume that the $a(\cdot)$ remains constant throughout \textbf{Case C} and let that constant be $a(\cdot)=a$ and as a consequence $\Phi^{(a)}(t)=e^{at}$. The following results show how under different conditions $I^{(a,b)}_{t}$ grow differently, from exponential with a random multiplicative coefficient to an asymptotically linear function of $t$ with different exponents.
\bt\label{T4} Suppose Assumptions \ref{As0} holds with $a(\cdot)=a$. Let $U\sim \pi,$ and $N\sim N(0,1)$ are independent of each other.  As $t\to\infty,$ for any $s>0$ 
\begin{enumerate}[(a)]
\item if  $a<0,$ and suppose $E\log^{+}|\int_{\tau_{0}^{j}}^{\tau_{1}^{j}}b(Y_{s})e^{a(s-\tau_{0}^{j})}ds|<\infty$ holds. Then as $t\to\infty,$
\beqn
e^{at}I^{(a,b)}_{t}\stackrel{d}{\to}\sum_{j\in S}^{}\delta_{U}(\{j\})\Big[\int_{0}^{\tau_{0}^{j}}b(Y_{s})e^{as}ds+ e^{a\tau_{0}^{j}} V_{j}^{*}\Big]
\eeqn
where $V_{j}^{*}$ satisfies \eqref{x=ax+b} with $(A,B)$ respectively are
$$(A,B)\stackrel{d}{=}\Big(e^{a(\tau_{1}^{j}-\tau_{0}^{j})},\int_{\tau_{0}^{j}}^{\tau_{1}^{j}}b(Y_{s})e^{a(s-\tau_{0}^{j})}ds\Big).$$

\item Suppose $a=0, b\neq 0$ and $b(\cdot)$ is integrable with respect to $\pi.$ As $t\to\infty,$
\beqn
\frac{I^{(0,b)}_{t}}{t}\stackrel{p}{\to}E_{\pi}b(\cdot). \non
\eeqn
This result can be further tri-furcated based on whether $$\sigma_{j,(b)}^{2}:=\Var\Big(\int_{\tau_{0}^{j}}^{\tau_{1}^{j}}\big(b(Y_{s}) - E_{\pi}b(\cdot)\big)ds\Big),$$ is finite, $=\infty,$ or $=0;$ and they are following:
\begin{enumerate}[(i)]
\item If $\sigma_{j,(b)}^{2}<\infty,$ as $t\to\infty,$ then $\frac{I^{(0,b)}_{t}}{\sqrt{t}}-\sqrt{t}E_{\pi}b(\cdot)\stackrel{d}{\to} \sum_{j\in S}^{}\delta_{U}(\{j\})\frac{\sigma_{j,(b)}}{\sqrt{E|\mathfrak{I}_{1}^{j}|}}N $.
\item Suppose Assumption \ref{as4}(a)(b)\& (d) hold replacing $a$ by $b$ and $\alpha$ by $\alpha_{b}$ (a special case of $\sigma_{j,(b)}^{2}=\infty$). The quartet $(\widetilde{\alpha}_{j,(b)}^{(+)},\widetilde{\alpha}_{j,(b)}^{(-)},\widetilde{\beta}_{j,(b)}, \sigma_{(j,\alpha_{b},(b))})$ is defined similarly as $(\widetilde{\alpha}_{j}^{(+)},\widetilde{\alpha}_{j}^{(-)},\widetilde{\beta}_{j},\sigma_{(j,\alpha)})$ in  \eqref{alpha+},\eqref{alpha-}, \eqref{sig} replacing $a$ by $b.$ Then as $t\to\infty$, $$\frac{I_{t}^{(0,b)}-tE_{\pi}b(\cdot)}{t^{\frac{1}{\alpha_{b}}}L_{1}(t)}\stackrel{d}{\to} \sum_{j\in S}^{}\delta_{U}(\{j\})\Big(\frac{1}{E|\mathfrak{I}_{1}^{j}|}\Big)^{\frac{1}{\alpha_{b}}}\sigma_{(j,\alpha_{b},(b))}\mathcal{S}_{\alpha_{b}}(1,-\widetilde{\beta}_{j,(b)},0)$$
for some slowly varying function $L_{1}(\cdot).$
\item If $b$ is a constant (case corresponding to $\sigma_{j,(b)}^{2}=0$), and then \(I_{t}^{(0,b)}=bt\) is a deterministic quantity.
\end{enumerate}
\end{enumerate}
\et

Following result indicates that the asymptotic behavior of $\log|I^{(a,b)}_{t}|$ from Theorems \ref{T2} and \ref{T3} also applies to the logarithm of any process that can be expressed as an arbitrary linear combination of $\{\Phi_{t}^{(a)}, I_{t}^{(a,b)}\}$ or $\{\Phi_{t}^{(a)}, |I_{t}^{(ma,b)}|^{\frac{1}{m}}\}$. However, the result for Theorem \ref{T4} will differ from those in Theorems \ref{T2} and \ref{T3}, which will be relevant for various applications. For brevity, we denote the weak limits of $\log|I^{(a,b)}_{t}|$ as $t \to \infty$ (after appropriate scaling and shifts) by $\mathcal{W}_{2(a)}$, $\mathcal{W}_{2(b)}$, $\mathcal{W}_{3(a)}$, and $\mathcal{W}_{3(b)}$, corresponding to Theorems \ref{T2}(a), \ref{T2}(b), \ref{T3}(a), and \ref{T3}(b), respectively, under the corresponding assumptions.

\begin{Corollary}\label{conn1}
Suppose two arbitrary stochastic processes $(\widetilde{Z}^{(1)}_{t}:t\ge 0), (\widetilde{Z}^{(2)}_{t}:t\ge 0)$ are marginally defined as a random linear combination of both $\Phi_{t}^{(a)}$ and $I_{t}^{(a,b)}$ weakly (for any given $a,b$ that are $S$ valued functions); i.e, for any $t>0,$ marginally 
\beqn
\widetilde{Z}^{(1)}_{t}\stackrel{d}{:=}a_{1}\Phi_{t}^{(a)}+b_{1}I_{t}^{(a,b)},\s \widetilde{Z}^{(2)}_{t}\stackrel{d}{:=}a_{2}\Phi_{t}^{(a)}+b_{2}|I_{t}^{(ma,b)}|^{\frac{1}{m}},\label{ecorr}
\eeqn
 for some $m\ge 0$, and any arbitrary real valued random variables $(a_{1},b_{1}), (a_{2},b_{2})$ that are independent with $(\Phi_{t}^{(a)},I_{t}^{(a,b)})$ and $(\Phi_{t}^{(a)},I_{t}^{(ma,b)})$ respectively. 
\begin{enumerate}[(a)]
\item
\begin{enumerate}[(1)]
\item Time asymptotic behaviors of both $\log|\widetilde{Z}^{(1)}_{t}|$ and $\log|I^{(a,b)}_{t}|$ will be identical in distribution under the same scaling and shifts, in all possible cases of Theorem \ref{T2}, Theorem \ref{T3} irrespective of any $(a_{1},b_{1})$. 
\item For $\log|\widetilde{Z}^{(2)}_{t}|$ we will use identical scaling and shifts as used for $\log|\widetilde{Z}^{(1)}_{t}|$.
\begin{enumerate}[(i)]
\item Under assumptions of Theorem \ref{T2}(a) with $(a,b)$ replaced by $(ma,b)$ in Assumption \ref{as3}(c),  as $t\to\infty,$ the quantity $$\frac{\log|\widetilde{Z}^{(2)}_{t}|}{\sqrt{t}}+\sqrt{t}E_{\pi}a(\cdot)\stackrel{d}{\to}\mathcal{W}_{2(a)},$$ irrespective of any $(a_{2},b_{2})$.  \\ Under assumptions of Theorem \ref{T2}(b) with $(a,b)$ replaced by $(ma,b)$ in \eqref{Theorem2bcond*} as $t\to\infty,$ the quantity $$\frac{\log|\widetilde{Z}^{(2)}_{t}|}{\sqrt{t}}\stackrel{d}{\to}\mathcal{W}_{2(b)}$$ irrespective of any $(a_{2},b_{2})$.
\item Under assumptions of Theorem \ref{T3}(a) with $(a,b)$ replaced by $(ma,b)$ in Assumption \ref{as4}(c),  as $t\to\infty,$ the quantity $$\frac{\log|\widetilde{Z}^{(2)}_{t}|+tE_{\pi}a(\cdot)}{t^{1/\alpha}L_{0}(t)}\stackrel{d}{\to}\mathcal{W}_{3(a)},$$ irrespective of any $(a_{2},b_{2})$.\\ Under assumptions of Theorem \ref{T3}(b) with $(a,b)$ replaced by $(ma,b )$ in \eqref{Theorem2bcond**} as $t\to\infty,$ the quantity $$\frac{\log|\widetilde{Z}^{(2)}_{t}|}{t^{1/\alpha}L_{0}(t)}\stackrel{d}{\to}\mathcal{W}_{3(b)},$$ irrespective of any $(a_{2},b_{2})$.
\end{enumerate}
\end{enumerate}

\item Suppose assumptions of Theorem \ref{T4} holds.
\begin{enumerate}[(i)]
\item If $a<0,$  and $E\log^{+}|\int_{\tau_{0}^{j}}^{\tau_{1}^{j}}b(Y_{s})e^{a(s-\tau_{0}^{j})}ds|<\infty,$ as $t\to\infty,$
\beqn
e^{at}\widetilde{Z}^{(1)}_{t}&\stackrel{d}{\to}& a_{1}+b_{1}\Big(\sum_{j\in S}^{}\delta_{U}(\{j\})\big[\int_{0}^{\tau_{0}^{j}}b(Y_{s})e^{as}ds+ e^{a\tau_{0}^{j}} V_{j}^{*}\big]\Big);\non
\eeqn
where $V_{j}^{*}$ is same as in Theorem \ref{T4}(a).

Suppose $a<0,\&\,\,\, E\log^{+}|\int_{\tau_{0}^{j}}^{\tau_{1}^{j}}b(Y_{s})e^{ma(s-\tau_{0}^{j})}ds|<\infty$ holds then as $t\to\infty,$
\beqn
e^{at}\widetilde{Z}^{(2)}_{t} &\stackrel{d}{\to}&a_{2}+b_{2}\Big(\sum_{j\in S}^{}\delta_{U}(\{j\})\big|\int_{0}^{\tau_{0}^{j}}b(Y_{s})e^{mas}ds+ e^{ma\tau_{0}^{j}} V_{j,m}^{*}\big|^{\frac{1}{m}}\Big)\non
\eeqn  where $V_{j,m}^{*}$ is defined similar to $V_{j}^{*}$ in \eqref{x=ax+b} with $(A,B)$ respectively defined as 
$$(A,B)\stackrel{d}{=}\Big(e^{ma(\tau_{1}^{j}-\tau_{0}^{j})},\int_{\tau_{0}^{j}}^{\tau_{1}^{j}}b(Y_{s})e^{ma(s-\tau_{0}^{j})}ds\Big).$$
\item If $a=0$ and conditions of Theorem \ref{T4}(b) hold, then $\frac{\widetilde{Z}^{(1)}_{t}}{t}\stackrel{d}{\to}  [E_{\pi}b(\cdot)] b_{1}$\,\,\, \&\,\,\, $\frac{\widetilde{Z}^{(2)}_{t}}{t^{\frac{1}{m}}}\stackrel{d}{\to}  \big|E_{\pi}b(\cdot)\big|^{\frac{1}{m}} b_{2}$ as $t\to\infty$.
\end{enumerate}
\end{enumerate}
\end{Corollary}

\begin{remark}
\begin{enumerate}[(i)]
\item Results of Corollary \ref{conn1} can be further extended in finding long-time limit results of any non-linear functions, such as polynomials of three quantities $\Phi_{t}^{(a)}, |I^{(a,b)}_{t}|,|I^{(ma,b)}_{t}|^{1/m}$ of any order as $t\to\infty$ in the divergent cases (i.e under $E_{\pi}a(\cdot)\le 0$).
\item The random coefficients $(a_{1},b_{1},a_{2},b_{2})$ in \eqref{ecorr}) can be time-dependent (or stochastic) as well.  All results of Corollary \ref{conn1}(a) will still identically hold if $(a_{1},b_{1},a_{2},b_{2})$ satisfies either $$\frac{\log\max\{|a_{1}|+|b_{1}|,|a_{2}|+|b_{2}|\}}{\sqrt{t}}\stackrel{P}{\to}0,\s \text{as}\s t\to\infty$$ in Corollary \ref{conn1}(a) part (i) under conditions of Theorem \ref{T2}; or  $$\frac{\log\max\{|a_{1}|+|b_{1}|,|a_{2}|+|b_{2}|\}}{t^{1/\alpha}L_{0}(t)}\stackrel{P}{\to}0,\s \text{as}\s t\to\infty$$  in Corollary \ref{conn1}(a) part (ii) under conditions of Theorem \ref{T3}.  Corollary \ref{conn1}(b) also holds identically if $(a_{1},b_{1},a_{2},b_{2})$ represents the $t\to\infty$ weak limits of their time-dependent/stochastic versions of the respective quantities (given that their weak limits exist uniquely).
\end{enumerate}
\end{remark}

\section{Applications}\label{secApplications}

\subsection{\textbf{Pitchfork Bifurcation}} In certain deterministic dynamical systems, a notable phenomenon called the `phase transition' takes place in the steady state behavior, depending on a critical parameter as its value shifts from one regime to another. While shifting, the previously stable steady state of the system becomes unstable and gives way to the emergence of more locally stable steady states. In computational terms, this transition corresponds to what is known as Hopf bifurcation. It occurs when one of the ``linearized" eigenvalues of the Jacobian matrix becomes purely imaginary. Importantly, the real part of this eigenvalue exhibits a disparity on either side of the bifurcation point, marking a critical change in the system's behavior.

We begin with the most common example in $2$-dimension which is known as the Hopf-Andronov's model (Pg. $57$, \cite{Yuri2013bifurcation}). A general form of it can be expressed in the polar coordinate $(\rho_{t},\phi_{t})$ form as
\beqn
\frac{d\rho_{t}}{dt}=\rho_{t}(a- b\rho_{t}^{2}),\s \frac{d\phi_{t}}{dt}=1.\label{model1}
\eeqn
Solving above system for $b=1$ (for simplicity) yields a closed-form solution for $\rho:=\big(\rho_{t}: t>0\big).$   
$$\rho^{2}_{t}=\bigg[\frac{1}{a}+\bigg[\frac{1}{\rho^{2}_{0}}-\frac{1}{a}\bigg]e^{-2a t}\bigg]^{-1}1_{\{a\neq 0\}}+\Big(t+\frac{1}{\rho^{2}_{0}}\Big)^{-1}1_{\{a = 0\}},\s \phi_{t}=\phi_{0}+t.$$
This leads to conclusions that 
\begin{itemize}
\item if $a\le 0,$ then $0$ is a globally stable steady state.
\item  If $a>0,$ then $\rho^{2}=a$ is a stable steady state (globally stable respectively in $(0,\infty)$) and $0$ is an unstable steady state. For $a>0,$ for initial value $\rho^{2}_{0},$ any arbitrary path of $\rho^{2}_{t}$ will exponentially converges to $a.$
\end{itemize}
 This is an example of Pitchfork bifurcation \cite{Yuri2013bifurcation} where $a=0$ is the bifurcation parameter value. For this problem, we give an exact structure of limit results when the bifurcation parameters are time-dependent and modulated by an $S$-valued semi-Markov process $Y$.

Consider the model described in \eqref{model2}. Unlike the deterministic part in \eqref{model1}, we see that sign of $E_{\pi}a(\cdot)$ will determine the long time behaviour of the process $\rho:=\big(\rho_{t}: t>0\big)$ in \eqref{model2}. In \cite{hurth2020random} authors found a closed form of the stationary distribution of \eqref{model2} under Markov-switching with $|S|=2.$ Two states made the computation a lot easier, however, for $|S|>2$  their technique does not yield a closed-form solution. Below for more general $S$-valued semi-Markov process $Y$, we account for explicit long-time results for $\rho^{2}$ in different cases determined by the sign of $E_{\pi}a(\cdot),$ and tail properties of the sojourn times of $Y$. These results will help construct the confidence intervals for $\rho^{2}$ that can be time-dependent, and asymptotically exact (in time).

\bt\label{Ex4}
Consider the model in \eqref{model2}. Suppose Assumptions \ref{As0} holds for the underlying semi-Markovian environment $Y$ with the limiting distribution $\pi$ and $U$ is a random variable such that $U\sim \pi.$ 
\begin{enumerate}[(A)]
\item  Suppose Assumption \ref{as2} holds with $(a,b)$ replaced by $(2a,b)$. Weak limit of $(\rho^{2}_{t},t\ge 0)$ will satisfy
\beqn
\rho^{2}_{t} \,\,\stackrel{d}{\to}\,\, \rho^{2}_{\infty}:= \sum_{j\in S}\delta_{U}(\{j\})\frac{1}{2 W^{(2a,b)}_{j}}\quad \text{as }\s  t\to\infty. \label{W}
\eeqn
where $U\ci (W^{(2a,b)}_{j})_{j\in S}$, and $W^{(2a,b)}_{j}$ is defined in \eqref{mixture} identically as $Z_{j}$ where $(a,b)$ is replaced by $(2a,b).$ 
\item Suppose Assumptions \ref{As0} \& \ref{as3} hold with with $(a,b)$ replaced by $(2a,b)$ and $$\widetilde{\sigma}_j^2:=4\Var\Big(\int_{\tau_{0}^{j}}^{\tau_{1}^{j}}\big(a(Y_{s})-E_{\pi}a(\cdot)\big)ds\Big)<\infty.$$
\begin{enumerate}[(i)] 
\item If $E_{\pi}a(\cdot)<0,$ then as $t\to\infty$
\beqn
\frac{-2\log|\rho_{t}| + 2tE_{\pi}a(\cdot)}{\sqrt{t}}\stackrel{d}{\to}\sum_{j\in S}\delta_{U}(\{j\})\frac{\widetilde{\sigma}_{j}}{\sqrt{E|\mathfrak{I}_{1}^{j}|}}N.
\eeqn
\item Suppose $E_{\pi}a(\cdot)=0, b\neq 0$ and Assumption \ref{as3}(c) is replaced by 
\beqn
E\Big[\Big(\log |\int_{\tau_{0}^{j}}^{\tau_{1}^{j}}b^{}(Y_{s})e^{-\int_{s}^{\tau_{1}^{j}}2a(Y_{r})dr}ds|\Big)^{2}\Big]<\infty,\label{Theorem2bcond1*}
\eeqn
then as $t\to\infty,$
\beqn
\frac{-2\log|\rho_{t}|}{\sqrt{t}}\stackrel{d}{\to}\sum_{j\in S}\delta_{U}\big(\{j\}\big)\frac{\widetilde{\sigma}_{j}}{\sqrt{E|\mathfrak{I}_{1}^{j}|}}|N|.
\eeqn
\end{enumerate} 
\item Suppose Assumption \ref{As0} \& \ref{as4} hold for some $\alpha\in (1,2)$ and slowly varying function $L(\cdot),$  and $(\widetilde{\alpha}_{j}^{(+)},\widetilde{\alpha}_{j}^{(-)},\beta_{j},\sigma_{(j,\alpha)})$ are re-defined with $(a,b)$ replaced by $(2a,b)$. Let $U\sim \pi,$ and Stable-$\alpha_{}$ process $\mathcal{\bf{S}}_{\alpha,\beta_{j}}$  be all mutually independent for each $j\in S.$ 
\begin{enumerate}[(i)]
\item If $E_{\pi}a(\cdot)<0,$ then as $t\to\infty,$ \beqn
\frac{-2\log|\rho_{t}| + 2tE_{\pi}a(\cdot)}{t^{\frac{1}{\alpha}}L_{0}(t)} \s\stackrel{d}{\to}\,\,\sum_{j\in S}\delta_{U}\big(\{j\}\big)\Big(\frac{1}{E|\mathfrak{I}_{1}^{j}|}\Big)^{\frac{1}{\alpha}}\sigma_{(j,\alpha)}\mathcal{S}_{\alpha}(1,\beta_{j},0).\label{eT2aEx1}
\eeqn
\item If $E_{\pi}a(\cdot)=0, b\neq 0$ and Assumption \ref{as4}(c) is replaced by  \beqn
E\Big[\Big(\log |\int_{\tau_{0}^{j}}^{\tau_{1}^{j}}b^{}(Y_{s})e^{-\int_{s}^{\tau_{1}^{j}}2a(Y_{r})dr}ds|\Big)^{\alpha+\epsilon}\Big]<\infty\s\text{for some }\,\, \epsilon>0.\label{Theorem2bcond*Ex1}
\eeqn
Then for a slowly varying function $L_{0}(\cdot)$ as $t\to\infty,$
\beqn
\frac{-2\log|\rho_{t}|}{t^{\frac{1}{\alpha}}L_{0}(t)}\stackrel{d}{\to}\sum_{j\in S}\delta_{U}(\{j\}) \Big(\frac{1}{E|\mathfrak{I}_{1}^{j}|}\Big)^{\frac{1}{\alpha}}\sigma_{(j,\alpha)}\sup_{0\le u \le 1}\mathcal{\bf{S}}_{\alpha,\beta_{j}}(u).\non
\eeqn
\end{enumerate} 
\item Suppose Assumption \ref{As0} holds and $a(\cdot)=$ a constant denoted by $a,$ and $U\sim\pi.$ As $t\to\infty,$ if
\begin{enumerate}[(i)]
\item $a<0,$ and $E\log^{+}|\int_{\tau_{0}^{j}}^{\tau_{1}^{j}}b(Y_{s})e^{2a(s-\tau_{0}^{j})}ds|<\infty,$  then $$e^{2at}\rho^{-2}_{t}\stackrel{d}{\to}\rho^{-2}_{0}+2\bigg(\sum_{j\in S}\delta_{U}(\{j\})\big[\int_{0}^{\tau_{0}^{j}}b^{}(Y_{s})e^{2as}ds +e^{2a\tau_{0}^{j}}V_{j}^{*}\big]\bigg),$$ \newline where $V_{j}^{*}$ satisfies \eqref{x=ax+b} with $$(A,B)\stackrel{d}{=}\Big(e^{2a(\tau_{1}^{j} - \tau_{0}^{j})},\int_{\tau_{0}^{j}}^{\tau_{1}^{j}}b^{}(Y_{s})e^{2a(s-\tau_{0}^{j})}ds\Big).$$
\item $a=0, \,\,\&\s E_{\pi}|b^{}|<\infty$. Then as $t\to\infty,$ $$\s\s\frac{\rho^{-2}_{t}}{t}\stackrel{d}{\to}E_{\pi}b(\cdot).$$
\end{enumerate}
\end{enumerate}
\et
\begin{remark}\label{rmk}
The Goldie-Kesten theorem (Theorem 2.4.4 in Buraczewski et al. \cite{buraczewski2016stochastic}) characterizes the heavy-tailed behavior of the solution to the fixed-point equation \eqref{x=ax+b}. If $A\ge 0$ a.s. and
$\mathcal{L}(\log A | A \ge 0)$ is non-arithmetic, $P[Ax + B = x] < 1$ for all $x \in \R$, and there exists $\alpha_{1} > 0$
such that 
$$EA^{\alpha_{1}} = 1, \s E|B|^{\alpha_{1}}<\infty,\s EA^{\alpha_{1}}\log^{+}A<\infty,$$
then there exist constants $c_{+}, c_{-} \ge 0$ with $c_{+} + c_{-} > 0$ such that  as $x\to\infty$
$$P[X > x] \sim c_{+}x^{-\alpha_{1}},\s P[X<-x] \sim c_{-}x^{-\alpha_{1}}.$$ 
Under $E_{\pi}a(\cdot)>0$ (the stable regime), irrespective of $\rho_{0},$ the trajectory $(\rho^{2}_{t}:t\ge 0)$ will converge to the weak limit \eqref{W}. Under this case, it can be shown that $P[\sup_{t\ge 0}\rho_{t}^{-2}<\infty]=1$ and hence $P[\inf|\rho_{t}|>0]=1.$ 
The aversion from $0$ in the weak limit of $\rho_{t}^{2}$ can be shown quantitatively, as a consequence of the Goldie-Kesten theorem result. Under Assumption \ref{As0} \& Assumption \ref{as2}, if $\inf_{j\in S}a(j)<0$ holds, then there exists $\nu^{*}>0,$ such that $$\nu^{*}:=\inf_{j\in S}\alpha_{j}>0,\,\,\text{ where }\,\,\alpha_{j}:=\sup\{\alpha>0: Ee^{-\alpha\int_{\tau_{0}^{j}}^{\tau_{1}^{j}} 2a(Y_{s})ds}=1\}.$$ Under regularity conditions of Goldie-Kesten theorem following holds, 
\beqn
\lim_{x\to\infty}x^{\nu^{*}}P\Big[\rho^{2}_{\infty}<\frac{1}{x}\Big]&=&\lim_{x\to\infty}\sum_{j\in S}\pi_{j}\Big[x^{\nu^{*}}P\big[W^{(2a,b)}_{j}>\frac{x}{2}\big]\Big]\to c,
\eeqn
for some explicit constant $c.$ This implies that if there exists at least one regime $x_{*} \in S$ where $a(x_{*}) < 0$ then  $P\big[\rho^{2}_{\infty} \le \epsilon\big]$ will scale as $c\epsilon^{\nu^{*}}$ as $\epsilon \to 0^{+}$, which is quite surprising.
\end{remark}

In epidemic model contexts, results on the deterministic SIS model under Markov switching are derived in Proposition 5.2 of  \cite{lindskog2018exact}. When the infection and the susceptibility rates are modulated by a semi-markovian $Y,$ similar results can be established in line with Theorem \ref{Ex4}. Also in a jump-type process context under Markovian regime-switching,  multivariate version of the fixed-point equation \eqref{x=ax+b} appears that is explored in \cite{cappelletti2021dynamics} with applications in stochastic mono-molecular biochemical reaction networks. Similar semi-Markovian extension can be established there as well.

\subsection{\textbf{Applications to regime-switching generalized O.U process}}

This framework can be used to establish long-term behaviors of a class of diffusions beyond just the Ornstein-Uhlenback type. Let $a(\cdot), b(\cdot): S\to \R$ be two arbitrary $S$-valued  functions. Consider the following descriptions of two semi-Markov modulated stochastic processes $X_{1}:=(X_{1,t}^{}:t\ge 0),\,\, X_{2}:=(X_{2,t}^{}:t\ge 0)$ satisfying stochastic differential equations (SDE) that do not immediately look like the Ornstein-Uhlenbeck processes:
\beqn
dX_{1,t}&=&(X^{2}_{1,t}+1)\Big[-a(Y_{t})\tan^{-1}(X_{1,t})+b^{2}(Y_{t})X_{1,t}\Big]dt+b(Y_{t}) (X^{2}_{1,t}+1)dW_{t},\non\\
dX_{2,t}&=&-[a(Y_{t})+\frac{b^{2}(Y_{t})}{2}e^{-2X_{2,t}}]dt+b(Y_{t})e^{-X_{2,t}}dW_{t}.\non
\eeqn

We generalize above examples in the following manner. Given any continuous and differentiable function $\beta(\cdot):\R\to\R_{\ge 0},$ that does not vanish in $\R,$ the function $h$ is defined as $$h(x):=\int_{c}^{x}\frac{1}{\beta(y)}dy$$ for some constant $c.$ The general structure $X:=(X_{t}:t\ge 0)$ with initial condition $X_{0}=x_{0}$ can be illustrated as the solution of the following SDE 
\beqn
dX_{t}=\Big[-a(Y_{t})h(X_{t})\beta(X_{t})+\frac{b^{2}(Y_{t})}{2}\beta(X_{t})\beta'(X_{t})\Big]dt+b(Y_{t})\beta(X_{t})dW_{t};\label{exe3}
\eeqn
that is modulated by an $S$-valued semi-Markovian environment $Y.$ Clearly $h$ is a monotone-increasing function, so the inverse exists. One can see that $X_{1}, X_{2}$ are special cases of \eqref{exe3} when $(\beta(x),c)$ is $(x^{2}+1,0)$ and $(e^{-x},-\infty)$ respectively $\forall x\ge 0.$ Observe that the simple Ornstein-Uhlenbeck process is a trivial case when $(h(x),\beta(x),c)$ is $(x,1,0)\s\forall x\ge 0.$  Following result displays the explicit long-time behavior of $X$ (in \eqref{exe3}) given the semi-Markov process $Y$:
\bt\label{Ex1}
Suppose Assumptions \ref{As0} holds for underlying semi-Markovian environment $Y$ with the limiting distribution $\pi$ and $U$ is a random variable such that $U\sim \pi.$ For any $h,\beta,c$ the long time vehaviour of the weak solution of $X:=(X_{t}:t\ge 0)$ defined in \eqref{exe3} will depend on different signs of $E_{\pi}a(\cdot)$ and some integrability conditions involving with $b(\cdot),$ and they are following: 
\begin{enumerate}[(A)]
\item Suppose Assumptions \ref{As0} and \ref{as2} hold  with $(a,b)$ replaced by $(2a,b^{2})$. Weak limit of $(X_{t},Y_{t})$ will satisfy
\beqn
(X_{t},Y_{t}) \,\,\stackrel{d}{\to}\,\, \bigg(\sum_{j\in S}\delta_U(\{j\})h^{-1}\big(\sqrt{Z_j} N\big),U\bigg)\quad \text{as } t\to\infty. \label{P1e1}
\eeqn
where $U\ci (Z_j)_{j\in S}$, $U\sim \pi$,   and $Z_j$ is defined in \eqref{mixture} with $(a,b)$ replaced by $(2a,b^{2}),$ and $N\sim N(0,1)$ is independent of everything.
\item Suppose Assumptions \ref{As0} \& \ref{as3} hold with $(a,b)$ replaced by $(2a,b^{2})$  in Assumption \ref{as3}(c). Let $U\sim \pi,$ and $N$ be a standard normal random variable that are mutually independent with each other.
\begin{enumerate}[(i)] 
\item If $E_{\pi}a(\cdot)<0,$ then as $t\to\infty$
\beqn
\frac{\log|h(X_{t})| + tE_{\pi}a(\cdot)}{\sqrt{t}}\stackrel{d}{\to}\sum_{j\in S}\delta_{U}(\{j\})\frac{\sigma_{j}}{\sqrt{E|\mathfrak{I}_{1}^{j}|}}N.
\eeqn
\item Suppose $E_{\pi}a(\cdot)=0, b\neq 0$ and Assumption \ref{as3}(c) is replaced by 
\beqn
E\Big[\Big(\log \int_{\tau_{0}^{j}}^{\tau_{1}^{j}}b^{2}(Y_{s})e^{-\int_{s}^{\tau_{1}^{j}}2a(Y_{r})dr}ds\Big)^{2}\Big]<\infty,\label{Theorem2bcond1*}
\eeqn
then as $t\to\infty,$
\beqn
\frac{\log|h(X_{t})|}{\sqrt{t}}\stackrel{d}{\to}\sum_{j\in S}\delta_{U}\big(\{j\}\big)\frac{\sigma_{j}}{\sqrt{E|\mathfrak{I}_{1}^{j}|}}|N|.
\eeqn
\end{enumerate} 
\item Suppose Assumptions \ref{As0} \& \ref{as4} hold for some $\alpha\in (1,2)$ and slowly varying function $L(\cdot),$  with $(a,b)$ replaced by $(2a,b^{2})$ in Assumption \ref{as4}(c). Let $U\sim \pi,$ and Stable-$\alpha_{}$ process $\mathcal{\bf{S}}_{\alpha,\beta_{j}}$  be all mutually independent for each $j\in S.$ As $t\to\infty,$ 
\begin{enumerate}[(i)] 
\item if $E_{\pi}a(\cdot)<0,$ \beqn
\frac{\log|h(X_{t})| + tE_{\pi}a(\cdot)}{t^{\frac{1}{\alpha}}L_{0}(t)} \s\stackrel{d}{\to}\,\,\sum_{j\in S}\delta_{U}\big(\{j\}\big)\Big(\frac{1}{E|\mathfrak{I}^{j}_{1}|}\Big)^{\frac{1}{\alpha}}\sigma_{(j,\alpha)}\mathcal{S}_{\alpha}(1,\beta_{j},0).\label{eT2aEx1}
\eeqn
\item if $E_{\pi}a(\cdot)=0, b\neq 0$ and Assumption \ref{as4}(c) is replaced by  \beqn
E\Big[\Big(\log \int_{\tau_{0}^{j}}^{\tau_{1}^{j}}b^{2}(Y_{s})e^{-\int_{s}^{\tau_{1}^{j}}2a(Y_{r})dr}ds\Big)^{\alpha+\epsilon}\Big]<\infty\s\text{for some }\,\, \epsilon>0.\label{Theorem2bcond*Ex1}
\eeqn
Then for a slowly varying function $L_{0}(\cdot)$ as $t\to\infty,$
\beqn
\frac{\log|h(X_{t})|}{t^{\frac{1}{\alpha}}L_{0}(t)}\stackrel{d}{\to}\sum_{j\in S}\delta_{U}(\{j\}) \Big(\frac{1}{E|\mathfrak{I}^{j}_{1}|}\Big)^{\frac{1}{\alpha}}\sigma_{(j,\alpha)}\sup_{0\le u \le 1}\mathcal{\bf{S}}_{\alpha,\beta_{j}}(u).\non
\eeqn
\end{enumerate} 
\item Supposse Assumption \ref{As0} holds and $a(\cdot)=$ a constant denoted by $a.$ $U\sim\pi,$ and $N\sim N(0,1)$ are independent. As $t\to\infty,$ if
\begin{enumerate}[(i)]
\item $a<0$ and $E\log^{+}\big[\int_{\tau_{0}^{j}}^{\tau_{1}^{j}}b^{2}(Y_{s})e^{2a(s-\tau_{0}^{j})}ds\big]<\infty,$ then $$e^{at}h(X_{t})\stackrel{d}{\to}h(X_{0})+\Big(\sum_{j\in S}\delta_{U}(\{j\})\sqrt{\int_{0}^{\tau_{0}^{j}}b^{2}(Y_{s})e^{2as}ds +e^{2a\tau_{0}^{j}}V_{j}^{*}}\Big)N,$$ \newline where $V_{j}^{*}$ satisfies \eqref{x=ax+b} with $$(A,B)\stackrel{d}{=}\Big(e^{2a(\tau_{1}^{j} - \tau_{0}^{j})},\int_{\tau_{0}^{j}}^{\tau_{1}^{j}}b^{2}(Y_{s})e^{2a(s-\tau_{0}^{j})}ds\Big).$$
\item $a=0, \,\,\&\s E_{\pi}b^{2}(\cdot)<\infty,$ then $\s\s\frac{h(X_{t})}{\sqrt{t}}\stackrel{d}{\to}\sqrt{E_{\pi}b^{2}(\cdot)}N.$
\end{enumerate}
\end{enumerate}
\et

We will see a little variations in the distributional weak limits in parts (A),(D) if we have a symmetric $\alpha-$stable noise process; but parts (B),(C) remain the same (except the integrability assumptions).

\begin{Corollary}\label{CC1} We consider a case where $(h(x),\beta(x),c)$ is $(x,1,0)\s\forall x\ge 0,$ in \eqref{exe3}, but the noise process is a symmetric stable-$\alpha^{*}$ process $\mathcal{\bf{S}}_{\alpha^{*},0}$ instead of a Wiener process for some $1< \alpha^{*}< 2$. Hence the process $X$ satisfies 
\beqn
dX_{t}=-a(Y_{t})X_{t}dt+b(Y_{t})d\mathcal{\bf{S}}_{\alpha^{*},0}(t).\label{SaS}
\eeqn
\begin{enumerate}[(A)]
\item Suppose Assumptions \ref{As0} and \ref{as2} hold  with $(a,b)$ replaced by $(\alpha^{*}a,b^{^{\alpha^{*}}})$. Weak limit of $(X_{t},Y_{t})$ will satisfy
\beqn
(X_{t},Y_{t}) \,\,\stackrel{d}{\to}\,\, \bigg(\sum_{j\in S}\delta_U(\{j\})|Z^{*}_j|^{1/\alpha^{*}} \mathcal{S}_{\alpha^{*}}(1,0,0),U\bigg)\quad \text{as } t\to\infty. \label{P1e1sas}
\eeqn
where $U\ci (Z^{*}_j)_{j\in S}$, $U\sim \pi$,   and $Z^{*}_j$ is defined in \eqref{mixture} with $(a,b)$ replaced by $(\alpha^{*}a,b^{^{\alpha^{*}}}),$ and $ \mathcal{S}_{\alpha^{*}}(1,0,0)$ is a symmetric-$\alpha^{*}$ stable random variable with scale $ =1$ and it is independent of everything. 
\item Suppose Assumptions \ref{As0} \& \ref{as3} hold with $(a,b)$ replaced by $(\alpha^{*}a,b^{^{\alpha^{*}}})$  in Assumption \ref{as3}(c). Let $U\sim \pi,$ and $N$ be a standard normal random variable that are mutually independent with each other.
\begin{enumerate}[(i)] 
\item If $E_{\pi}a(\cdot)<0,$ then as $t\to\infty$
\beqn
\frac{\log|X_{t}| + tE_{\pi}a(\cdot)}{\sqrt{t}}\stackrel{d}{\to}\sum_{j\in S}\delta_{U}(\{j\})\frac{\sigma_{j}}{\sqrt{E|\mathfrak{I}_{1}^{j}|}}N.
\eeqn
\item Suppose $E_{\pi}a(\cdot)=0, b\neq 0$ and Assumption \ref{as3}(c) is replaced by 
\beqn
E\Big[\Big(\log \int_{\tau_{0}^{j}}^{\tau_{1}^{j}}b^{^{\alpha^{*}}}(Y_{s})e^{-\int_{s}^{\tau_{1}^{j}}\alpha^{*} a(Y_{r})dr}ds\Big)^{2}\Big]<\infty,\label{Theorem2bcond1*}
\eeqn
then as $t\to\infty,$
\beqn
\frac{\log|X_{t}|}{\sqrt{t}}\stackrel{d}{\to}\sum_{j\in S}\delta_{U}\big(\{j\}\big)\frac{\sigma_{j}}{\sqrt{E|\mathfrak{I}_{1}^{j}|}}|N|.
\eeqn
\end{enumerate} 
\item Suppose Assumptions \ref{As0} \& \ref{as4} hold for some $\alpha\in (1,2)$ and slowly varying function $L(\cdot)$, with $(a,b)$ replaced by $(\alpha^{*}a,b^{^{\alpha^{*}}})$ in Assumption \ref{as4}(c). Let $U\sim \pi,$ and Stable-$\alpha_{}$ process $\mathcal{\bf{S}}_{\alpha,\beta_{j}}$  be all mutually independent for each $j\in S.$ As $t\to\infty,$ 
\begin{enumerate}[(i)] 
\item if $E_{\pi}a(\cdot)<0,$ \beqn
\frac{\log|X_{t}| + tE_{\pi}a(\cdot)}{t^{\frac{1}{\alpha}}L_{0}(t)} \s\stackrel{d}{\to}\,\,\sum_{j\in S}\delta_{U}\big(\{j\}\big)\Big(\frac{1}{E|\mathfrak{I}^{j}_{1}|}\Big)^{\frac{1}{\alpha}}\sigma_{(j,\alpha)}\mathcal{S}_{\alpha}(1,\beta_{j},0).\label{eT2aEx1}
\eeqn
\item if $E_{\pi}a(\cdot)=0, b\neq 0$ and Assumption \ref{as4}(c) is replaced by  \beqn
E\Big[\Big(\log \int_{\tau_{0}^{j}}^{\tau_{1}^{j}}b^{^{\alpha^{*}}}(Y_{s})e^{-\int_{s}^{\tau_{1}^{j}}\alpha^{*}a(Y_{r})dr}ds\Big)^{\alpha+\epsilon}\Big]<\infty\s\text{for some }\,\, \epsilon>0.\label{Theorem2bcond*Ex1}
\eeqn
Then for a slowly varying function $L_{0}(\cdot)$ as $t\to\infty,$
\beqn
\frac{\log|X_{t}|}{t^{\frac{1}{\alpha}}L_{0}(t)}\stackrel{d}{\to}\sum_{j\in S}\delta_{U}(\{j\}) \Big(\frac{1}{E|\mathfrak{I}^{j}_{1}|}\Big)^{\frac{1}{\alpha}}\sigma_{(j,\alpha)}\sup_{0\le u \le 1}\mathcal{\bf{S}}_{\alpha,\beta_{j}}(u).\non
\eeqn
\end{enumerate} 
\item Supposse Assumption \ref{As0} holds and $a(\cdot)= a$ (constant). Let $U\sim\pi,$ and $ \mathcal{S}_{\alpha^{*}}(1,0,0)$ be a symmetric-$\alpha$ stable random variable with scale$ =1,$ and they are independent. As $t\to\infty,$ if
\begin{enumerate}[(i)]
\item $a<0$ and $E\log^{+}\big[\int_{\tau_{0}^{j}}^{\tau_{1}^{j}}b^{\alpha^{*}}(Y_{s})e^{\alpha^{*}a(s-\tau_{0}^{j})}ds\big]<\infty,$ then $$e^{at}X_{t}\stackrel{d}{\to}X_{0}+\Big(\sum_{j\in S}\delta_{U}(\{j\})\big|\int_{0}^{\tau_{0}^{j}}b^{\alpha^{*}}(Y_{s})e^{\alpha^{*}as}ds +e^{\alpha^{*}a\tau_{0}^{j}}V_{j,\alpha^{*}}^{*}\big|^{\frac{1}{\alpha^{*}}}\Big)\mathcal{S}_{\alpha^{*}}(1,0,0),$$ \newline where $V_{j,\alpha^{*}}^{*}$ satisfies \eqref{x=ax+b} with $$(A,B)\stackrel{d}{=}\Big(e^{\alpha^{*}a(\tau_{1}^{j} - \tau_{0}^{j})},\int_{\tau_{0}^{j}}^{\tau_{1}^{j}}b^{\alpha^{*}}(Y_{s})e^{\alpha^{*}a(s-\tau_{0}^{j})}ds\Big).$$
\item $a=0, \,\,\&\s E_{\pi}b^{\alpha^{*}}(\cdot)<\infty,$ then $\s\s\frac{X_{t}}{t^{\frac{1}{\alpha^{*}}}}\stackrel{d}{\to}\big|E_{\pi}b^{^{\alpha^{*}}}(\cdot)\big|^{\frac{1}{\alpha^{*}}}\mathcal{S}_{\alpha^{*}}(1,0,0).$
\end{enumerate}
\end{enumerate}
\end{Corollary}

\begin{remark}\label{RR1*}
\begin{enumerate}[(i)]
\item Both Theorem \ref{exe3} and Corollary \ref{CC1} show that weak limits depend on the noise process \textbf{only} in the stable domain (i.e., when \( E_{\pi}a(\cdot) > 0 \)) and in the divergence domain when \( \sigma_{j}^{2} = 0 \) for all \( j \in S \). However, in other divergent cases (B) \& (C) (where \( \sigma_{j}^{2} \) is finite or infinite for all \( j \in S \)) the weak limits do not depend on the volatility term \( b(\cdot) \). Consequently, the limit results in cases (B) and (C) are universal and remain unchanged regardless of choice of noise processes, provided the relevant integrability conditions (Assumptions \ref{as3}(c), \ref{as4}(c), or conditions \eqref{Theorem2bcond1*}, \eqref{Theorem2bcond*Ex1}, with exponent \( 2 \) replaced by \( \alpha^{*} \) in all stable \( \alpha^{*} \) cases) hold. This phenomenon arising from the logarithmic transformation and associated scaling used in the divergent cases is quite surprising.

\item The stable case of Corollary \ref{CC1}(A) provides an explicit form of the limiting measure, clearly showing how its heavy tail structure arises from the stable distribution of noise process and the stochastic environments through the mixing measure. This explicit form can be used for parametric statistical inference to estimate $\alpha^{*}$ when partial information about the latent dynamics is available.
\end{enumerate}
\end{remark}

\section{An important Lemma}\label{Lem2} 
By notation $\mathcal{F}^{Y}_{\tau_{g^{j}_{t}}^{j}},$ define the sigma field generated by the class of sets $\{A: A\cap \{g_{t}^{j}=n\}\in \mathcal{F}^{Y}_{\tau_{n}^{j}}\}.$ Crucial elements of main proofs rely on the following lemma on asymptotic conditional independence of $t-\tau_{g_{t}^{j}}^{j}$ and a $\Big\{\mathcal{F}^{Y}_{\tau_{g^{j}_{t}}^{j}}\Big\}_{t\ge 0}$ measurable $\R^{d}$ valued process $(H_{t}:t\ge 0)$ conditioned on $\{Y_{t}=j\}.$ We present Lemma \ref{lem01} separately, It is useful in eliminating $Y_{t}=j$ in the $t \to \infty$ limit from the conditioning part of a probability, particularly when the main event is regarding some functional of $Y,$ that exhibits a marginal weak limit with some properties (see \eqref{key}). By $\mb{0}_{d},I_{d}$ we respectively denote a $d$-dimensional vector of zeros and a diagonal matrix of order $d.$

\begin{lemma}\label{lem01}
Fix $j\in S.$ Suppose $(H_{t}:t\ge 0)$ is a $\Big\{\mathcal{F}^{Y}_{\tau_{g^{j}_{t}}^{j}}\Big\}_{t\ge 0}$ measurable $\R^{d}$ valued process such that there exists a random variable $H_{\infty},$ and an increasing deterministic function $\epsilon(\cdot):\R_{\ge 0}\to\R_{\ge 0}$ such that  as $t\to\infty$
\beqn
H_{t}\stackrel{d}{\to} H_{\infty},\s\frac{\epsilon(t)}{t}\to 0,\,\,\,\,\epsilon(t)\to\infty,\,\,\text{as}\,\,\, t\to\infty, \s\text{and}\s H_{t}-H_{t-\epsilon(t)}\stackrel{P}{\to} \mb{0}_{d}\label{key}
\eeqn
 and $Y$ is a semi-Markov process with Assumption \ref{As0}. Suppose further that $\big(L^{(1)}_{t},L^{(2)}_{t}\big)_{t\ge 0}$ are arbitrary $d$-dimensional vector and matrix valued processes satisfying $\big(L^{(1)}_{t},L^{(2)}_{t}\big)\stackrel{P}{\to} (\mb{0}_{d},I_{d})$ as $t\to\infty$. Then, $L^{(1)}_{t}+L^{(2)}_{t}H_{t}$ and $(t-\tau_{g_{t}^{j}}^{j})$ are asymptotically independent conditioned on $\{Y_{t}=j\};$ i.e, for any $A\in\mathcal{B}(\R^{d})$ such that $P[H_{\infty}\in \partial A]=0$, 

\begin{align}
\lim_{t\to\infty}P\Big[L^{(1)}_{t}+L^{(2)}_{t}H_{t}\in A,\,\,t-\tau_{g_{t}^{j}}^{j}>x \mid Y_{t}=j\Big]=P\Big[H_{\infty}\in A\Big]\,\,\,\frac{\sum_{k\in S}P_{jk}\int_{x}^{\infty}(1-F_{jk}(y))dy}{m_{j}},\non
\end{align}
for any $x\in\R_{\ge 0}.$
\end{lemma}
\begin{remark}
Both the process $(H_{t}:t\ge 0)$ and the random variable $H_{\infty}$ depend on $j\in S,$ but for notational simplicity the subscript $j$ is omitted.
\end{remark}
\begin{proof} Since under necessary assumptions $\lim_{t\to\infty}P[Y_{t}=j]=\pi_{j}=\frac{\mu_{j}m_{j}}{\sum_{k\in S}\mu_{k}m_{k}},$ the main assertion will follow if we prove that, for any $A\in\mathcal{B}(\R^{d})$ with $P\big[H_{\infty}\in \partial A\big]=0$, as $t\to\infty,$ 
\beqn
\s P\Big[H_{t}\in A,\,\,t-\tau_{g_{t}^{j}}^{j}>x, Y_{t}=j\Big] \to P\Big[H_{\infty}\in A\Big]\,\,\,\frac{\mu_{j}\sum_{k\in S}P_{jk}\int_{x}^{\infty}(1-F_{jk}(y))dy}{\sum_{k\in S}\mu_{k}m_{k}},\label{cond0.02}
\eeqn
then, by Slutsky's theorem, the assertion will hold as
$$\lim_{t\to\infty}P\Big[L_{t}^{(1)}+L_{t}^{(2)}H_{t}\in A,\,\,t-\tau_{g_{t}^{j}}^{j}>x, Y_{t}=j\Big] =\lim_{t\to\infty}P\Big[H_{t}\in A,\,\,t-\tau_{g_{t}^{j}}^{j}>x, Y_{t}=j\Big].$$
We prove \eqref{cond0.02} in following steps sequentially. 
\begin{itemize}
\item \textbf{Step $1$:} $\lim_{t\to\infty}P_{_i}\Big[t-\tau_{g_{t}^{j}}^{j}>x, Y_{t}=j\Big] =\frac{\mu_{j}\sum_{k\in S}P_{jk}\int_{x}^{\infty}(1-F_{jk}(y))dy}{\sum_{k\in S}\mu_{k}m_{k}}.$
\item \textbf{Step $2$:} $\lim_{t\to\infty}P\Big[H_{t-\epsilon(t)}\in A\,\mid\,t-\tau_{g_{t}^{j}}^{j}>x, Y_{t}=j\Big]=P[H_{\infty}\in A].$
\item \textbf{Step $3$:} $\lim_{t\to\infty}P\Big[H_{t}\in A\,\mid\,t-\tau_{g_{t}^{j}}^{j}>x, Y_{t}=j\Big]=\lim_{t\to\infty}P\Big[H_{t-\epsilon(t)}\in A\,\mid\,t-\tau_{g_{t}^{j}}^{j}>x, Y_{t}=j\Big]$ for any $A$ such that $P\big[H_{\infty}\in\partial A\big]=0.$
\end{itemize}
\textbf{Step 1} follows by similar arguments used in proving asymptotic time limit of forward residual time $\lim_{t\to\infty}P[\tau^{j}_{g_{t}^{j}+1} - t>x, Y_{t}=j]$ at display (8.5) in \cite{cinlar1969markov}. For completeness we give a proof in the appendix as Lemma \ref{step1}. 
\item \textbf{Step 2:} Denote the event $\{Y_{t}=j,t-\tau^{j}_{g_{t}^{j}}>x\}$ by $\widetilde{C}^{(j)}_{t,x}.$ Observe that for any set $A\in\mathcal{B}(\R^d),$
\begin{eqnarray}
P\Big[H_{t-\epsilon(t)}\in A\,\mid\,\widetilde{C}^{(j)}_{t,x}\Big] &=&\sum_{j'\in S}P\Big[H_{t-\epsilon(t)}\in A\,\mid\, Y_{t-\epsilon(t)}=j',\widetilde{C}^{(j)}_{t,x}\Big]P\Big[Y_{t-\epsilon(t)}=j'\mid \widetilde{C}^{(j)}_{t,x}\Big]\non\\
&=&\sum_{j'\in S}\frac{P\big[\widetilde{C}^{(j)}_{t,x}\mid Y_{t-\epsilon(t)}=j', H_{t-\epsilon(t)}\in A\big]}{P\big[\widetilde{C}^{(j)}_{t,x}\big]}P\big[H_{t-\epsilon(t)}\in A, Y_{t-\epsilon(t)}=j'\big]\non\\
&=& P[H_{t-\epsilon(t)}\in A]+ \sum_{j'\in S}P\Big[H_{t-\epsilon(t)}\in A,\, Y_{t-\epsilon(t)}=j'\Big]\label{eqSt2.10}\\ &&\s\s\times\Bigg(\frac{P[\widetilde{C}^{(j)}_{t,x}\mid H_{t-\epsilon(t)}\in A, Y_{t-\epsilon(t)}=j' ]}{P[\widetilde{C}^{(j)}_{t,x}]}-1\Bigg).\non
\end{eqnarray}

Denote the second term in RHS of \eqref{eqSt2.10} by $\mathcal{D}_{t}$ and the quantity in first bracket by $$M_{t}:=\frac{P[t-\tau_{g_{t}^{j}}^{j}>x, Y_{t}=j\mid H_{t-\epsilon(t)}\in A, Y_{t-\epsilon(t)}=j' ]}{P[t-\tau_{g_{t}^{j}}^{j}>x, Y_{t}=j]}.$$ In order to show $\mathcal{D}_{t}\to 0,$ we first prove that $\lim_{t\to\infty}M_{t}\to 1.$  

Denote the set $\{\tau_{g_{t-\epsilon(t)}^{j}+2}^{j}\le \tau_{g_{t}^{j}}^{j}\}$ by $K_{t}.$ Observe that 

$$K_{t}=\Big\{B_{j}(t-\epsilon(t))+(\tau_{g^{j}_{t-\epsilon(t)}+2}^{j}-\tau_{g^{j}_{t-\epsilon(t)}+1}^{j})+A_{j}(t)\le \epsilon(t)\Big\}.$$ 

Clearly $P[K_{t}]\to 1,$ as $t\to\infty$ (since all quantities $B_{j}(t-\epsilon(t)), \tau_{g^{j}_{t-\epsilon(t)}+2}^{j}-\tau_{g^{j}_{t-\epsilon(t)}+1}^{j}$ and $A_{j}(t)$ are $O_{_P}(1)$ (using \eqref{eRes} and Proposition \ref{Pr1}(i)) and $t-\epsilon(t)\to\infty,\epsilon(t)\to\infty$ as $t\to\infty$).

Define  $\widetilde{a}_{t},\widetilde{b}_{t}$ as
\beqn
\widetilde{a}_{t}&:=& P\big[K_{t}\mid Y_{t-\epsilon(t)}=j', H_{t-\epsilon(t)}\in A\big],\non\\
\widetilde{b}_{t}&:=& P\big[K_{t}^{c}, (t-\tau_{g_{t}^{j}}^{j})>x, Y_{t}=j\mid H_{t-\epsilon(t)}\in A, Y_{t-\epsilon(t)}=j'\big],\non
\eeqn
 and it is easy to see  that as $t\to\infty,$ $$(\widetilde{a}_{t},\widetilde{b}_{t})\to (1,0).$$

Observe that the numerator of $M_{t}$
\begin{align}
&P\big[t-\tau_{g_{t}^{j}}^{j}>x, Y_{t}=j\mid H_{t-\epsilon(t)}\in A, Y_{t-\epsilon(t)}=j'\big]\non\\&\s=\,\,\,\,\,\,\widetilde{a}_{t}\,\,P\big[t-\tau_{g_{t}^{j}}^{j}>x, Y_{t}=j\mid K_{t}, H_{t-\epsilon(t)}\in A, Y_{t-\epsilon(t)}=j'\big]+\widetilde{b}_{t}\non\\
&\s\stackrel{(a)}{=}\,\,\,\,\,\,\,\,\widetilde{a}_{t}\,\, P\big[t-\tau_{g_{t}^{j}}^{j}>x, Y_{t}=j\mid  K_{t}\big] +\widetilde{b}_{t}\,\,  ,\label{eq_Mt1}
\end{align}
where equality in $(a)$ follows by observing that conditioned on $K_{t}, \{Y_{t-\epsilon(t)}=j'\}$ the event $\{t-\tau^{j}_{g_{t}^{j}}>x, Y_{t}=j\}$ is $\sigma\{\mathcal{H}_{i}: i\ge g_{t-\epsilon(t)}^{j}+2\}$ measurable where $H_{t-\epsilon(t)}$ is $\sigma\{\mathcal{H}_{i}: i\le g_{t-\epsilon(t)}^{j}\}$ measurable so using Proposition \ref{Pr1}(iii) one can remove $\{H_{t-\epsilon(t)}\in A\}$ from the conditioning event. Furthermore, $\{Y_{t-\epsilon(t)}=j'\}$ can be removed from the conditional part as well, since it is $\mathcal{F}^{Y}_{\tau_{g^{j}_{t-\epsilon(t)}+1}^{j}}$ measurable which is independent of $\sigma\{\mathcal{H}_{i}: i\ge g_{t-\epsilon(t)}^{j}+2\}$ under which $\{t-\tau^{j}_{g_{t}^{j}}>x, Y_{t}=j\}$ is measurable conditioned on $K_{t}.$

Since $(\widetilde{a}_{t},\widetilde{b}_{t})\to (1,0),$ both the numerator and denominator of $M_{t}$ will have identical time limit, it follows that $\lim_{t\to\infty}M_{t}\to 1.$ Now we prove $D_{t}\to0$ as $t\to\infty.$

 Denote $\frac{\mu_{j}\sum_{k\in S}P_{jk}\int_{x}^{\infty}(1-F_{jk}(y))dy}{\sum_{k\in S}\mu_{k}m_{k}}$ by $\pi_{j,x}.$ For a fixed $j\in S, x\ge 0,$ and any $\delta_{x} \in (0,\pi_{j,x})$  there exists $t_{\delta_{x}}>0$ such that 
\beqn
\Big|P\big[t-\tau_{g_{t}^{j}}^{j}>x, Y_{t}=j\big]-\pi_{j,x}\Big|\leq \delta_{x},\label{estep2}
\eeqn
 for $t\ge t_{\delta_{x}}$.

Observe that each summand in the second term of the RHS of \eqref{eqSt2.10} converges to $0$ as $t\to\infty$. Also note that, for any $x>0, t\ge t_{\delta_{x}}$, the $j'$-th term has the following upper bound
\begin{align*}
&P\big[H_{t-\varepsilon(t)}\in A, Y_{t-\varepsilon(t)}=j'\big]\big|M_{t}-1\big| \\
&\s\s\le P\big[Y_{t-\varepsilon(t)}=j'\big]\frac{2}{P\big[t-\tau_{g_{t}^{j}}^{j}>x, Y_{t}=j\big]} \\
&\s\s\le P\big[Y_{t-\varepsilon(t)}=j'\big]\frac{2}{\pi_{j,x}-\delta_{x}}
\end{align*} 
and the upper bound is summable over $j'\in S$ and has a limit as $t\to\infty$ which is also summable over $j'\in S$. Hence, by Pratt's lemma (p.~101 in Schilling \cite{Schilling05}), the sum in the RHS of \eqref{eqSt2.10} converges to $D_{t}\to0$ as $t\to\infty$.

From assumption ($\frac{\epsilon(t)}{t}\to 0$) the assertion $H_{t-\epsilon(t)}\stackrel{d}{\to} H_{\infty}$ follows, showing the first term in \eqref{eqSt2.10} converging to $P[H_{\infty}\in A]$ as asserted.

\textbf{Step 3:} Take any $\delta>0$ and define $A^{\delta}:=\{x: d(x,A)<\delta\},A^{-\delta}:=\{x\in A: d(x,A^c)> \delta\}$ and note that $A^{-\delta}\subseteq A\subseteq A^{\delta}$. Denoting $B_{t}:=H_{t}-H_{t-\epsilon(t)}$ one has 
\begin{eqnarray}
P\big[H_{t}\in A \,\mid\,t-\tau_{g_{t}^{j}}^{j}>x, Y_{t}=j\big]=P\big[H_{t-\epsilon(t)}+B_{t}\in A, |B_{t}|\le \delta \,\mid\,t-\tau_{g_{t}^{j}}^{j}>x, Y_{t}=j\big]+C_{t}\non\label{cond0.03}
\end{eqnarray}
where $$ \limsup_{t\to\infty}C_{t}\le \limsup_{t\to\infty}\frac{P[|B_{t}|>\delta]}{P\big[t-\tau_{g_{t}^{j}}^{j}>x, Y_{t}=j\big]}=0.$$  
Observe that
\beqn
P\big[H_{t-\epsilon(t)}\in A^{-\delta}\mid \,\,t-\tau_{g_{t}^{j}}^{j}>x, Y_{t}=j\big]&\le& P\big[H_{t-\epsilon(t)}+B_{t}\in A, |B_{t}|\le \delta\mid \,\,t-\tau_{g_{t}^{j}}^{j}>x, Y_{t}=j\big]\non\\
&\le& P\big[H_{t-\epsilon(t)}\in A^{\delta}\mid \,\,t-\tau_{g_{t}^{j}}^{j}>x, Y_{t}=j\big].\non
\eeqn
Step $3$ follows by letting $\delta\to 0$ after taking $t\to\infty$ to both sides of aforementioned step and applying result  obtained from Step 2 (using continuity set assumption of $A$ i.e \\ $P[H_{\infty}\in \partial A]=0$).
\end{proof}

\section{Conclusion and future directions} \label{secConclusion}
\begin{enumerate}[(i)]
\item  
The regime-switching framework of CTMC in \cite{lindskog2018exact} is extended to a semi-Markovian regime switching while preserving the renewal regenerative structure. This extension is crucial for the modeling point of view as well as for making long-term prediction of observed stochastic processes via finding the scaled limits. Exact long-time behaviors of several classes of stochastic models, admitting a function of a  specific integral  $\int_{0}^{t}b(Y_{s})e^{-\int_{s}^{t}a(Y_{r})dr}ds,$  become explicit in terms of the parameters that determine the dynamics of the underlying regime process. We applied it further in two examples (Pitchfork bifurcation and regime switching diffusions) and got explicit long-time behaviors in different cases with several interesting observations (Remark \ref{rmk} and Remark \ref{RR1*}). We are curious whether, for regime-switching diffusions that are more general than \eqref{exe3} and \eqref{SaS}, phenomena similar to those in Remark \ref{RR1*} still occur or not. Specifically, we wonder if the universal weak limit results emerge in the divergent cases, that do not depend on the noise process (as in cases (B) and (C)) in the right hand side weak limit, while strong dependence on the noise process is observed in the stable case (A) and the special divergent case (D) in both Theorem \ref{Ex1} and Corollary \ref{CC1}.

 Can this framework be extended to more general regime processes while maintaining the renewal regenerative structure? To address this partially, several works on long-memory (see \cite{berbee1987chains}) and infinite-memory processes (see \cite{graham2021regenerative}) suggest that under certain restrictive conditions, the regenerative renewal property may still hold. However, integrating such formulation with our technique requires the distributional estimate of the asymptotic residual time $$\mathcal{L}\Big(t-\tau^{j}_{g_{t}^{j}}\mid Y_{t}=j\Big)$$ as well as developing an analogous version of Lemma \ref{lem01}. Such derivations may not be as explicit as those for the semi-Markov regime process that we have here, which will be addressed later.

As mentioned in Remark \ref{R0} the renewal regenerative structure of $Y$ will be absent if Assumption \ref{As0}(c) does not hold. Can we have any alternative methodology to circumvent this or at least have some closed-form limit results in such non-regular cases? We conjecture that in the divergent cases, one may still get similar limit theorems for $\frac{\log |I_{t}^{(a,b)}|}{a'(t)}-b'(t)$ for suitable choices of $a'(\cdot),b'(\cdot)$ depending on different cases. Such results for non-regenerative $Y,$ will be explored further in future.

\item In contrast to \cite{lindskog2018exact}, the varying tail structures of sojourn time distributions play a crucial role in shaping the different behaviors of the scaling limits. In the stable region that change is not abrupt as it only affects the common law of $\{(L_{i}^{j},Q_{i}^{j}):i\ge 1\}$ as well as the law of the residual time $(t-\tau^{j}_{g_{t}^{j}}),$ but in the transient region it changes the scaling factor of $\log |I_{t}|$ as it controls the rate on how it will diverge in a long time. In a way, this work illustrates universal results on long-time behaviour of some semi-markovian regime-switching stochastic processes with different tail structures of the sojourn times. 

\vspace{0.5 cm}

As hinted in Remark \ref{Rinfinity}, tails of $|\widetilde{A}_{1}^{j}\cap \mathbb{S}^{c}_{\alpha}|$ and $|\widetilde{A}_{1}^{j}\cap \mathbb{S}^{}_{\alpha}|$ are governed by the underlying chain $J,$ and the assumptions in \eqref{bound3} and $E|\widetilde{A}_{1}^{j}\cap \mathbb{S}^{c}_{\alpha}|^{1+\alpha}<\infty$ are quite restrictive in a sense that they force $|\widetilde{A}_{1}^{j}\cap \mathbb{S}^{c}_{\alpha}|, |\widetilde{A}_{1}^{j}\cap \mathbb{S}^{}_{\alpha}|$ to have tails lighter than $\alpha,$ which is the heaviest tails of the sojourn times. In pathological cases when the tail index of $|\widetilde{A}_{1}^{j}\cap \mathbb{S}^{}_{\alpha}|$ is equal to $\alpha$ or less than that, it will drive the aggregate tail index of $\int_{\tau_{0}^{j}}^{\tau_{1}^{j}}\big(a(Y_{s})-E_{\pi}a(\cdot)\big)ds$ in line of Lemma 3.7 (part (2) and other parts) of \cite{jessen2006regularly} and that will change the statement of Theorem \ref{T3} as well (which will be explored with details in future).

\item As a standing condition Assumption \ref{As0}(b) ensures that all elements of $\mathcal{G}$ should have finite first moment. If that is violated, several problems will arise such as $\int_{\tau_{0}^{j}}^{\tau_{1}^{j}}(a(Y_{s})-E_{\pi}a(\cdot))ds$ can be regularly varying with some $\alpha\in (0,1),$ and as a consequence almost surely/ in probability convergence of $\frac{g_{t}^{j}}{t}$ will not hold anymore,  instead for an increasing function $\alpha(\cdot)$ one has $\frac{g_{t}^{j}}{\alpha(t)}\stackrel{d}{\to} X$ for some strictly positive random variable $X.$ As a consequence, Anscome's law via verifying the contiguity condition \eqref{Anscombe} will not be useful as it requires in probability convergence of the stopped random time to a constant, such as $\frac{g_{t}^{j}}{\alpha(t)}\stackrel{p}{\to} a$ for some $a>0$. There should be a way to circumvent this issue, and it will be addressed in the future.

\item Lemma \ref{lem01} can be seen as a conditional version of the Anscombe's law when the stopped random time (here the renewal time $g_{t}^{j}$) appears from the semi-Markovian contexts and it can be used in more general directions such as in finding the weak limit of  $H_{t}:=\sup_{t\ge 0}I_{t}$. It would be interesting to use similar methods in the context of extremal processes under general regime-switching structures as well.
\end{enumerate}
\section*{Acknowledgements}
The author thanks Prof. Filip Lindskog and Prof. Jochen Broecker for their valuable comments on the first draft, and Prof. Jeffrey Collamore for suggesting the article \cite{berbee1987chains} during his visit to Stockholm University.

\section{Set-up of proof-ideas \& notations}\label{proof} We proceed with the convention $\sum_{i=j}^k a_i=0$ and $\prod_{i=j}^k a_i=1$ if $j>k$ for any $a_i$.
For any functions $c,d:S\to\R$ and $j\in S$, define 
\beqn
\,\,\,\,\, G^{c,d}_{j}(x):=\int_{0}^{x}d(j)e^{-c(j)(x-s)}ds=x d(j)1_{\{c(j)=0\}}+ \frac{d(j)}{c(j)}\Big(1-e^{-xc(j)}\Big)1_{\{c(j)\neq 0\}}.
\label{Gfun}
\eeqn
Let $e^{-\int_{s}^{t}a(Y_{r})dr}$ be denoted by $\Phi(s,t)$ and by this notation $\Phi_{t}=\Phi(0,t).$ Given two functions $a(\cdot),b(\cdot): S\to \R$, we define some random variables $\big\{\big(L_{i}^{j},Q_{i}^{j}\big):i\ge 1\big\}, \big(L_{0}^{j},Q_{0}^{j}\big)$ as
\beqn
(L_{i}^{j},Q_{i}^{j})&:=&\bigg(e^{-\int_{\tau_{i-1}^{j}}^{\tau_{i}^{j}}a(Y_{s})ds},\int_{\tau_{i-1}^{j}}^{\tau_{i}^{j}}b(Y_{s})e^{-\int_{s}^{\tau_{i}^{j}}a(Y_{r})dr}ds\bigg)\,\,\,\,\, \forall i\ge 1,\s\text{and}\label{LKintro1}\\
(L_{0}^{j},Q_{0}^{j})&:=&\bigg(e^{-\int_{0}^{\tau_{0}^{j}}a(Y_{s})ds},\int_{0}^{\tau_{0}^{j}}b(Y_{s})e^{-\int_{s}^{\tau_{0}^{j}}a(Y_{r})dr}ds\bigg).\label{LKintro2}
\eeqn
Observe that if $Y_{0}=j$, then $\tau_{0}^{j}=0$ and $(L^{j}_{0},Q^{j}_{0})=(1,0)$. The sequence of intervals $(\mathfrak{I}^{j}_{k})_{k\ge 1}$ of renewal cycles (recall $\mathfrak{I}^{j}_{k}:=[\tau_{k-1}^{j},\tau_{k}^{j})$) is an i.i.d. sequence and therefore also $(L^{j}_{i},Q^{j}_{i})_{i=1}^{\infty}$ is an i.i.d. sequence under Assumption \ref{As0}. However, for fixed $t>0,$ $(L^{j}_{i},Q^{j}_{i})_{i=1}^{g^{j}_{t}}$ is not an i.i.d. sequence since $g^{j}_{t}$, defined in \eqref{g_{t}}, is a renewal time which depends on the sum of all renewal cycle lengths before time $t$.

Given an arbitrary sequence of iid random variables $\{(A_{i},B_{i}):i\ge 1\}$ having common law $\mathcal{L}(A,B),$ define the random variables $\big\{\big(P^{}_{n}(A,B),\widetilde{P}^{}_{n}(A,B)\big):n\ge 1\big\}$ as following 
\beqn
P^{}_{n}(A,B):=\sum_{k=1}^{n}\bigg(\prod_{i=1}^{k-1}A_{i}^{}\bigg)B_{k}^{},\s\s\widetilde{P}^{}_{n}(A,B):=\sum_{k=1}^{n}\bigg(\prod_{i=k+1}^{n}A_{i}^{}\bigg)B_{k}^{}.\label{PP}
\eeqn
Further note that in general $\{\widetilde{P}^{}_{n}(A,B):n\ge 1\}$ is Markovian and have a random coefficient autoregressive process type evolution $$\widetilde{P}^{}_{n+1}(A,B)=A_{n+1}^{}\widetilde{P}^{}_{n}(A,B)+B_{n+1}^{},\s\text{ with }\widetilde{P}^{}_{n}(A,B)\ci \big(A_{n+1}^{},B_{n+1}^{}\big),$$ while the sequence $\{P^{}_{n}(A,B):n\ge 1\}$ is not Markovian but admits the following representation $$P^{}_{n+1}(A,B)=P^{}_{n}(A,B)+\Big(\prod_{i=1}^{n}A_{i}^{}\Big)B_{n+1}^{}.$$  Observe that $\{(L_{i}^{j},Q_{i}^{j}):i\ge 1\}$ are iid under Assumption \ref{As0}. For notational simplicity, by $P_{g_{t}^{j}}^{j},\widetilde{P}_{g_{t}^{j}}^{j}$ we denote $P_{g_{t}^{j}}^{}(L_{i}^{j},Q_{i}^{j}),\widetilde{P}_{g_{t}^{j}}^{}(L_{i}^{j},Q_{i}^{j})$ respectively. Define 
\beqn
R_{t}:=\big(R_{t}^{(1)}, R_{t}^{(2)}\big):=\bigg(\bigg(\prod_{k=1}^{g_{t}^{j}}L^{j}_{k}\bigg)^{} L^{j}_{0},\s\bigg(\prod_{k=1}^{g_{t}^{j}}L^{j}_{k}\bigg)^{} Q^{j}_{0}+P^{j}_{g_{t}^{j}}\bigg).\label{Rt}
\eeqn
From the aforementioned definitions, $R_{t}$ depends on $j\in S$ but for notational simplicity $j$ is omitted.

\section{Proof of Theorem \ref{T1}} 
We use the following two lemmas to prove the theorem.

\begin{lemma}\label{T1lem1}
Under Assumption \ref{As0}, for any $t>0$ and $A=(A_{1},A_{2})\in\mathcal{B}(\R^{2})$,
\begin{align*}
P\big[(\Phi^{(a)}_{t},I^{(a,b)}_{t}) \in A \mid Y_{0}=i,Y_{t}=j\big]=P\big[\big(Q^{(1)}_{t},Q^{(2)}_{t}\big)\in (A_{1},A_{2}) \mid Y_0=i,Y_t=j\big],
\end{align*}
where $(Q^{(1)}_{t},Q^{(2)}_{t})_{t\ge 0}$ is an $(\mathcal{F}_{t}^{Y})_{t\ge 0}$-adapted process defined as
\beqn
 (Q^{(1)}_{t},Q_{t}^{(2)})&:=&\bigg(e^{ -a(j)(t-\tau_{g_{t}^{j}}^{j})}\bigg(\prod_{k=1}^{g^{j}_{t}}L^{j}_{k}\bigg)L^{j}_{0}\s, \s \label{joint}\\ &&\s\s G_{j}^{a,b^{}}\big(t-\tau_{g_{t}^{j}}^{j}\big)+ e^{-a(j)(t-\tau_{g_{t}^{j}}^{j})}\bigg[\bigg(\prod_{k=1}^{g_{t}^{j}}L^{j}_{k}\bigg)^{} Q^{j}_{0}+P^{j}_{g_{t}^{j}}\bigg]\bigg). \non
\eeqn
\end{lemma}

\begin{lemma}\label{T1lem2}
Suppose Assumptions \ref{As0} and \ref{as2} hold.
\begin{enumerate}[(a)]
\item For each $j\in S,$ as $t\to\infty,$ the process $R:=(R_{t}:t\ge 0)$ defined in \eqref{Rt} marginally satisfies $$R_{t}=\big(R_{t}^{(1)}, R_{t}^{(2)}\big)\stackrel{d}{\to}(0,V^{*}_j)$$  where the random variable $V^{*}_j$ is distributed identically as in \eqref{mixture}.
\item The process $\big(R_{t}^{(2)}:t\ge 0\big)$ satisfies all properties of $(H_{t})_{t\ge 0}$ in \eqref{key} of Lemma \ref{lem01}.
\end{enumerate}
\end{lemma}

We sketch how aforementioned lemmas will help us proving the assertion in \eqref{P1e01}. For any $A_{3}\in\mathcal{B}(\R_{\ge 0}\times \R^{}\times S)$ taking $t\to\infty$ limit on both sides of 

\[
P_{i}[(\Phi^{(a)}_{t},I^{(a,b)}_{t},Y_{t}) \in A_{3}] = \sum_{j\in S}^{}P[Y_{t}=j \mid Y_{0}=i]P\big[(\Phi^{(a)}_{t},I^{(a,b)}_{t},j) \in A_{3} \mid Y_{0}=i,Y_{t}=j\big],\]

and using $P[Y_{t}=j \mid Y_{0}=i]\to \pi_{j}$ (which holds under Assumption \ref{As0}), main assertion will be a consequence of Dominated Convergence Theorem if we show that as $t\to\infty,$

\[
P\big[(\Phi^{(a)}_{t},I^{(a,b)}_{t},Y_{t}) \in A_{3} \mid Y_{0}=i,Y_{t}=j\big]\to P[(0,Z_{j},j)\in A_{3}]\s
\]

where for each $j\in S,$ the random variable $Z_{j}$ is defined in \eqref{P1e01}. Since on conditioning with respect to $\{Y_{t}=j\}$ last co-ordinate of $(\Phi^{(a)}_{t},I^{(a,b)}_{t},Y_{t})$ becomes deterministic, we proceed with $P\big[(\Phi^{(a)}_{t},I^{(a,b)}_{t}) \in A \mid Y_{0}=i,Y_{t}=j\big]$ which has a distributional characterization from Lemma \ref{T1lem1} through $(Q^{(1)}_{t},Q_{t}^{(2)})$. The assertion holds if  we show that as $t\to\infty,$

\beqn
P\Big[(Q^{(1)}_{t},Q^{(2)}_{t})\in A_{} \mid Y_{0}=i,Y_{t}=j\Big]\to P\Big[(0,Z_j)\in A_{}\Big]
\quad A_{}\in \mathcal{B}(\R_{\ge 0}\times \R^{}). \label{toprove2}
\eeqn

Observe that 

\beqn
(Q^{(1)}_{t},Q_{t}^{(2)})=\Big(e^{ -a(j)(t-\tau_{g_{t}^{j}}^{j})}R_{t}^{(1)},G_{j}^{a,b^{}}\big(t-\tau_{g_{t}^{j}}^{j}\big)+e^{ -a(j)(t-\tau_{g_{t}^{j}}^{j})}R_{t}^{(2)}\Big).\label{z_to_r}
\eeqn

The limit in \eqref{toprove2} holds if one can find $t\to\infty$ limit of $\mathcal{L}\Big((t-\tau_{g_{t}^{j}}^{j}), R_{t}\mid Y_{0}=i,Y_{t}=j\Big).$ From Lemma \ref{T1lem2}(a) one has the unconditional $t\to\infty$ weak limit for $R_{t}.$ As $R_{t}$ is a $\mathcal{F}^{Y}_{\tau_{g^{j}_{t}}^{j}}$ measurable process with $R_{t}^{(1)}\stackrel{d}{\to} 0$ established in Lemma \ref{T1lem2}(a), It suffices to find the limit of $\mathcal{L}\Big((t-\tau_{g_{t}^{j}}^{j}), R^{(2)}_{t}\mid Y_{0}=i,Y_{t}=j\Big).$

With the help of Lemma \ref{T1lem2}(b),  Lemma \ref{lem01} is applied by setting $H_{t}:=R_{t}^{(2)}, (L_{t}^{(1)},L_{t}^{(2)})=(0,1),$ and we yield for any $A_{1}\in\mathcal{B}(\R), A_{2}\in\mathcal{B}(\R_{\ge 0})$ as $t\to\infty$

\[P_{i}\big[\big(R_{t}^{(2)},(t-\tau_{g_{t}^{j}}^{j})\big)\in A_{1}\times A_{2}\mid Y_{t}=j\big]\to P[V^{*}_{j}\in A_{1}] \pi_{j}^{*}(A_{2})\]

  where $V^{*}_{j}$ is defined as the unconditional limit of $R^{(2)}_{t}$ in Lemma \ref{T1lem2}(a) and the measure $\pi_{j}^{*}(\cdot)$ is defined in \eqref{pi_j^*}. This way \eqref{toprove2} is established by setting the random variable $T^{j}$ (that is independent of rest of the random variables) such that $\mathcal{L}(T^{j})=\pi^{*}_{j},$ and observing  
\begin{align}
&\mathcal{L}\Big(e^{ -a(j)(t-\tau_{g_{t}^{j}}^{j})}R_{t}^{(1)}, G_{j}^{a,b^{}}\big(t-\tau_{g_{t}^{j}}^{j}\big)+e^{ -a(j)(t-\tau_{g_{t}^{j}}^{j})}R_{t}^{(2)}\mid Y_{0}=i, Y_{t}=j\Big)\non\\ &\s\s\s\s\stackrel{w}{\to}\mathcal{L}\big(0, G_{j}^{a,b^{}}(T^{j})+e^{-a(j)T^{j}}V^{*}_{j}\big)\non\\ &\s\s\s\s
=\mathcal{L}(0,Z_{j})\non
\end{align}

as $t\to\infty$ (where $Z_{j}$ is defined in \eqref{mixture}) proving the assertion in \eqref{toprove2}. 

\hfill$\square$

\subsection{Proof of Lemma \ref{T1lem1}} 
\begin{proof}
Define $\Phi^{}(s,t)$ as $e^{-\int_{s}^{t}a(Y_{r})dr}$ for any $0<s<t.$ On set $\{Y_{0}=i,Y_{t}=j\}=\{\omega\in\Omega: Y_{0}(\omega)=i,Y_{t}(\omega)=j\}$ we decompose the integral type expressions of $(\Phi^{(a)}_{t},I^{(a,b)}_{t})$ through partitioning the interval $[0,t)$ into disjoint intervals as $[0,t)=[0,\tau_{0}^{j})\cup_{k=1}^{g_{t}^{j}}[\tau_{k-1}^{j},\tau_{k}^{j})\cup [\tau^{j}_{g_{t}^{j}},t).$ One has the following representation using the notations in \eqref{LKintro1},\eqref{LKintro2} and \eqref{Rt}
\beqn
\begin{pmatrix}
\Phi^{(a)}_ {t}\\
I^{(a,b)}_{t}  
\end{pmatrix} =\begin{pmatrix}
\Phi(0,\tau_{0}^{j})\Big(\prod_{i=1}^{g_{t}^{j}} \Phi\big(\tau_{i-1}^{j}, \tau_{i}^{j}\big)\Big)\Phi(\tau_{g_{t}^{j}}^{j},t)\\
\int_{0}^{\tau_{0}^{j}}b^{}(Y_{s})\Phi^{}(s,t)ds+\sum_{k=1}^{g_{t}^{j}}\int_{\tau_{k-1}^{j}}^{\tau_{k}^{j}}b^{}(Y_{s})\Phi^{}(s,t)ds  \quad+\int_{\tau_{g_{t}^{j}}^{j}}^{t}b^{}(Y_{s})\Phi^{}(s,t)ds.
\end{pmatrix} \non\\
\stackrel{}{=} \begin{pmatrix}
L^{j}_{0}\bigg(\prod_{k=1}^{g^{j}_{t}}L^{j}_{k}\bigg)\Phi(\tau_{g_{t}^{j}}^{j},t)\\
\Phi^{}(\tau_{g_{t}^{j}}^{j},t)\Big[Q_{0}^{j}\prod_{k=1}^{g_{t}^{j}}L_{k}^{j}+  \sum_{k=1}^{g_{t}^{j}}\Big(\prod_{i=k}^{g_{t}^{j}-1} L_{i+1}^{j}\Big)Q_{k}^{j}\Big]+\int_{\tau_{g_{t}^{j}}^{j}}^{t}b^{}(Y_{s})\Phi^{}(s,t)ds
\end{pmatrix}\non\\
\stackrel{\{Y_{0}=i,Y_{t}=j\}}{=} \begin{pmatrix}
L^{j}_{0}\bigg(\prod_{k=1}^{g^{j}_{t}}L^{j}_{k}\bigg)e^{ -a(j)(t-\tau_{g_{t}^{j}}^{j})}\\
e^{-a(j)(t-\tau_{g_{t}^{j}}^{j})}\bigg[\bigg(\prod_{k=1}^{g_{t}^{j}}L^{j}_{k}\bigg)^{} Q^{j}_{0}+\widetilde{P}^{j}_{g_{t}^{j}}\bigg]+G_{j}^{a,b^{}}\big(t-\tau_{g_{t}^{j}}^{j}\big)\label{joint1}
\end{pmatrix}
\eeqn
where the last equality is a consequence of the fact that on $\{Y_{0}=i,Y_{t}=j\}$ the latent process $Y$ will always be at state $j$ in interval $[\tau_{g_{t}^{j}}^{j},t)$ and that implies 

\[
\int_{\tau_{g_{t}^{j}}^{j}}^{t}b^{}(Y_{s})\Phi^{}(s,t)ds
= \int_{\tau_{g_{t}^{j}}^{j}}^{t} b(j)e^{-a(j)(t-s)}ds=\int_{0}^{t-\tau_{g_{t}^{j}}^{j}} b(j)e^{-a(j)(t-r)}dr= G_{j}^{a,b^{}}\big(t-\tau_{g_{t}^{j}}^{j}\big).\]

 Note that in \eqref{joint} we have $P_{g_{t}^{j}}^{j}$ which is not same as $\widetilde{P}_{g_{t}^{j}}^{j}$ in \eqref{joint1}.  Observe that 
\beqn
&\mathcal{L}\bigg((L_{0}^{j},Q_{0}^{j}), (L_{1}^{j},Q_{1}^{j}),(L_{2}^{j},Q_{2}^{j}),\ldots,(L_{g_{t}^{j}}^{j},Q_{g_{t}^{j}}^{j}),(t-\tau_{g_{t}^{j}}^{j})\mid Y_{0}=i,Y_{t}=j\bigg)\non\\
&\,\,\,\s\s\s\stackrel{}{=}\mathcal{L}\bigg((L_{0}^{j},Q_{0}^{j}), (L_{g_{t}^{j}}^{j},Q_{g_{t}^{j}}^{j}), (L_{g_{t}^{j}-1}^{j},Q_{g_{t}^{j}-1}^{j}),\ldots,(L_{1}^{j},Q_{1}^{j}),(t-\tau_{g_{t}^{j}}^{j})\mid Y_{0}=i,Y_{t}=j\bigg).\label{invar}
\eeqn
This follows from the fact that \(g_{t}^{j}\) represents the renewal time, which depends solely on the total sum of the interval lengths \( [0, \tau_{0}^{j}) \cup_{k=1}^{g_{t}^{j}} [\tau_{k-1}^{j}, \tau_{k}^{j}) \). This sum remains unchanged if we reverse the order of the regenerating renewal intervals in \( \cup_{k=1}^{g_{t}^{j}} [\tau_{k-1}^{j}, \tau_{k}^{j}) \) by replacing each \(k\) with \(g_{t}^{j}-k+1\) for \(k=1,\ldots, g_{t}^{j}\). As a result, the equality in distribution holds for both sides of \eqref{invar}.

 Since $P_{g_{t}^{j}}^{j}$ and $\widetilde{P}_{g_{t}^{j}}^{j}$ are outputs of identical function of the random variables in LHS and RHS of \eqref{invar}. This implies that
\beqn
\Big((L_{0}^{j},Q_{0}^{j}), P_{g_{t}^{j}}^{j},(t-\tau_{g_{t}^{j}}^{j})\Big)\stackrel{\mathcal{L}(\cdot\mid Y_{0}=i,Y_{t}=j)}{=}\Big((L_{0}^{j},Q_{0}^{j}), \widetilde{P}_{g_{t}^{j}}^{j},(t-\tau_{g_{t}^{j}}^{j})\Big)\label{ePP}
\eeqn
 are weakly identical. Hence, 
\beqn
&&\mathcal{L}\bigg(L^{j}_{0}\bigg(\prod_{k=1}^{g^{j}_{t}}L^{j}_{k}\bigg)e^{ -a(j)(t-\tau_{g_{t}^{j}}^{j})}\s,\s e^{-a(j)(t-\tau_{g_{t}^{j}}^{j})}\bigg[\bigg(\prod_{k=1}^{g_{t}^{j}}L^{j}_{k}\bigg)^{} Q^{j}_{0}+\widetilde{P}^{j}_{g_{t}^{j}}\bigg]\non\\&&\s\s\s\s\s\s\s\s\s\s\s\s\s\s\s+\,\,G_{j}^{a,b^{}}\big(t-\tau_{g_{t}^{j}}^{j}\big)\mid Y_{0}=i,Y_{t}=j\bigg)
\non\\&&\s=\mathcal{L}\Big(Q_{t}^{(1)},Q_{t}^{(2)}\mid Y_{0}=i,Y_{t}=j\Big),\non
\eeqn
holds, proving the assertion of Lemma \ref{T1lem1}.
\end{proof}

Proof of Lemma \ref{T1lem2} is given in the Supplementary Material. Since the proof of Theorem \ref{T2} closely follows the approach of Theorem \ref{T3} (and Theorem 2 of \cite{lindskog2018exact}), it is added in the Supplementary material.

\section{Proof of Theorem \ref{T3}}
Denote $t^{1/\alpha}L_{0}(t), n^{1/\alpha}L_{0}(n)$ by $c_{t},c_{n}$ respectively. We begin by stating the following Lemma and prove the theorem in the following subsections. 
\begin{lemma}\label{regV}
Suppose Assumptions \ref{As0} and \ref{as4}(a,b) hold. Then for each $j\in S,$ there exist $\alpha\in(1,2)$ and a slowly varying function $L(\cdot)$ (specified in Assumption \ref{as4}(b)) such that 
\beqn
\lim_{x\to\infty}\frac{P\Big[\int_{\tau_{0}^{j}}^{\tau_{1}^{j}}\big(a(Y_{s})-E_{\pi}a(\cdot)\big)ds>x\Big]}{x^{-\alpha}L(x)}&=&\widetilde{\alpha}_{j}^{(+)}=\frac{1-\beta_{j}}{2}\sigma^{\alpha}_{(j,\alpha)},\label{reg1}\\ \lim_{x\to\infty}\frac{P\Big[\int_{\tau_{0}^{j}}^{\tau_{1}^{j}}\big(a(Y_{s})-E_{\pi}a(\cdot)\big)ds < -x\Big]}{x^{-\alpha}L(x)}&=&\widetilde{\alpha}_{j}^{(-)}=\frac{1+\beta_{j}}{2}\sigma^{\alpha}_{(j,\alpha)}.\label{reg2}
\eeqn
\end{lemma}

\subsection{Proof of Theorem \ref{T3}(a)}
\begin{proof}
We begin with $I_{t}^{(a,b)}=\Phi_{t}^{(a)}\widetilde{I}_{t}^{(a,b)},$ where $\widetilde{I}^{(a,b)}_{t}=\int_{0}^{t}b(Y_{s})e^{\int_{0}^{s}a(Y_{r})dr}ds.$ Fix arbitrary $i,j\in S$. Observe that on $\{Y_{0}=i,Y_{t}=j\}$ using renewal regenerative structure of $Y$ (as a consequence of Assumption 1)
\beqn
\Phi_{t}^{(a)}=L_{0}^{j}\Big(\prod_{i=1}^{g_{t}^{j}}L_{i}^{j}\Big)e^{-a(j)(t-\tau_{g_{t}^{j}}^{j})},\,\,\, \widetilde{I}^{(a,b)}_{t}=\bigg[\frac{Q_{0}^{j}+S^{**}_{g_{t}^{j}}}{L_{0}^{j}}+\frac{G_{j}^{(a,b)}(t-\tau_{g_{t}^{j}})}{\Phi_{t}^{(a)}}\bigg]\non
\eeqn
where for any integer $n\ge 1,\s$ $S^{**}_{n}:=P_{n}\big(\frac{1}{L_{1}^{j}},\frac{Q_{1}^{j}}{L_{1}^{j}}\big)=\sum_{i=1}^{n}\frac{Q_{i}^{j}}{L_{i}^{j}}\prod_{k=1}^{i-1}\Big(\frac{1}{L_{k}^{j}}\Big).$

 It follows that
\begin{align}
&\mathcal{L}\bigg[\bigg(\frac{\log \Phi_{t}^{(a)} +tE_{\pi}a(\cdot)}{t^{\frac{1}{\alpha}}L_{0}(t)},\frac{\log |I_{t}^{(a,b)}| +tE_{\pi}a(\cdot)}{t^{\frac{1}{\alpha}}L_{0}(t)}\bigg)| Y_{0}=i, Y_{t}=j\bigg]\non\\
&\s = \mathcal{L}\bigg[\Big(\widetilde{G}^{*(1)}_{t}+\widetilde{G}^{(2)}_{t},\widetilde{G}^{*(1)}_{t}+\widetilde{G}^{(2)}_{t}+\widetilde{G}^{(3)}_{t}\Big)| Y_{0}=i, Y_{t}=j\bigg],\label{T3e1}
\end{align}
where
\begin{align}
&\widetilde{G}^{*(1)}_{t}=\frac{\sum_{i=1}^{g_{t}^{j}}\log L_{i}^{j}+tE_{\pi}a(\cdot)}{t^{\frac{1}{\alpha}}L_{0}(t)},\,\,\s\widetilde{G}^{(2)}_{t}=\frac{\int_{0}^{\tau_{0}^{j}}a(Y_{s})ds+a(j)(t-\tau_{g_{t}^{j}}^{j})}{t^{\frac{1}{\alpha}}L_{0}(t)},\non\\
&\widetilde{G}^{(3)}_{t}=\frac{\log\Big|\frac{Q_{0}^{j}+S^{**}_{g_{t}^{j}}}{L_{0}^{j}}+\frac{G_{j}^{(a,b)}(t-\tau_{g_{t}^{j}})}{\Phi_{t}^{(a)}}\Big|}{t^{\frac{1}{\alpha}}L_{0}(t)}.\non
\end{align}
Hence for any $A_{1},A_{2}\in\mathcal{B}(\R_{\ge 0}),$
\begin{align}
& P\bigg[\bigg(\frac{\log \Phi_{t}^{(a)} +tE_{\pi}a(\cdot)}{t^{\frac{1}{\alpha}}L_{0}(t)},\frac{\log |I_{t}^{(a,b)}| +tE_{\pi}a(\cdot)}{t^{\frac{1}{\alpha}}L_{0}(t)}\bigg)\in  \cdot | Y_{0}=i\bigg]\label{T3e2}\\
&\s= \sum_{j\in S}P[Y_{t}=j|Y_{0}=i]P\bigg[\big(\widetilde{G}^{*(1)}_{t}+\widetilde{G}^{(2)}_{t},\widetilde{G}^{*(1)}_{t}+\widetilde{G}^{(2)}_{t}+\widetilde{G}^{(3)}_{t}\big)\in \cdot | Y_{t}=j, Y_{0}=i\bigg].\non
\end{align}
Observe that
\beqn
\widetilde{G}^{*(1)}_{t} &=&\frac{\sum_{i=1}^{g_{t}^{j}}(\log L_{i}^{j}+ E_{\pi}a(\cdot) |\mathfrak{I}_{i}^{j}|)}{t^{\frac{1}{\alpha}}L_{0}(t)}+\frac{\big[\tau_{0}^{j}+(t-\tau_{g_{t}^{j}}^{j})\big]E_{\pi}a(\cdot)}{t^{\frac{1}{\alpha}}L_{0}(t)}\non\\
&=&-\frac{c_{g_{t}^{j}}}{c_{t}}\frac{\sum_{i=1}^{g_{t}^{j}}\int_{\tau_{i-1}^{j}}^{\tau_{i}^{j}}\big(a(Y_{s})-E_{\pi}a(\cdot)\big)ds}{(g_{t}^{j})^{\frac{1}{\alpha}}L_{0}(g_{t}^{j})} +\frac{\big[\tau_{0}^{j}+(t-\tau_{g_{t}^{j}}^{j})\big]E_{\pi}a(\cdot)}{t^{\frac{1}{\alpha}}L_{0}(t)}.\non
\eeqn
Now clearly $\Big\{-\int_{\tau_{i-1}^{j}}^{\tau_{i}^{j}}(a(Y_{s})-E_{\pi}a(\cdot))ds: i\ge 1\Big\}$ are iid regularly varying with $1<\alpha<2$, mean $=0,$ with right asymptotic tail-index $\widetilde{\alpha}_{j}^{(-)}=\frac{1+\beta_{j}}{2}\sigma^{\alpha}_{(j,\alpha)}$, left asymptotic tail-index $\widetilde{\alpha}_{j}^{(+)}=\frac{1-\beta_{j}}{2}\sigma^{\alpha}_{(j,\alpha)}$, after scaling by $x^{-\alpha}L(x)$ (from Lemma \ref{regV}). Furthermore, $\lim_{t\to\infty}\frac{c_{g_{t}^{j}}}{c_{t}}=\Big(\frac{1}{E|\mathfrak{I}^{j}_{1}|}\Big)^{\frac{1}{\alpha}}$ (see Lemma \ref{regV2}(a)).

Using for each $j\in S$ the quantity $\beta_{j}$ in \eqref{sig}, applying the stable central limit Theorem 4.5.1 of \cite{whitt2002stochastic}, as sample size $n\to\infty$ 
\beqn
-c^{-1}_{n}\Big(\sum_{i=1}^{n}\int_{\tau_{i-1}^{j}}^{\tau_{i}^{j}}\big(a(Y_{s})-E_{\pi}a(\cdot)\big)ds\Big)\stackrel{d}{\to}\sigma_{(j,\alpha)}\mathcal{S}_{\alpha}(1,\beta_{j},0).\label{stable3a}
\eeqn 
Using \eqref{stable3a} and Theorem 3.2 of \cite{gut2009stopped} for the stopped random time $g_{t}^{j}$ as $t\to\infty,$ 
$$-\frac{\sum_{i=1}^{g_{t}^{j}}\int_{\tau_{i-1}^{j}}^{\tau_{i}^{j}}\big(a(Y_{s})-E_{\pi}a(\cdot)\big)ds}{(g_{t}^{j})^{\frac{1}{\alpha}}L_{0}(g_{t}^{j})}\stackrel{d}{\to}\sigma_{(j,\alpha)}\mathcal{S}_{\alpha}(1,\beta_{j},0).$$
Furthermore since $\tau_{0}^{j}+(t-\tau_{g_{t}^{j}}^{j})$ is $O_{p}(1)$ and  hence 
\beqn
\mathcal{L}\Big(\widetilde{G}_{t}^{*(1)}| Y_{0}=i\Big)\stackrel{w}{\to} \mathcal{L}\bigg(\big(E|\mathfrak{I}^{j}_{1}|\big)^{-\frac{1}{\alpha}}\sigma_{(j,\alpha)} \mathcal{S}_{\alpha}(1,\beta_{j},0)\bigg).\label{clt01}
\eeqn
Since both $\log L_{0}^{j}, (t-\tau_{g_{t}^{j}}^{j})$ are $O_{P}(1)$, hence
\beqn
\mathcal{L}(\widetilde{G}_{t}^{(2)}|Y_{0}=i)\stackrel{w}{\to} \delta_{\{0\}}\label{e3a2}
\eeqn
Lastly using similar arguments used in proof of Theorem \ref{T2}(a), as $t\to\infty,$ one has $S^{**}_{g_{t}^{j}}\stackrel{d}{\to} S^{**}$ where $S^{**}$ satisfies the stochastic recurrence equation
\beqn
S^{**}\stackrel{d}{=}\frac{1}{L_{i}^{j}} S^{**}+\frac{Q_{i}^{j}}{L_{i}^{j}},\s S^{**} \perp \Big(\frac{1}{L_{i}^{j}},\frac{Q_{i}^{j}}{L_{i}^{j}}\Big).\label{e2SRE1}
\eeqn

Also $\Phi_{t}^{(a)}\to\infty$ as $t\to\infty,$ we have $\log \big|\frac{Q_{0}^{j}+S^{**}_{g_{t}^{j}}}{L_{0}^{j}}+\frac{G_{j}^{(a,b)}(t-\tau_{g_{t}^{j}})}{\Phi_{t}^{(a)}}\big|$ is $O_{P}(1).$ 
\beqn
\mathcal{L}(\widetilde{G}_{t}^{(3)}|Y_{0}=i)\stackrel{w}{\to} \delta_{\{0\}}\label{e3a3}
\eeqn
Combining \eqref{clt01},\eqref{e3a2},\eqref{e3a3} and taking 
$$ L_{t}^{(1)}=(\widetilde{G}^{(2)}_{t},\widetilde{G}^{(2)}_{t}+\widetilde{G}^{(3)}_{t})',\s L_{t}^{(2)}=\begin{pmatrix} 1 & 0\\ 0 &1
\end{pmatrix}, \s H_{t}= \begin{pmatrix} \widetilde{G}^{*(1)}_{t} \\ \widetilde{G}^{*(1)}_{t} 
\end{pmatrix}.$$
  in Lemma \ref{lem01} as $t\to\infty$
\begin{align}
&\mathcal{L}\Big(\big(\widetilde{G}^{*(1)}_{t}+\widetilde{G}^{(2)}_{t},\widetilde{G}^{*(1)}_{t}+\widetilde{G}^{(2)}_{t}+\widetilde{G}^{(3)}_{t}\big) | Y_{t}=j, Y_{0}=i\Big)\non\\ &\s\s\stackrel{w}{\to}\big(E|\mathfrak{I}^{j}_{1}|\big)^{-\frac{1}{\alpha}}\sigma_{(j,\alpha)}\big(\mathcal{S}_{\alpha}(1,\beta_{j},0),\mathcal{S}_{\alpha}(1,\beta_{j},0)\big)\label{e3am}
\end{align}
where $\widetilde{G}^{*(1)}_{t}$ satisfies \eqref{key}, i.e,
\beqn
\widetilde{G}^{*(1)}_{t}-\widetilde{G}^{*(1)}_{t-\varepsilon(t)}\to 0,\s\text{as}\s t\to\infty,\label{e3aML}
\eeqn
for some $\varepsilon(t)\to\infty, \& \frac{\varepsilon(t)}{t}\to 0,$  which follows from Lemma \ref{lemverify3a}. Using \eqref{e3am} and $P[Y_{t}=j| Y_{0}=i]\to\pi_{j}$ in \eqref{T3e2} as $t\to\infty$ the assertion will follow. 
\end{proof}
\begin{lemma}\label{lemverify3a}
There exists an increasing function $t\to \varepsilon(t)$ such that $\varepsilon(t)\to\infty$ and $\frac{\varepsilon(t)}{t}\to 0$ as $t\to\infty$ for which \eqref{e3aML} holds. 
\end{lemma}
Proofs of Lemma \ref{regV} and Lemma \ref{lemverify3a} are given in the supplementary material.

\subsection{Proof of Theorem \ref{T3}(b)}
\begin{proof}
We will expand $I_{t}^{(a,b)}$ in a different manner than in Theorem \ref{T3}(a). Fix $i,j\in S$. Given $E_{\pi}a(\cdot)=0,$ the random walk generated by the process $\big(\sum_{k=1}^{n}\log L_{k}^{j}:n\ge 1\big)$ has zero drift as the $k$-th increment $\log L_{k}^{j}=-\int_{\tau_{i-1}^{j}}^{\tau_{i}^{j}}a(Y_{s})ds$ has mean $0,$ but Var$(\log L_{k}^{j})$ doesn't exist unlike Theorem \ref{T2}(b). Let $$S_{n}^{\otimes j}:=\sum_{i=1}^{n}\log L_{i}^{j},\s M_{n}^{\otimes j}:=\max_{1\le k \le n}\Big\{\sum_{i=1}^{k}\log L_{i}^{j}\Big\},\s\text{ and }\s Z_{n}^{\otimes j}:=\max_{1\le k \le n}\Big\{\big(\prod_{i=1}^{k-1}L_{i}^{j}\big)|Q_{k}^{j}|\Big\}.$$ Observe that on $\{Y_{0}=i, Y_{t}=j\},$ using \eqref{joint}
\beqn
\big(\log\Phi_{t}^{(a)},\log |I_{t}^{(a,b)}|\big)&\stackrel{\mathcal{L}(\cdot|Y_{0}=i, Y_{t}=j)}{=}&
\bigg(\log L^{j}_{0}+\log\bigg(\prod_{k=1}^{g^{j}_{t}}L^{j}_{k}\bigg)+\log(e^{ -a(j)(t-\tau_{g_{t}^{j}}^{j})}),\non\\&& \log\bigg|e^{-a(j)(t-\tau_{g_{t}^{j}}^{j})}\bigg[\bigg(\prod_{k=1}^{g_{t}^{j}}L^{j}_{k}\bigg)^{} Q^{j}_{0}+\widetilde{P}^{j}_{g_{t}^{j}}\bigg]+G_{j}^{a,b^{}}\big(t-\tau_{g_{t}^{j}}^{j}\big)\bigg|\bigg)\non\\
&\stackrel{d}{=}&\bigg(S_{g_{t}^{j}}^{\otimes j}+[\log L^{j}_{0}-a(j)(t-\tau_{g_{t}^{j}}^{j})], \non\\&&\s\s
\log Z_{g_{t}^{j}}^{\otimes j} + \log\bigg|B_{t}^{(1)}+\frac{B_{t}^{(2)}}{Z_{g_{t}^{j}}^{\otimes j}}P_{g_{t}^{j}}^{j}\bigg|\bigg)\label{T2be1}
\eeqn
where 
\[B_{t}^{(1)}=\frac{e^{-a(j)(t-\tau_{g_{t}^{j}}^{j})}\bigg(\prod_{k=1}^{g_{t}^{j}}L^{j}_{k}\bigg)^{} Q^{j}_{0}+G_{j}^{a,b^{}}\big(t-\tau_{g_{t}^{j}}^{j}\big) }{Z_{g_{t}^{j}}^{\otimes j}},\s\s
B_{t}^{(2)}=e^{-a(j)(t-\tau_{g_{t}^{j}}^{j})},\]
and denote $\bigg(B_{t}^{(1)}+\frac{B_{t}^{(2)}}{Z_{g_{t}^{j}}^{\otimes j}}P_{g_{t}^{j}}^{j}\bigg)$ by $B^{(3)}_{t}.$ Consider the probability 
\begin{align}
&P\Big[\Big(\frac{\log \Phi_{t}^{(a)}}{t^{\frac{1}{\alpha}}L_{0}(t)},\frac{\log|I_{t}^{(a,b)}|}{t^{\frac{1}{\alpha}}L_{0}(t)}\Big)'\in \cdot |Y_{0}=i\Big]\non\\
&\s=\sum_{j\in S}P[Y_{t}=j|Y_{0}=i]P\Big[\Big(\frac{\log \Phi_{t}^{(a)}}{t^{\frac{1}{\alpha}}L_{0}(t)},\frac{\log|I_{t}^{(a,b)}|}{t^{\frac{1}{\alpha}}L_{0}(t)}\Big)'\in \cdot |Y_{0}=i, Y_{t}=j\Big]\non\\
&\s=\sum_{j\in S}P[Y_{t}=j|Y_{0}=i]P\Big[\widetilde{L}_{t}^{(1)}+L_{t}^{(2)}\widetilde{H}_{t}\in \cdot |Y_{0}=i, Y_{t}=j\Big]\label{T3be1}
\end{align}
where the identity \eqref{T2be1} is used in \eqref{T3be1} and 
$$\widetilde{L}_{t}^{(1)}=\Big(\frac{\log L_{0}^{j}-a(j)(t-\tau_{g_{t}^{j}}^{j})}{t^{\frac{1}{\alpha}}L_{0}(t)},\frac{\log|B_{t}^{(3)}|}{t^{\frac{1}{\alpha}}L_{0}(t)}\Big)',\,\, L_{t}^{(2)}= \begin{pmatrix} 1 & 0\\ 0 &1
\end{pmatrix},\, \widetilde{H}_{t}=\Big(\frac{S^{\otimes}_{g_{t}^{j}}}{t^{\frac{1}{\alpha}}L_{0}(t)},\frac{\log Z_{g_{t}^{j}}^{\otimes}}{t^{\frac{1}{\alpha}}L_{0}(t)}\Big)'.$$ 
We will apply Lemma \ref{lem01} with the above specifications. Observe that both \( B_{t}^{(1)} \) and \( B_{t}^{(2)} \) are \( O_{p}(1) \). For each \( 1 \le k \le n \), we have
\[
-Z_{n}^{\otimes j} \le Q_{k}^{j} \left( \prod_{i=1}^{k-1} L_{i}^{j} \right) \le Z_{n}^{\otimes j}.
\]
This implies that \( -nZ_{n}^{\otimes j} \le P_{n}^{j} \le n Z_{n}^{\otimes j} \) for all \( n \ge 1 \), which further implies that
$\left| \frac{P_{g_{t}^{j}}^{j}}{Z_{g_{t}^{j}}^{\otimes j}} \right| \le g_{t}^{j}.$ Hence, \( |B_{t}^{(3)}| = O_{p}(t) \) and is strictly positive.

Since $L_{0}^{j}, (t-\tau_{g_{t}^{j}}^{j})$ are both $O_{P}(1),$ hence $\widetilde{L}_{t}^{(1)}\stackrel{P}{\to}(0,0)'$ and by Lemma \ref{lemverify3b}(c) $ \widetilde{H}'_{t}$ verifies the condition $\widetilde{H}'_{t}-\widetilde{H}'_{t - \varepsilon(t)}\stackrel{P}{\to} (0,0)$ as $t\to\infty.$ The proof will be done if we show as $t\to\infty$ $$\widetilde{H}_{t}\stackrel{d}{\to}\bigg(\frac{1}{E|\mathfrak{I}^{j}_{1}|}\bigg)^{\frac{1}{\alpha}}\sigma_{(j,\alpha)}\big(\mathcal{\bf{S}}_{\alpha,\beta_{j}}(1),\sup_{0\le u \le 1}\mathcal{\bf{S}}_{\alpha,\beta_{j}}(u)\big),$$
 since then using $P[Y_{t}=j|Y_{0}=i]\to\pi_{j},$ after applying the Lemma \ref{lem01} the LHS of \eqref{T3be1} will lead to
\begin{align}
&\lim_{t\to\infty}P\Big[\Big(\frac{\log \Phi_{t}^{(a)}}{t^{\frac{1}{\alpha}}L_{0}(t)},\frac{\log|I_{t}^{(a,b)}|}{t^{\frac{1}{\alpha}}L_{0}(t)}\Big)\in \cdot |Y_{0}=i\Big]\non\\
&\s=\sum_{j\in S}^{}\pi_{j}\lim_{t\to\infty}P\Big[\Big(\frac{S_{g_{t}^{j}}^{\otimes j}}{t^{\frac{1}{\alpha}}L_{0}(t)}, \frac{\log Z_{g_{t}^{j}}^{\otimes j}}{t^{\frac{1}{\alpha}}L_{0}(t)} \Big)\in \cdot |Y_{0}=i\Big] \label{T3be2}
\end{align}
where we got rid of $\{Y_{t}=j\}$ in the conditioning part of RHS of \eqref{T3be1} by virtue of the Lemma \ref{lem01}. From the definitions of $M_{n}^{\otimes j}, Z_{n}^{\otimes j},$ it follows that
\beqn
M_{g_{t}^{j}}^{\otimes j} -\max_{1\le k\le g_{t}^{j}}\log|Q_{k}^{j}|\s\le\s\log Z_{g_{t}^{j}}^{\otimes j}\s\le\s M_{g_{t}^{j}}^{\otimes j}+\max_{1\le k\le g_{t}^{j}}\log|Q_{k}^{j}|\label{th*}
\eeqn
(see also (4.4) of \cite{hitczenko2011renorming}). Now under the condition \eqref{Theorem2bcond**} Lemma \ref{lemverify3b}(a) holds and Lemma \ref{lemverify3b}(b) below verifies the weak convergence result as $t\to\infty$. Combining these in \eqref{T3be2} the assertion follows in $t\to\infty$ limit. 
\end{proof}

\begin{lemma}\label{lemverify3b}
Under assumptions of Theorem \ref{T3}(b) as $t\to\infty,$
\beqn
&&(a)\s\text{ under \eqref{Theorem2bcond**}\s}\frac{\max_{1\le k\le g_{t}^{j}}\log |Q_{k}^{j}|}{t^{\frac{1}{\alpha}}L_{0}(t)} \stackrel{a.s}{\to} 0,\s\non\\&& (b)\s 
\bigg(\frac{S_{g_{t}^{j}}^{\otimes j}}{t^{\frac{1}{\alpha}}L_{0}(t)}, \frac{ M_{g_{t}^{j}}^{\otimes j}}{t^{\frac{1}{\alpha}}L_{0}(t)} \bigg)\stackrel{d}{\to}\bigg(\frac{1}{E|\mathfrak{I}^{j}_{1}|}\bigg)^{\frac{1}{\alpha}}\sigma_{(j,\alpha)}\big(\mathcal{\bf{S}}_{\alpha,\beta_{j}}(1),\sup_{0\le u \le 1}\mathcal{\bf{S}}_{\alpha,\beta_{j}}(u)\big)\non\\
&&(c) \s \widetilde{H}'_{t}=\bigg(\frac{S_{g_{t}^{j}}^{\otimes j}}{t^{\frac{1}{\alpha}}L_{0}(t)}, \frac{\log Z_{g_{t}^{j}}^{\otimes j}}{t^{\frac{1}{\alpha}}L_{0}(t)} \bigg)\s \text{satisfies the condition in \eqref{key}, i.e }\non\\ &&\s\s\s \widetilde{H}_{t}-\widetilde{H}_{t - \varepsilon(t)}\stackrel{P}{\to} (0,0)'\s\text{for some }\s \varepsilon(\cdot)\s\text{as specified in \eqref{key}}.\non
\eeqn
\end{lemma}

Proof of Lemma \ref{lemverify3b}, along with proofs of results in Remark \ref{verifycond}, Proposition \ref{aconstant}, Theorem \ref{T4} are all given in the supplementary material.

\section{ Corollary \ref{conn1}}

\subsection{Proof of Corollary \ref{conn1}}
For some time-dependent deterministic functions $a_{2}(\cdot),b_{2}(\cdot)$ (such that $a_{2}(t)\to\infty$ as $t\to\infty$)  limit results of $\log|I^{(a,b)}_{t}|$ in Theorem \ref{T2} \& Theorem \ref{T3}  can be summarized as $$\mathcal{L}\Big(\frac{\log|I^{(a,b)}_{t}|}{a_{2}(t)}-b_{2}(t)\Big)\stackrel{w}{\to}\text{some mixture type law},$$
where the functions $a_{2}(\cdot)$ and $b_{2}(\cdot)$ can be figured out separately from the results of Theorems \ref{T2} and \ref{T3}. Notably, $a_{2}(t)$ grows polynomially in $t$, specifically as $t^{\frac{1}{2}}$ in Theorem \ref{T2}, and as $t^{\frac{1}{\alpha}}L_{0}(t)$ for some slowly varying function $L_{0}(\cdot)$ in Theorem \ref{T3}. The term $b_{2}(\cdot)$ takes the form $(-\sqrt{t} E_{\pi}a(\cdot) \& 0)$ (respectively $(t^{1-\frac{1}{\alpha}}L^{-1}_{0}(t) \& 0)$) in Theorem \ref{T2}(a) and \ref{T2}(b) (respectively in Theorem \ref{T3}(a) and \ref{T3}(b)). 

We first prove part (a) for $|\widetilde{Z}^{(1)}_{t}|$ in a general setting, while the proof for $|\widetilde{Z}^{(2)}_{t}|$ requires additional work. Under the assumptions of Theorems \ref{T2}(a) and \ref{T2}(b), we establish the results, with analogous conclusions for Theorems \ref{T3}(a) and \ref{T3}(b), which we omit here. The proof of part (b) follows similar steps as in Theorem \ref{T4}, completing the assertion.

\subsubsection{Proof of part (a)}
\begin{proof}  For simplicity of notation, we denote the scaling and shifting functions as $a_{2}(\cdot),b_{2}(\cdot).$ We use the fact that for any two non-negative random variables $X,Y$ one has $\log \max\{X,Y\}\stackrel{d}{=} \max\{\log X,\log Y\}$ (See Lemma \ref{CorLem}).

Observe that $\widetilde{Z}^{(1)}_{t}=\widetilde{Z}^{*}_{t}c_{t}$ where $\widetilde{Z}^{*}_{t}:=\max\{\Phi_{t}^{(a)}, |I^{(a,b)}_{t}|\}$ and  $c_{t}:=a_{1}\frac{\Phi_{t}^{(a)}}{\widetilde{Z}^{*}_{t}}+b_{1}\frac{|I^{(a,b)}_{t}|}{\widetilde{Z}^{*}_{t}}.$ Note that for all $t\ge 0,$ almost surely $$0\le|c_{t}|\le |a_{1}|+|b_{1}|,$$
showing that $c_{t}$ is $O_{p}(1).$ 
\beqn
\frac{\log|\widetilde{Z}^{(1)}_{t}|}{a_{2}(t)}-b_{2}(t)&=&\frac{\log|\widetilde{Z}^{*}_{t}c_{t}|}{a_{2}(t)}-b_{2}(t)\non\\
&=&\bigg(\frac{\log|\widetilde{Z}^{*}_{t}|}{a_{2}(t)}-b_{2}(t)\bigg)+\frac{\log|c_{t}|}{a_{2}(t)}.\non
\eeqn
Since $\frac{\log|c_{t}|}{a_{2}(t)}\stackrel{P}{\to}0,$ it is enough to show that $t\to\infty$ weak limit of $\frac{\log|\widetilde{Z}^{*}_{t}|}{a_{2}(t)}-b_{2}(t)$ matches with $t\to\infty$ weak limit of $\frac{\log|I^{(a,b)}_{t}|}{a_{2}(t)}-b_{2}(t).$ Using Lemma \ref{CorLem} we observe that
\beqn
\frac{\log|\widetilde{Z}^{*}_{t}|}{a_{2}(t)}-b_{2}(t)&=&\frac{\log\max\{\Phi_{t}^{(a)}, |I^{(a,b)}_{t}|\}}{a_{2}(t)}-b_{2}(t)\non\\
&\stackrel{d}{=}&\max\Big\{\frac{\log \Phi_{t}^{(a)}}{a_{2}(t)}-b_{2}(t),\frac{\log|I^{(a,b)}_{t}|}{a_{2}(t)}-b_{2}(t)\Big\}.\non
\eeqn
Note that any two variables $x,y;$ the function $\max\{x,y\}=\frac{(x+y)}{2}+\frac{|x-y|}{2},$ is a continuous function in its arguments. Using the continuous mapping theorem it is easy to verify that the weak limit of $\max\Big\{\frac{\log \Phi_{t}^{(a)}}{a_{2}(t)}-b_{2}(t),\frac{\log|I^{(a,b)}_{t}|}{a_{2}(t)}-b_{2}(t)\Big\}$ matches with the $t\to\infty$ marginal weak limit of $\frac{\log|I^{(a,b)}_{t}|}{a_{2}(t)}-b_{2}(t)$ in all cases of Theorem \ref{T2} and Theorem \ref{T3}, proving the assertion (1) of part (a).

 For simplicity, let the bivariate $t\to\infty$ weak limit of $\Big(\frac{\big(\log \Phi_{t}^{(a)},\log|I^{(a,b)}_{t}|\big)}{a_{2}(t)}-b_{2}(t)\Big)$ be represented by the laws $\mathcal{W}^{*}_{2(a)},\mathcal{W}^{*}_{2(b)},\mathcal{W}^{*}_{3(a)}, \mathcal{W}^{*}_{3(b)}$ under the assumptions of Theorem \ref{T2}(a),\ref{T2}(b),\ref{T3}(a) and Theorem \ref{T3}(b) respectively, whose second marginals are laws $\mathcal{W}_{2(a)},\mathcal{W}_{2(b)},\mathcal{W}_{3(a)}$ and $\mathcal{W}_{3(b)}$ respectively as defined before Corollary \ref{conn1}.

\vspace{0.3 cm}
To prove the result (2) of part (a) for $\widetilde{Z}^{(2)}_{t},$ we show that the time asymptotic behaviour of $\big(\log \Phi_{t}^{(a)},\frac{\log|I^{(ma,b)}_{t}|}{m}\big)$ and $\big(\log \Phi_{t}^{(a)},\log|I^{(a,b)}_{t}|\big)$ will be identical in distribution under same scaling and shifts, under conditions of Theorem \ref{T2}(a) and \ref{T2}(b) (with some modifications). We show that under conditions of Theorem \ref{T2}(a) with $(a,b)$ replaced by $(ma,b)$ in Assumption \ref{as3}(c) 
\begin{align}
\lim_{t\to\infty}\mathcal{L}\bigg(\frac{\big(\log \Phi_{t}^{(a)},\log|I^{(ma,b)}_{t}|^{\frac{1}{m}}\big)}{a_{2}(t)}-b_{2}(t)\bigg) \stackrel{}{=}\mathcal{W}^{*}_{2(a)},\label{eq4a2}
\end{align}
and under conditions of Theorem \ref{T2}(b) with $(a,b)$ replaced by $(ma,b)$ in \eqref{Theorem2bcond*}
\begin{align}
\lim_{t\to\infty}\mathcal{L}\bigg(\frac{\big(\log \Phi_{t}^{(a)},\log|I^{(ma,b)}_{t}|^{\frac{1}{m}}\big)}{a_{2}(t)}\bigg) \stackrel{}{=}\mathcal{W}^{*}_{2(b)}.\label{eq4a3}
\end{align}

The verifications of \eqref{eq4a2} and \eqref{eq4a3} will be done in Lemma \ref{Cor4lem}. Now we finish the proof by proceeding similar to the proof of $|\widetilde{Z}^{(1)}_{t}|$ and using \eqref{eq4a2} \& \eqref{eq4a3}. Replacing $|I^{(a,b)}_{t}|$ by $|I^{(ma,b)}_{t}|^{\frac{1}{m}},$ we define $$\widetilde{Z}^{(2)}_{t}:=\widetilde{Z}^{*}_{m,t}c_{m,t}$$ where $\widetilde{Z}^{*}_{m,t}:=\max\{\Phi_{t}^{(a)}, |I^{(ma,b)}_{t}|^{\frac{1}{m}}\}$ and  $c_{m,t}:=a_{2}\frac{\Phi_{t}^{(a)}}{\widetilde{Z}^{*}_{m,t}}+b_{2}\frac{|I^{(ma,b)}_{t}|^{1/m}}{\widetilde{Z}^{*}_{m,t}}.$ Since $c_{m,t}$ is asymptotically $O_{p}(1),$ hence $\lim_{t\to\infty}\frac{|c_{m,t}|}{a_{2}(t)}\stackrel{P}{\to} 0.$ Following along the lines of proof of $|\widetilde{Z}^{(1)}_{t}|,$ 
\beqn
\lim_{t\to\infty}\mathcal{L}\bigg(\frac{\log|\widetilde{Z}^{(2)}_{t}|}{a_{2}(t)}-b_{2}(t)\bigg)&=&\lim_{t\to\infty}\mathcal{L}\bigg(\frac{\log|\widetilde{Z}^{*}_{m,t}|}{a_{2}(t)}-b_{2}(t)\bigg)\non\\
&=&\lim_{t\to\infty}\mathcal{L}\bigg(\frac{\log\max\{\Phi_{t}^{(a)}, |I^{(ma,b)}_{t}|^{\frac{1}{m}}\}}{a_{2}(t)}-b_{2}(t)\bigg)\non\\
&\stackrel{d}{=}&\lim_{t\to\infty}\mathcal{L}\bigg(\max\Big\{\frac{\log \Phi_{t}^{(a)}}{a_{2}(t)}-b_{2}(t),\frac{\log|I^{(ma,b)}_{t}|^{\frac{1}{m}}}{a_{2}(t)}-b_{2}(t)\Big\}\bigg).\non
\eeqn
Using Lemma \ref{Cor4lem}(a) \& (b), the result for $|\widetilde{Z}^{(2)}_{t}|$ will follow. 

Analog versions (i.e the second part of part (a)(2)) under conditions of Theorem \ref{T3}(a) and \ref{T3}(b) (with $(a,b)$ replaced by $(ma,b)$ in Assumption \ref{as4}(c) and in \eqref{Theorem2bcond**} respectively) will be similar and we omit their proofs. 
\begin{lemma}\label{Cor4lem}
\begin{enumerate}[(a)]
\item \eqref{eq4a2} holds under assumptions of Theorem \ref{T2}(a) with $(a,b)$ replaced by $(ma,b)$ in Assumption \ref{as3}(c).
\item \eqref{eq4a3} holds under assumptions of Theorem \ref{T2}(b) with $(a,b)$ replaced by $(ma,b)$ in \eqref{Theorem2bcond*}.
\end{enumerate}
\end{lemma}

Assertions of part (b) of the Corollary \ref{conn1} follow similarly by proceeding like proofs of Theorem \ref{T4} (hence omitted).
\end{proof}
Proofs of Lemma \ref{Cor4lem} (a),(b) are given in the supplementary material.

\subsection{Proof of Theorem \ref{Ex4}}
\begin{proof}We begin with the equation at display \eqref{model2} with initial value $\rho^{}_{0}$. Dividing both sides by $\rho^{3}_{t}$ will yield
\beqn
\frac{1}{\rho^{3}_{t}}\frac{d\rho^{}_{t}}{dt}&=&\rho^{-2}_{t}a(Y_{t})-b(Y_{t}),\non\\
-\frac{1}{2}\frac{d\big(\rho^{-2}_{t}\big)}{dt}&=&\big(\rho^{-2}_{t}\big)a(Y_{t})-b(Y_{t}).\non
\eeqn
Denoting $X_{t}:=\rho^{-2}_{t},$ we get a linear ODE for with semi-Markov $Y$ modulated coefficients $a(Y_{t}),b(Y_{t})$ as
\beqn
\frac{dX_{t}}{dt}&=&2b(Y_{t})-2a(Y_{t})X_{t}.\label{andronov0}\\
\frac{d\big[e^{\int_{0}^{t}2a(Y_{s})ds}X_{t}\big]}{dt}&=&2b(Y_{t})e^{\int_{0}^{t}2a(Y_{s})ds}.\non\\
X_{t}\Phi_{t}^{(-2a)}-X_{0}\Phi^{(-2a)}_{0}&=& \int_{0}^{t}2b(Y_{s})\Phi^{(-2a)}_{s}ds,\label{andronov1}\\
&=&2 I^{(-2a,b)}_{t},\non
\eeqn 
where $\Phi^{(-2a)}_{0}=1,$ and \eqref{andronov1} is obtained by multiplying \eqref{andronov0} by $\Phi_{t}^{(-2a)}$ in both sides and integrating in $[0,t],$ using the notations of integrals in \eqref{modelint}. After simplifying we get that conditioned on $\mathcal{F}_{t}^{Y},$
\beqn
X_{t}&=& X_{0} \Phi_{t}^{(2a)}+2I_{t}^{(2a,b)},\non\\
\rho^{-2}_{t}&=&\rho^{-2}_{0}\Phi_{t}^{(2a)}+2I_{t}^{(2a,b)},\label{andronov3}\\
\rho^{2}_{t}&=&\frac{1}{\rho_{0}^{-2}\Phi_{t}^{(2a)}+2I_{t}^{(2a,b)}}.\label{andronov2}
\eeqn
We finish the proof of part (a) when $E_{\pi}a(\cdot)>0.$ Observe that unconditionally under assumptions of Theorem \ref{Ex4}(a) the result of Theorem \ref{T1} translates to
$$(\Phi_{t}^{(2a)},I_{t}^{(2a,b)},Y_{t})\stackrel{d}{\to}\Big(0,\sum_{j\in S}\delta_{U}(\{j\})W^{(2a,b)}_{j},U\Big) \s \text{as} \s t\to\infty,$$
where $W^{(2a,b)}_{j}$ is defined in \eqref{mixture} identically as $Z_{j}$ where $(a,b)$ is replaced by $(2a,b).$

Hence using the continuous mapping theorem, one gets 
\begin{eqnarray}
\lim_{t\to\infty}\mathcal{L}(\rho^{2}_{t},Y_{t})&=&\lim_{t\to\infty}\mathcal{L}\bigg(\frac{1}{\rho_{0}^{-2}\Phi_{t}^{(2a)}+2I_{t}^{(2a,b)}}, Y_{t}\bigg)\non\\
 &=&\Big(\sum_{j\in S}\delta_{U}(\{j\})\frac{1}{2 W^{(2a,b)}_{j}}, U\Big)\non
\end{eqnarray}
proving the assertion in \eqref{W} of the theorem.

Clearly in divergent cases (i.e, $E_{\pi}a(\cdot)\le 0$) the state-space representation of the process $(\rho^{-2}_{t}:t\ge 0)$ in \eqref{andronov3} marginally evolves like $(\widetilde{Z}_{t}^{(1)}:t\ge 0)$ in Corollary \ref{conn1} with $a_{1}=\rho_{0}^{-2},b_{1}=2,$ and $(a,b)$ is replaced by $(2a,b).$ Hence the results for $\log \rho_{t}^{-2}=-2\log |\rho_{t}|$  will be identical with the results for $\log |\widetilde{Z}_{t}^{(1)}|$ under the assumptions of Theorem \ref{T2}, Theorem \ref{T3}, and Theorem \ref{T4} with  $a(\cdot)$ replaced by $2a(\cdot)$ as the only change.  As a result in part (B), the quantity $\widetilde{\sigma}^{2}_{j}$ will be identical with the term $\sigma_{j}^{2}$ in Theorem \ref{T2}. Using Corollary \ref{conn1} the rest parts of Theorem \ref{Ex4} will follow. 
\end{proof}

\subsection{Proof of Theorem \ref{Ex1}}
\begin{proof}
Beginning with the set-up at \eqref{exe3} using differentiability of $\beta(\cdot)$ we observe that $h'(x)=\frac{1}{\beta(x)}, h''(x)=-\frac{\beta'(x)}{\beta^{2}(x)}$. Given the trajectory of $Y$ in [0,t], the SDE of $X$ is defined by
\beqn
dX_{s}&=&\Big[-a(Y_{s})h(X_{s})\beta(X_{s})+\frac{b^{2}(Y_{s})}{2}\beta(X_{s})\beta'(X_{s})\Big]dt+b(Y_{s})\beta(X_{s})dW_{s},\s X_{0}=x_{0}.\non\eeqn
 Applying Ito's lemma with the transformation $h(\cdot)$ on $X$ above yields 
\beqn
dh(X_{s})&=&\Big[0+h'(X_{s})\Big(-a(Y_{s})h(X_{s})\beta(X_{s})+\frac{b^{2}(Y_{s})}{2}\beta(X_{s})\beta'(X_{s})\Big) \non\\&&+ \frac{b^{2}(Y_{s})}{2}h''(X_{s}) \beta^{2}(X_{s})\Big]ds+\Big(h'(X_{s})b(Y_{s})\beta(X_{s})\Big)dW_{s}\non\\
&=& -a(Y_{s})h(X_{s})ds+b(Y_{s})dW_{s}.\non
\eeqn
Now integrating both sides with $s$ in the range $[0,t]$ yields, 
\beqn
h(X_{t})&=& h(x_{0})e^{-\int_{0}^{t}a(Y_{s})ds}+ \int_{0}^{t}b(Y_{s})e^{-\int_{s}^{t}a(Y_{r})dr}dW_s\non\\
&\stackrel{d}{=}& h(x_{0})\Phi^{(a)}_{t}+\sqrt{\int_{0}^{t}b^{2}(Y_{s})e^{-\int_{s}^{t}2a(Y_{r})dr}ds}N\non\\
&\stackrel{}{=}& h(x_{0})\Phi^{(a)}_{t}+\sqrt{I^{(2a,b^{2})}_{t}} N\label{OUe1}
\eeqn
for some $N\sim N(0,1)$ independent of $Y,$ hence independent of $(\Phi^{(a)}_{t},I^{(2a,b^{2})}_{t}).$ Since $h$ does not vanish in $\mathbb{R},$ and it is monotone, hence the inverse $h^{-1}(\cdot)$ exists. Using representation \eqref{OUe1},  under Assumptions \ref{As0} and \ref{as2} using the result of Theorem \ref{T1} and continuous mapping theorem following holds:
$$(h(X_{t}),Y_{t})\stackrel{d}{\to}\Big(\sum_{j\in S}\delta_{U}(\{j\})\sqrt{Z_{j}}N, U\Big)$$
where $U\sim \pi$ independent of $Z_{j}$ and $N\sim N(0,1),$ where $Z_{j}$ satisfies weak solution of \eqref{mixture} with $(a,b)$ replaced by $(2a,b^{2}).$ Now using the continuity of $h^{-1}(\cdot)$ and continuous mapping theorem the result in part (A) follows.

For the rest of part (B), (C), (D) we observe that marginal evolution of the process $$h(X_{t})\stackrel{d}{=}h(x_{0})\Phi^{(a)}_{t}+\sqrt{I^{(2a,b^{2})}_{t}} N,$$ resembles with $Z_{t}^{(2)}$ in Corollary \ref{conn1} with $a_{2}=h(x_{0}),b_{2}= N, m=2;$ hence all the results will follow by replacing $(a,b^{})$ by $(a,b^{2})$ in Corollary \ref{conn1} for $|Z_{t}^{(2)}|$. 
\end{proof}

Proof of Corollary \ref{CC1} is given in the supplementary material.

\section{Appendix}
\begin{lemma}\label{step1} Under Assumption \ref{As0},
$$\lim_{t\to\infty}P_{_i}\Big[t-\tau_{g_{t}^{j}}^{j}>x, Y_{t}=j\Big] =\frac{\mu_{j}\sum_{k\in S}P_{jk}\int_{x}^{\infty}(1-F_{jk}(y))dy}{\sum_{k\in S}\mu_{k}m_{k}}.$$
\end{lemma}

\begin{lemma}\label{lem_T1finite}
Under Assumptions \ref{As0} and \ref{as2} the quantities $\log L_{0}^{j}$ and $|Q_{0}^{j}|$ are $O_{P}(1).$
\end{lemma}

\begin{lemma}\label{regV2}
\begin{enumerate}[(a)]
\item Suppose Assumption \ref{As0} holds and $c_{\cdot}$ is a function defined on $\mathbb{N}$ as \eqref{cn} and its definition is extended to the whole $\R_{\ge 0}$ for any arbitrary slowly varying function $L(\cdot)$. Let $g^{j}_{t}$ be defined as the number of total renewals to the state $j,$ before time $t.$ Then for any fixed $s>0,$ $$\lim_{t\to\infty}c^{-1}_{g_{st}^{j}}/c^{-1}_{t}\stackrel{a.s}{\to} \Big(\frac{s}{E|\mathfrak{I}^{j}_{1}|}\Big)^{-\frac{1}{\alpha}}.$$
\item For any function $L(\cdot)$ slowly varying at $+\infty,$ and any two functions $a_{1}(\cdot),b_{1}(\cdot):\R_{\ge 0}\to \R_{\ge 0}$ such that both $a_{1}(t),b_{1}(t)\to \infty$ as $t\to\infty$ and $a_{1}\le b_{1};$ for any $\eps>0$ there exists  $t_{0,\eps}\ge 0,$ such that for all $t\ge t_{0,\eps},$ 
\beqn
(1-\eps)\Big(\frac{b_{1}(t)}{a_{1}(t)}\Big)^{-\eps}\le\frac{L(b_{1}(t))}{L(a_{1}(t))}\le (1+\eps)\Big(\frac{b_{1}(t)}{a_{1}(t)}\Big)^{\eps}\label{regVp3}
\eeqn
\item For any function $L(\cdot)$ slowly varying at $t\to\infty,$  $\lim_{t\to\infty}\frac{\log L(t)}{\log t}\to 0$. 
\end{enumerate}
\end{lemma}
\begin{lemma}\label{CorLem}
For any two non-negative random variables $X,Y$ one has $$\log \max\{X,Y\}\stackrel{d}{=} \max\{\log X,\log Y\}.$$
\end{lemma}
Proofs of these lemmas are given in the \textbf{Supplementary material}.

\newpage
\section{Supplementary materials of  ``Long time behavior of semi-Markov modulated perpetuity and some related processes"}

The supplementary section contains the proofs of 
\begin{itemize}
\item Lemma \ref{T1lem2} (a)(b) of Theorem \ref{T1}; 
\item Theorem \ref{T2}; 
\item Lemma \ref{regV}, Lemma \ref{lemverify3b} and Lemma \ref{lemverify3a} of Theorem \ref{T3}; 
\item proofs of results in Remark \ref{verifycond}, Proposition \ref{aconstant};
\item Theorem \ref{T4}; 
\item Lemma \ref{Cor4lem}(a)(b) of Corollary \ref{conn1}; Corollary \ref{CC1}, and the results in the Appendix appearing in the main article.
\end{itemize}

\section{Proof of supplementary results of Theorem 1}

We prove Lemma 9.2 here. 
\subsection{Proof of \textbf{Lemma 9.2(a)}} 

We prove this lemma by showing
\begin{itemize}
\item\textbf{Step $1$:} $R_{t}^{(1)}$ and absolute value of first term of $R_{t}^{(2)}$ both will converge to $0$ in probability as $t\to\infty,$ 
\item \textbf{Step $2$:} $P_{g_{t}^{j}}^{j}\stackrel{d}{\to}V_{j}^{*}$ where $V_{j}^{*}$ is defined in (3.3), (3.4) (i.e second term of $R_{t}^{(2)}$ will converge to the solution of SRE),
\end{itemize}
and as a consequence of Slutsky's theorem, the assertion of this lemma holds. 

To show \textbf{Step $1$}, observe that $R_{t}^{(1)}$ and first term of $R_{t}^{(2)}$ are respectively 
\beqn
\bigg(\exp\Big\{t\Big(\frac{\sum_{i=1}^{g_{t}}\log L^{j}_{i}}{g_{t}}+\frac{\log L^{j}_{0}}{g_t}\Big)\frac{g_{t}}{t}\Big\},\exp\Big\{t\Big(\frac{\sum_{i=1}^{g_{t}}\log L^{j}_{i}}{g_{t}}+\frac{\log |Q^{j}_{0}|}{g_t}\Big)\frac{g_{t}}{t}\Big\}\bigg).\non
\eeqn
A consequence of renewal theorem, law of large number (validated under Assumption 1) will give that 
\beqn
\frac{\sum_{i=1}^{g_{t}}\log L^{j}_{i}}{g_{t}}\stackrel{P}{\to} E_{\pi}\log L^{j}_{i}=-E|\mathfrak{I}^{j}_{1}|E_{\pi}a(\cdot)<0\s \text{and}\s\frac{g_{t}}{t}\stackrel{a.s}{\to}\frac{1}{E|\mathfrak{I}^{j}_{1}|}.\label{eRenewal}
\eeqn
Since $\log L_{0}^{j}, |Q_{0}^{j}|$ are both $O_{_{P}}(1)$ a consequence of Lemma 12.2, both of $\frac{\log L^{j}_{0}}{g_t}, \frac{\log |Q^{j}_{0}|}{g_t}$ will go to $0$ in probability. This implies as $t\to\infty,$  $$\bigg(\bigg(\prod_{k=1}^{g_{t}^{j}}L^{j}_{k}\bigg)^{} L^{j}_{0},\bigg(\prod_{k=1}^{g_{t}^{j}}L^{j}_{k}\bigg)^{} |Q^{j}_{0}|\bigg)\stackrel{P}{\to}(0,0).$$ 
\begin{lemma}\label{lemPerp}
\textbf{Step $2$} holds under assumptions of Theorem 1, i.e under Assumptions 1 and 2, $P_{g_{t}^{j}}^{j}\stackrel{d}{\to}V_{j}^{*}$ as $t\to\infty$.
\end{lemma}

\begin{subsubsection}{Proof of Lemma \ref{lemPerp}}
\begin{proof}
Define $P^{j}_{\infty}:=\sum_{k=1}^{\infty}\bigg(\prod_{i=1}^{k-1}L_{i}^{j}\bigg)Q_{k}^{j},$ and we show that $P^{j}_{g_{t}^{j}}\stackrel{P}{\to}P^{j}_{\infty}$ and $P^{j}_{\infty}$ satisfies unique solution of SRE that is identical as $\mathcal{L}(V_{j}^{*})$ (defined in (3.3), (3.4)). Note that
\beqn
P^{j}_{\infty}=\sum_{k=1}^{\infty}\bigg(\prod_{i=1}^{k-1}L_{i}^{j}\bigg)Q_{k}^{j}=Q_{1}^{j}+L_{1}^{j}\Big(\sum_{k=2}^{\infty}\bigg(\prod_{i=2}^{k-1}L_{i}^{j}\bigg)Q_{k}^{j}\Big).\non
\eeqn
Now since $(L_{1}^{j},Q_{1}^{j})\perp \{(L_{i}^{j},Q_{i}^{j}): i\ge 2\},$ it follows that the second term inside first bracket of RHS $\sum_{k=2}^{\infty}\bigg(\prod_{i=2}^{k-1}L_{i}^{j}\bigg)Q_{k}^{j}\stackrel{d}{=} P^{j}_{\infty} $ and also $(L_{1}^{j},Q_{1}^{j})\perp \sum_{k=2}^{\infty}\bigg(\prod_{i=2}^{k-1}L_{i}^{j}\bigg)Q_{k}^{j},$ it follows that $P^{j}_{\infty}$ satisfies the SRE at (3.3). Now since under Assumption 2 $$E\log|L_{1}^{j}|=-E|\mathfrak{I}^{j}_{1}|E_{\pi}a(\cdot)<0,\s\text{ and }\s E\log^{+}|Q_{1}^{j}|<\infty,$$ it follows that the solution to the SRE $X\stackrel{d}{=}Q_{1}^{j}+L_{1}^{j}X$ is unique in distribution (Lemma 1.4(a) and Vervaat's Theorem 1.5 in \cite{vervaat1979stochastic}). So it follows that $\mathcal{L}(P_{\infty}^{j})=\mathcal{L}(V_{j}^{*})$ in (3.3), (3.4).

Observe that for any $t>0$,  
\beqn
P^{j}_{\infty}-P^{j}_{g_{t}^{j}}&\stackrel{}{=}& \Big(\prod_{i=1}^{g_{t}^{j}}L_{i}^{j}\Big)\bigg[\sum_{k=g_{t}^{j}+1}^{\infty}\Big(\prod_{i=g_{t}^{j}+1}^{k-1}L_{i}^{j}\Big)Q^{j}_{k}\bigg]\non\\
&=& \exp\Big\{t\Big(\frac{\sum_{i=1}^{g_{t}}\log L^{j}_{i}}{g_{t}}\Big)\frac{g_{t}^{j}}{t}\Big\} \bigg[\sum_{k=g_{t}^{j}+1}^{\infty}\bigg(\prod_{i=g_{t}^{j}+1}^{k-1}L_{i}^{j}\Big)Q^{j}_{k}\bigg].\non
\eeqn
As $t\to\infty,$ using \eqref{eRenewal} the first term $\prod_{i=1}^{g_{t}^{j}}L_{i}^{j}\stackrel{P}{\to}0.$ To prove the assertion, we prove that the second term (in the RHS above) is $O_{P}(1).$ Observe that for any $t\ge 0,$
\beqn
\big((L^{j}_{g^{j}_{t}+1},Q^{j}_{g^{j}_{t}+1}),(L^{j}_{g^{j}_{t}+2},Q^{j}_{g^{j}_{t}+2}),\ldots \big)\stackrel{d}{=}\big((L^{j}_{1},Q^{j}_{1}),(L^{j}_{2},Q^{j}_{2}),\ldots\big)\label{s'2}
\eeqn
which follows by conditioning w.r.t $g^{j}_{t}$ and using argument similar to Proposition 2.2 (ii). A consequence of \eqref{s'2} gives 
\beqn
P^{j}_{\infty}\stackrel{d}{=} \sum_{k=g_{t}^{j}+1}^{\infty}\bigg(\prod_{i=g_{t}^{j}+1}^{k-1}L_{i}^{j}\Big)Q^{j}_{k}=Q^{j}_{g_{t}+1}+L^{j}_{g_{t}+1}\big(Q^{j}_{g_{t}+2}+L^{j}_{g_{t}+2}\big(\ldots\big)\big),\label{s2'}
\eeqn
and since distribution of $P^{j}_{\infty}$ uniquely exists under Assumption 2, the random quantity $\sum_{k=g_{t}^{j}+1}^{\infty}\bigg(\prod_{i=g_{t}^{j}+1}^{k-1}L_{i}^{j}\Big)Q^{j}_{k}
$ is $O_{_P}(1),$ proving our assertion. 
\end{proof}
\end{subsubsection}

\subsection{Proof of \textbf{Lemma 9.2(b)}} Since first term of $R^{(2)}_{t}$ i.e $\big(\prod_{k=1}^{g_{t}^{j}}L^{j}_{k}\big)^{} Q^{j}_{0}\stackrel{P}{\to}0$ as $t\to\infty.$ The assertion holds if we show (6.1) of Lemma 6.1  with $P_{g_{t}^{j}}^{j}$ as $H_{t}.$ Note that from Lemma 9.2(a) it follows that unconditionally $P_{g_{t}^{j}}^{j}\stackrel{d}{\to}V^{*}_{j},$ so we can take $V^{*}_{j}$ as $H_{\infty}$  in (6.1). We are only left to prove $P_{g_{t}^{j}}^{j} - P_{g_{t-\epsilon(t)}^{j}}^{j}\stackrel{P}{\to} 0$ for some function $\epsilon(\cdot)$ such that $\epsilon(t)\to\infty$ and $\frac{\epsilon(t)}{t}\to 0$ as $t\to\infty.$

\beqn
P_{g_{t}^{j}}^{j}- P_{g_{t-\epsilon(t)}^{j}}^{j}&=&\sum_{k=1}^{g_{t}^{j}}\Big(\prod_{i=1}^{k-1}L_{i}^{j}\Big)Q_{k}^{j}- \sum_{k=1}^{g_{t-\epsilon(t)}^{j}}\Big(\prod_{i=1}^{k-1}L_{i}^{j}\Big)Q_{k}^{j}\non\\
&=&\bigg(\prod_{i=1}^{g_{t-\epsilon(t)}^{j}}L_{i}^{j}\bigg)\bigg[\sum_{k=g^{j}_{t-\epsilon(t)}+1}^{g^{j}_{t}}\bigg(\prod_{i=g_{t-\epsilon(t)}^{j}+1}^{k-1}L_{i}^{j}\bigg)Q_{k}^{j}\bigg]\label{st3e1}
\eeqn
Observe that absolute value of second term of the product in \eqref{st3e1} can be bounded by 
\beqn
\bigg|\sum_{k=g^{j}_{t-\epsilon(t)}+1}^{g^{j}_{t}}\bigg(\prod_{i=g_{t-\epsilon(t)}^{j}+1}^{k-1}L_{i}^{j}\bigg)Q_{k}^{j}\bigg|&\le& \sum_{k=g^{j}_{t-\epsilon(t)}+1}^{g_{t}^{j}}\bigg(\prod_{i=g_{t-\epsilon(t)}^{j}+1}^{k-1}L_{i}^{j}\bigg)|Q_{k}^{j}| \non\\ &\le& \sum_{k=g^{j}_{t-\epsilon(t)}+1}^{\infty}\bigg(\prod_{i=g_{t-\epsilon(t)}^{j}+1}^{k-1}L_{i}^{j}\bigg)|Q_{k}^{j}|.\non
\eeqn
Note that $\sum_{k=g^{j}_{t-\epsilon(t)}+1}^{\infty}\bigg(\prod_{i=g_{t-\epsilon(t)}^{j}+1}^{k-1}L_{i}^{j}\bigg)|Q_{k}^{j}|\stackrel{d}{=}\sum_{k=1}^{\infty}\bigg(\prod_{i=1}^{k-1}L_{i}^{j}\bigg)|Q_{k}^{j}|$ (using ideas similar to \eqref{s'2} and \eqref{s2'}) which satisfies the unique solution of the SRE $$X=L_{1}^{j}X+|Q_{1}^{j}|,\s X\perp (L_{1}^{j},Q_{1}^{j})$$
under conditions $E\log L_{k}^{j}=-E|\mathfrak{I}^{j}_{1}|E_{\pi}a(\cdot)<0,E\log^{+}|Q_{k}^{j}|<\infty$  ensured by Assumption 2. This implies that second term of the product in \eqref{st3e1} is $O_{P}(1)$.

Observe that first term in the product of \eqref{st3e1} is $$\bigg(\prod_{i=1}^{g_{t-\epsilon(t)}^{j}}L_{i}^{j}\bigg)= \exp\Big\{(t-\epsilon(t))\Big(\frac{\sum_{i=1}^{g_{t-\epsilon(t)}}\log L^{j}_{i}}{g_{t-\epsilon(t)}}\Big)\frac{g_{t-\epsilon(t)}^{j}}{t-\epsilon(t)}\Big\}.$$
A consequence of Renewal Theorem and Law of large number leads to \eqref{eRenewal} with $t$ replaced by $t-\epsilon(t)$ and as $t-\epsilon(t) \to\infty,$ one has $\prod_{i=1}^{g_{t-\epsilon(t)}^{j}}L_{i}^{j}\stackrel{P}{\to}0.$

\hfill$\square$

\section{Proof of Theorem 2}
\subsection{Proof of \textbf{Theorem 2(a)}}
$\vspace{0.2cm}$

We begin with $I^{(a,b)}_{t}=\Phi^{(a)}_{t}\widetilde{I}^{(a,b)}_{t}$ where $\widetilde{I}^{(a,b)}_{t}=\int_{0}^{t}b(Y_{s})e^{\int_{0}^{s}a(Y_{r})dr}ds.$ Fix arbitrary $i,j\in S$. Observe that on $\{Y_{0}=i,Y_{t}=j\}$ using renewal regenerative structure of $Y$ (as a consequence of Assumption 1)
\beqn
\Phi_{t}^{(a)}=L_{0}^{j}\Big(\prod_{i=1}^{g_{t}^{j}}L_{i}^{j}\Big)e^{-a(j)(t-\tau_{g_{t}^{j}}^{j})},\,\,\, \widetilde{I}^{(a,b)}_{t}=\bigg[\frac{Q_{0}^{j}+S^{**}_{g_{t}^{j}}}{L_{0}^{j}}+\frac{G_{j}^{(a,b)}(t-\tau_{g_{t}^{j}})}{\Phi_{t}^{(a)}}\bigg]\non
\eeqn
where for any integer $n\ge 1,\s$ $S^{**}_{n}:=P_{n}\big(\frac{1}{L_{1}^{j}},\frac{Q_{1}^{j}}{L_{1}^{j}}\big)=\sum_{i=1}^{n}\frac{Q_{i}^{j}}{L_{i}^{j}}\prod_{k=1}^{i-1}\Big(\frac{1}{L_{k}^{j}}\Big).$

Let $G^{(1)}_{t},G^{(2)}_{t},G^{(3)}_{t}$ be defined as 
\beqn
&&G^{(1)}_{t}=\frac{ \Big(\prod_{i=1}^{g_{t}^{j}}L_{i}^{j}\Big)^{\frac{1}{\sqrt{t}}}}{e^{-\sqrt{t}E_{\pi}a(\cdot)}},\s G^{(2)}_{t}=\Big[L_{0}^{j} e^{-a(j)(t-\tau_{g_{t}^{j}}^{j})}\Big]^{\frac{1}{\sqrt{t}}},\,\non\\
&& G^{(3)}_{t}=\bigg[\frac{Q_{0}^{j}+S^{**}_{g_{t}^{j}}}{L_{0}^{j}}+\frac{G_{j}^{(a,b)}(t-\tau_{g_{t}^{j}})}{\Phi_{t}^{(a)}}\bigg]^{\frac{1}{\sqrt{t}}}.\non
\eeqn

Observe that 
\beqn
\mathcal{L}\Big(|\widetilde{I}^{(a,b)}_{t}|^{\frac{1}{\sqrt{t}}}\mid Y_{0}=i,Y_{t}=j\Big)=\mathcal{L}\Big(G^{(3)}_{t}\mid Y_{0}=i,Y_{t}=j\Big)
\eeqn
and
\beqn
&& \mathcal{L}\bigg(\bigg(\frac{[\Phi^{(a)}_{t}]^{\frac{1}{\sqrt{t}}}}{e^{-\sqrt{t}E_{\pi}a(\cdot)}}, \frac{|I^{(a,b)}_{t}|^{\frac{1}{\sqrt{t}}}}{e^{-\sqrt{t}E_{\pi}a(\cdot)}}\bigg)\mid Y_{0}=i,Y_{t}=j\bigg)\non\\ &&\,\,\,\,\,\,\,\,\,\,=\mathcal{L}\Big(\big(G^{(1)}_{t}G^{(2)}_{t},G^{(1)}_{t}G^{(2)}_{t}G^{(3)}_{t}\big)\mid Y_{0}=i,Y_{t}=j\Big).\non
\eeqn
For any $A_{1},A_{2}\in \mathcal{B}(\R_{\ge 0}),$
\beqn
&&P\bigg[\bigg(\frac{[\Phi^{(a)}_{t}]^{\frac{1}{\sqrt{t}}}}{e^{-\sqrt{t}E_{\pi}a(\cdot)}}, \frac{|I^{(a,b)}_{t}|^{\frac{1}{\sqrt{t}}}}{e^{-\sqrt{t}E_{\pi}a(\cdot)}}\bigg) \in A_{1}\times A_{2}\mid Y_{0}=i\bigg]\non\\
&&\,\,\,\,\,\,=\s\sum_{j\in S}P\bigg[\bigg(\frac{[\Phi^{(a)}_{t}]^{\frac{1}{\sqrt{t}}}}{e^{-\sqrt{t}E_{\pi}a(\cdot)}}, \frac{|I^{(a,b)}_{t}|^{\frac{1}{\sqrt{t}}}}{e^{-\sqrt{t}E_{\pi}a(\cdot)}}\bigg) \in A_{1}\times A_{2}, Y_{t}=j\mid Y_{0}=i\bigg]\non\\
&&\,\,\,\,\,\,=\s\sum_{j\in S} P[Y_{t}=j \mid Y_{0}=i] P\Big[\big(G^{(1)}_{t}G^{(2)}_{t},G^{(1)}_{t}G^{(2)}_{t}G^{(3)}_{t}\big)\in A_{1}\times A_{2}\mid Y_{0}=i,Y_{t}=j\Big].\label{beg}
\eeqn

As a consequence of the Ergodic theorem and Renewal theorem, 
$$E\log L_{1}^{j}=-E|\mathfrak{I}^{j}_{1}| E_{\pi}a(\cdot),\,\,\,\,\&\,\, \,\,\frac{g_{t}^{j}}{t} \stackrel{a.s}{\to}\frac{1}{E|\mathfrak{I}^{j}_{1}|}\,\, \,\,\text{as}\,\,\,\, t\to\infty.$$

Applying Central Limit Theorem for the renewal reward process $\big(\sum_{k=1}^{g_{t}^{j}}\log L_{k}^{j}: t\ge 0\big)$, with the reward $\log L_{k}^{j}$ for the $k$-th interval $\mathfrak{I}_{k}^{j}$ (see Theorem 2.2.5 of \cite{tijms2003first}), with the continuous mappling theorem using the map $x\to e^{x},$ with $N\sim N(0,1),$ 

\beqn
\mathcal{L}\Big(G_{t}^{(1)}| Y_{0}=i\Big)\stackrel{w}{\to} \mathcal{L}\bigg(exp\Big\{\frac{\sigma_{j}}{\sqrt{E|\mathfrak{I}_{k}^{j}|}}N\Big\}\bigg)\label{clt}
\eeqn 

Using arguments similar to step 2 of Lemma 9.2(a)  in Theorem 1, with the conditions (i.e, $E_{\pi}a(\cdot)<0$ \& Assumption 3(c))
$$E\Big[\frac{1}{L_{i}^{j}}\Big]<0, \s E\log^{+}\Big(\frac{Q_{i}^{j}}{L_{i}^{j}}\Big)<\infty,$$
one finds that $S^{**}_{g^{j}_{t}}\to S^{**}$ as $t\to\infty,$ where 
\beqn
S^{**}\stackrel{d}{=}\frac{1}{L_{i}^{j}} S^{**}+\frac{Q_{i}^{j}}{L_{i}^{j}},\s S^{**} \perp \Big(\frac{1}{L_{i}^{j}},\frac{Q_{i}^{j}}{L_{i}^{j}}\Big).\label{e2SRE}
\eeqn

Observe further, that in the transient regime, as $t\to\infty,$
$$\frac{\log \Phi_{t}^{(a)}}{t}\to - E_{\pi}a(\cdot),$$
showing that $\Phi_{t}^{(a)}$ diverges exponentially fast as $t\to\infty.$ Also, under Assumption 1 $G^{(a,b)}_{j}(t-\tau_{g_{t}^{j}})$ is $O_{p}(1)$ and using $S^{**}_{g^{j}_{t}}\to S^{**},$ as $t\to\infty$
\beqn
\mathcal{L}\Big(\big(G_{t}^{(3)}\big)^{\sqrt{t}}|Y_{0}=i\Big)\stackrel{w}{\to} \mathcal{L}\Big(\frac{Q_{0}^{j} + S^{**}}{L_{0}^{j}}\Big),\label{SREdiv}
\eeqn
and this further implies that 
\beqn
\mathcal{L}\big(G_{t}^{(3)}|Y_{0}=i\big)\stackrel{w}{\to} \delta_{\{1\}}\s\text{as} \s t\to\infty.\label{eCLT2}
\eeqn
 Since $G_{t}^{(2)}=e^{\frac{\log L_{0}^{j}}{\sqrt{t}}-a(j)\frac{(t-\tau_{g_{t}^{j}})}{\sqrt{t}}},$ and both $L_{0}^{j},(t-\tau_{g_{t}^{j}})$ are $O_{p}(1),$ as $t\to\infty$
\beqn
\mathcal{L}\big(G_{t}^{(2)}|Y_{0}=i\big)\stackrel{w}{\to} \delta_{\{1\}}.\label{eCLT3}
\eeqn
Combining \eqref{clt}, \eqref{eCLT2}, \eqref{eCLT3} and taking
$$L_{t}^{(1)}=(0 , 0)',\s L_{t}^{(2)}=\begin{pmatrix} G_{t}^{(2)} & 0\\ 0 & G_{t}^{(2)}G_{t}^{(3)}
\end{pmatrix}, \s H_{t}= \begin{pmatrix} G_{t}^{(1)}\\ G_{t}^{(1)}
\end{pmatrix},$$
in Lemma 6.1 as $t\to\infty,$
\beqn
&&\mathcal{L}\Big(G_{t}^{(1)}G_{t}^{(2)},G_{t}^{(1)}G_{t}^{(2)}G_{t}^{(3)}| Y_{0}=i, Y_{t}=j\Big)\non\\ &&\s\stackrel{w}{\to} \mathcal{L}\bigg(exp\Big\{\frac{\sigma_{j}}{\sqrt{E|\mathfrak{I}_{k}^{j}|}}N\Big\},exp\Big\{\frac{\sigma_{j}}{\sqrt{E|\mathfrak{I}_{k}^{j}|}}N\Big\}\bigg)\label{efin}
\eeqn
given $H_{t}$ satisfies the condition (6.1), which holds if $G^{(1)}_{t}$ satisfies (6.1). This holds if $\log G^{(1)}_{t}$ satisfies (6.1), 
\beqn
\log G^{(1)}_{t} - \log G^{(1)}_{t-\epsilon(t)} \stackrel{P}{\to} 0, \s \text{as} \s t\to\infty,\label{everify2a}
\eeqn
for some $\epsilon(\cdot)$ such that $\frac{\epsilon(t)}{t}\to 0$ which follows from Lemma \ref{lemverify2a}.

It follows that, using $\log(\cdot)$ transformation on both sides of \eqref{beg}
\beqn
&& P\bigg[\bigg(\frac{\log\Phi^{(a)}_{t} + tE_{\pi}a(\cdot)}{\sqrt{t}},\frac{\log|I^{(a,b)}_{t}| + tE_{\pi}a(\cdot)}{\sqrt{t}}\bigg)\in A_{1}\times A_{2} | Y_{0}=i\bigg]\non\\
&&\s =\sum_{j\in S}P\big[\big(\log(G^{(1)}_{t}G^{(2)}_{t}),\log(G^{(1)}_{t}G^{(2)}_{t}G^{(3)}_{t})\big)\in A_{1}\times A_{2}| Y_{0}=i, Y_{t}=j\big]\non\\ && \s\s\s \times P[Y_{t}=j | Y_{0}=i],\label{efin2}
\eeqn
Observe that $P[Y_{t}=j | Y_{0}=i]\to \pi_{j},$  and taking $\log(\cdot)$ transformation in both co-ordinates of \eqref{efin} desired limit of each term inside the sum will be obtained. Applying the dominated convergence theorem using the trivial upper bound $1,$ for the first term in the above sum completes the proof. Finally taking the $t\to \infty$ limit, the assertion will follow.
\hfill$\square$

\begin{lemma}\label{lemverify2a}
There exists an increasing function $t\to \epsilon(t)$ such that $\epsilon(t)\to\infty$ and $\frac{\epsilon(t)}{t}\to 0$ as $t\to\infty,$ such that \eqref{everify2a} holds for $\log G_{t}^{(1)}.$
\end{lemma}

\begin{proof} 
Take $\varepsilon(t):=\sqrt{t}$. 
Define $\widetilde{G}^{(1)}_{t} :=\sum_{i=1}^{g_{t}^{j}}(Y^{*}_{i}-Z^{*}_{i})/\sqrt{t}$, where $$Y^{*}_{k}:=\log L^{j}_{k}-E\log L^{j}_{k}, \s Z^{*}_{k}:=( |\mathfrak{I}^{j}_{k}|-E|\mathfrak{I}^{j}_{k}|)E_{\pi}a(\cdot), $$ for $k=1,\ldots,g_{t}^{j}$ are two sequences of random variables with zero means. Note that it is sufficient to prove condition \eqref{everify2a} for $\widetilde{G}^{(1)}_{t}$ since 
\beqn
\log G^{(1)}_{t}&=&\frac{\sum_{i=1}^{g^{j}_{t}}\log L_{i}^{j} - t E_{\pi}a(\cdot)}{\sqrt{t}}\non\\ &=& \widetilde{G}^{(1)}_{t} - \frac{[\tau_{0}+(t-\tau_{g^{j}_{t}})]E_{\pi}a(\cdot)}{\sqrt{t}}
\eeqn
and the second term in the right-hand side above is $o_{P}(1).$ Now observe 
\beqn
\widetilde{G}^{(1)}_{t} - \widetilde{G}^{(1)}_{t-\sqrt{t}}  
&=&\frac{\sum_{i=1}^{g^{j}_{t-\sqrt{t}}}Y^{*}_{i}}{\sqrt{t}}\Big[1-\Big(1-\frac{1}{\sqrt{t}}\Big)^{-1/2}\Big]+\frac{\sum_{i=g^{j}_{t-\sqrt{t}}+1}^{g^{j}_{t}}Y^{*}_{i}}{\sqrt{t}}\non\\
&&- \frac{\sum_{i=1}^{g^{j}_{t-\sqrt{t}}}Z^{*}_{i}}{\sqrt{t}}\Big[1-\Big(1-\frac{1}{\sqrt{t}}\Big)^{-1/2}\Big]-\frac{\sum_{i=g^{j}_{t-\sqrt{t}}+1}^{g^{j}_{t}}Z^{*}_{i}}{\sqrt{t}}
\eeqn
Using the asymptotic approximation 
$$
\Big(1-\frac{1}{\sqrt{t}}\Big)^{-1/2}=1+\frac{1}{2\sqrt{t}}+O(t^{-1})
$$ 
shows that
$$
\frac{\sum_{i=1}^{g^{j}_{t-\sqrt{t}}}Y^{*}_{i}}{\sqrt{t}}\Big[1-\Big(1-\frac{1}{\sqrt{t}}\Big)^{-1/2}\Big]
=-\frac{\sum_{i=1}^{g^{j}_{t-\sqrt{t}}}Y^{*}_{i}}{2t}+o_{_P}(t^{-1/2})\stackrel{P}{\to}-\frac{EY^{*}_{k}}{2E|\mathfrak{I}^{j}_{1}|}=0.
$$ 
 by the renewal version of the law of large numbers. Also $\frac{\sum_{i=1}^{g^{j}_{t-\sqrt{t}}}Z^{*}_{i}}{\sqrt{t}}\Big[1-\Big(1-\frac{1}{\sqrt{t}}\Big)^{-1/2}\Big]\stackrel{P}{\to} 0$ holds similarly. By a similar argument as $t\to\infty,$
\beqn
\frac{\sum_{i=g^{j}_{t-\sqrt{t}}+1}^{g^{j}_{t}}Y^{*}_{i}}{\sqrt{t}}
=\frac{\sum_{i=g^{j}_{t-\sqrt{t}}+1}^{g^{j}_{t}}Y^{*}_{i}}{g^{j}_{t} - g^{j}_{t-\sqrt{t}}}\frac{g^{j}_{t} - g^{j}_{t-\sqrt{t}}}{\sqrt{t}} &\stackrel{\text{a.s.}}{\to}&\frac{1}{E|\mathfrak{I}^{j}_{1}|} EY^{*}_{k}=0,\s\text{and}\non\\
\frac{\sum_{i=g^{j}_{t-\sqrt{t}}+1}^{g^{j}_{t}}Z^{*}_{i}}{\sqrt{t}}&\stackrel{a.s.}{\to}&\frac{EZ^{*}_{k}}{E|\mathfrak{I}^{j}_{1}|} =0.\non
\eeqn
The verification is complete.
\end{proof}

\subsection{Proof of \textbf{Theorem 2(b)}} 
\begin{proof}
We will expand $I_{t}^{(a,b)}$ in a different manner than in Theorem 2(a), but before that we set few ideas and notations regarding random walk with zero drift. Fix $i,j\in S$. Given $E_{\pi}a(\cdot)=0,$ the random walk generated by the process $\big(\sum_{k=1}^{n}\log L_{k}^{j}:n\ge 1\big)$ has zero drift as the $k$-th increment $\log L_{k}^{j}=-\int_{\tau_{i-1}^{j}}^{\tau_{i}^{j}}a(Y_{s})ds$ has mean $0$ and following variance $$\text{Var}\big(\log L_{k}^{j}\big)=\text{Var}\Big(-\int_{\tau_{i-1}^{j}}^{\tau_{i}^{j}}a(Y_{s})ds\Big)=\sigma_{j}^{2}\s\text{(since}\,\, E_{\pi}a(\cdot)=0).$$ Let $$S_{n}^{\otimes j}:=\sum_{i=1}^{n}\log L_{i}^{j},\s M_{n}^{\otimes j}:=\max_{1\le k \le n}\Big\{\sum_{i=1}^{k}\log L_{i}^{j}\Big\},\s\text{ and }\s Z_{n}^{\otimes j}:=\max_{1\le k \le n}\Big\{\big(\prod_{i=1}^{k-1}L_{i}^{j}\big)|Q_{k}^{j}|\Big\}.$$ Observe that on $\{Y_{0}=i, Y_{t}=j\},$ using (9.1)

\beqn
\big(\log\Phi_{t}^{(a)},\log |I_{t}^{(a,b)}|\big)&\stackrel{\mathcal{L}(\cdot|Y_{0}=i, Y_{t}=j)}{=}&
\bigg(\log L^{j}_{0}+\log\bigg(\prod_{k=1}^{g^{j}_{t}}L^{j}_{k}\bigg)+\log(e^{ -a(j)(t-\tau_{g_{t}^{j}}^{j})}),\non\\&& \log\bigg|e^{-a(j)(t-\tau_{g_{t}^{j}}^{j})}\bigg[\bigg(\prod_{k=1}^{g_{t}^{j}}L^{j}_{k}\bigg)^{} Q^{j}_{0}+\widetilde{P}^{j}_{g_{t}^{j}}\bigg]+G_{j}^{a,b^{}}\big(t-\tau_{g_{t}^{j}}^{j}\big)\bigg|\bigg)\non\\
&\stackrel{d}{=}&\bigg(S_{g_{t}^{j}}^{\otimes j}+[\log L^{j}_{0}-a(j)(t-\tau_{g_{t}^{j}}^{j})], \non\\&&\s\s
\log Z_{g_{t}^{j}}^{\otimes j} + \log\bigg|B_{t}^{(1)}+\frac{B_{t}^{(2)}}{Z_{g_{t}^{j}}^{\otimes j}}P_{g_{t}^{j}}^{j}\bigg|\bigg)\label{T2be1}
\eeqn

where 
\beqn
B_{t}^{(1)}&=&\frac{e^{-a(j)(t-\tau_{g_{t}^{j}}^{j})}\bigg(\prod_{k=1}^{g_{t}^{j}}L^{j}_{k}\bigg)^{} Q^{j}_{0}+G_{j}^{a,b^{}}\big(t-\tau_{g_{t}^{j}}^{j}\big) }{Z_{g_{t}^{j}}^{\otimes j}},\non\\
B_{t}^{(2)}&=&e^{-a(j)(t-\tau_{g_{t}^{j}}^{j})},\non
\eeqn
and denote $\bigg(B_{t}^{(1)}+\frac{B_{t}^{(2)}}{Z_{g_{t}^{j}}^{\otimes j}}P_{g_{t}^{j}}^{j}\bigg)$ by $B^{(3)}_{t}.$ We now consider the scaled limit of $\Big(\frac{\log\Phi_{t}^{(a)}}{\sqrt{t}},\frac{\log |I_{t}^{(a,b)}|}{\sqrt{t}}\Big)$ as $t\to\infty.$

Observe that both $B_{t}^{(1)},B_{t}^{(2)}$ are $O_{p}(1),$ and for each $1\le k\le n,$ $$-Z_{n}^{\otimes j}\le Q_{k}^{j}\Big(\prod_{i=1}^{k-1}L_{i}^{j}\Big)\le Z_{n}^{\otimes j}.$$ 
This implies that $-nZ_{n}^{\otimes j}\le P_{n}^{j} \le n Z_{n}^{\otimes j},$ for all $n\ge 1,$ which further implies that $\Big|\frac{P_{g_{t}^{j}}^{j}}{Z_{g_{t}^{j}}^{\otimes j}}\Big|\le g_{t}^{j}$; hence $|B^{(3)}_{t}|=O_{p}(t)$ and strictly positive. Hence as $t\to\infty,$
$$\frac{\log|B_{t}^{(3)}|}{\sqrt{t}}=\frac{\log \Big|B_{t}^{(1)}+\frac{B_{t}^{(2)}}{Z_{g_{t}^{j}}^{\otimes j}}P_{g_{t}^{j}}^{j}\Big|}{\sqrt{t}}\stackrel{P}{\to} 0.$$

Consider the probability,
\beqn
&&P\bigg[\Big(\frac{\log \Phi_{t}^{(a)}}{\sqrt{t}},\frac{\log |I_{t}^{(a,b)}|}{\sqrt{t}}\Big)\in \cdot | Y_{0}=i\bigg]\non\\ \s\s\s\s &&=\sum_{j\in S} P[Y_{t}=j|Y_{0}=i]P\Big[\Big(\frac{\log \Phi_{t}^{(a)}}{\sqrt{t}},\frac{\log |I_{t}^{(a,b)}|}{\sqrt{t}}\Big)\in \cdot  | Y_{0}=i, Y_{t}=j\Big].\label{T2be2}
\eeqn

Let $t\to\infty,$ in both sides of \eqref{T2be2}, and to get rid of $\{Y_{t}=j\}$ in the conditioning part of the RHS in the limit, we apply the Lemma 6.1 with 
$$ L_{t}^{(1)}=\bigg(\frac{[\log L^{j}_{0}-a(j)(t-\tau_{g_{t}^{j}}^{j})]}{\sqrt{t}} , \frac{\log|B_{t}^{(3)}|}{\sqrt{t}}\bigg)',\s L_{t}^{(2)}=\begin{pmatrix} 1 & 0\\ 0 &1
\end{pmatrix}, \s H_{t}= \begin{pmatrix} \frac{S_{g_{t}^{j}}^{\otimes j}}{\sqrt{t}} \\ \frac{\log Z_{g_{t}^{j}}^{\otimes j}}{\sqrt{t}} 
\end{pmatrix}.$$
Since both $\log L^{j}_{0}$ and $a(j)(t-\tau_{g_{t}^{j}}^{j})$ are $O_{P}(1),$ hence $L_{t}^{(1)}\stackrel{P}{\to}0,$ and by Lemma \ref{lemverify2b}(c) $H_{t}$ verifies the condition in (6.1). We will be done if we prove that $H'_{t}\stackrel{d}{\to}\frac{\sigma_{j}}{\sqrt{E|\mathfrak{I}^{j}_{1}|}}(F_{1},F_{2})$ as $t\to\infty.$ This is because using $P\big[Y_{t}=j|Y_{0}=i\big]\to\pi_{j},$ LHS of \eqref{T2be2} would lead to
\beqn
&&\lim_{t\to\infty}P\bigg[\Big(\frac{\log \Phi_{t}^{(a)}}{\sqrt{t}},\frac{\log |I_{t}^{(a,b)}|}{\sqrt{t}}\Big)\in \cdot | Y_{0}=i\bigg]\non\\
\s\s &&=\sum_{j\in S}\pi_{j}\lim_{t\to\infty} P\bigg[\bigg(\frac{S_{g_{t}^{j}}^{\otimes j}}{\sqrt{t}},\frac{\log Z_{g_{t}^{j}}^{\otimes j}}{\sqrt{t}}\bigg)\in \cdot | Y_{0}=i\bigg]\label{T2be3}
\eeqn
proving the result. From the definitions of $M_{n}^{\otimes j}, Z_{n}^{\otimes j},$ it follows that
\beqn
M_{g_{t}^{j}}^{\otimes j} -\max_{1\le k\le g_{t}^{j}}\log|Q_{k}^{j}|\s\le\s\log Z_{g_{t}^{j}}^{\otimes j}\s\le\s M_{g_{t}^{j}}^{\otimes j}+\max_{1\le k\le g_{t}^{j}}\log|Q_{k}^{j}|\label{th*}
\eeqn
(see also (4.4) of \cite{hitczenko2011renorming}). Now under the condition (4.3) Lemma \ref{lemverify2b}(a) holds and Lemma \ref{lemverify2b}(b) below verifies the weak convergence result as $t\to\infty$. Combining these in \eqref{T2be3} the assertion follows by taking $t\to\infty$ limit.

\begin{lemma}\label{lemverify2b}
Under assumptions of Theorem 2(b) as $t\to\infty,$
\beqn
&&(a)\s\frac{\max_{1\le k\le g_{t}^{j}}\log |Q_{k}^{j}|}{\sqrt{t}} \stackrel{a.s}{\to} 0,\s (b)\s 
\bigg(\frac{S_{g_{t}^{j}}^{\otimes j}}{\sqrt{t}}, \frac{ M_{g_{t}^{j}}^{\otimes j}}{\sqrt{t}} \bigg)\stackrel{d}{\to}\frac{\sigma_{j}}{\sqrt{E|\mathfrak{I}^{j}_{1}|}} (F_{1},F_{2})\non\\
&&(c) \s H'_{t}=\bigg(\frac{S_{g_{t}^{j}}^{\otimes j}}{\sqrt{t}}, \frac{\log Z_{g_{t}^{j}}^{\otimes j}}{\sqrt{t}} \bigg)\s \text{satisfies the condition in (6.1), i.e }\non\\ &&\s\s\s H_{t}-H_{t - \varepsilon(t)}\stackrel{P}{\to} (0,0)'\s\text{for some }\s \varepsilon(\cdot)\s\text{specified in (6.1)}.\non
\eeqn
\end{lemma}
\subsubsection{Proof of \textbf{Lemma \ref{lemverify2b} (a) \& (b)}}

\begin{proof} 
For ease of notations, write $\mu_{j}:=1/E|\mathfrak{I}^{j}_{1}|$. Since $\frac{g_{t}^{j}}{t}\to \mu_{j}$ as $t\to\infty.$

Part (a) will be proved if we show $\frac{\max_{1\le k\le g_{t}^{j}}\log |Q_{k}^{j}|}{\sqrt{g_{t}^{j}}} \stackrel{a.s}{\to} 0.$ As $(\log |Q^{j}_{i}|)_{i\ge 1}$ is an i.i.d. sequence and (4.3) implies that $E\big[\log |Q_{1}^{j}|\big]^{2}<\infty.$ By the Borel-Cantelli lemma, we will show that
\beqn
\frac{\max_{1\le k \le n}\log |Q^{j}_{k}|}{\sqrt{n}}\stackrel{\text{a.s.}}{\to}0 \quad\text{as }n\to\infty. \label{2be1}
\eeqn

We first show $\log |Q^{j}_{n}|/\sqrt{n}\stackrel{\text{a.s.}}{\to}0$ as a consequence of Borel-Cantelli Lemma and $E\big[\log |Q_{1}^{j}|\big]^{2}<\infty.$ This is evident as $E\big[\log |Q_{1}^{j}|\big]^{2}=E\big[\log |Q_{n}^{j}|\big]^{2}<\infty$ implies
\begin{align*}
&\sum_{n=1}^{\infty}P\big[\big(\log |Q_{n}^{j}|\big)^{2}>\varepsilon^{2} n\big]<\infty \quad\text{for every } \varepsilon>0,\\
&\quad \Leftrightarrow P\big[\log |Q_{n}^{j}|>\varepsilon \sqrt{n} \text{ i.o.}\big]=0\quad\text{for every } \varepsilon>0.
\end{align*}
Since only finitely many of $\{\log |Q_{k}^{j}|/\sqrt{k}, 1\le k\le n\}$ can exceed $\varepsilon$ for any arbitrarily small $\varepsilon>0,$ \eqref{2be1} holds. Then from this and $g_{t}\stackrel{\text{a.s.}}{\to}\infty;$  part (a) will follow by observing 
\beqn
\Big\{\frac{\max_{1\le k\le g_{t}^{j}}\log |Q_{k}^{j}|}{\sqrt{g_{t}^{j}}}\not\to 0\Big\}= \Big\{\frac{\max_{1\le k \le n}\log |Q^{j}_{k}|}{\sqrt{n}}\not\to 0\Big\}\cup\{g_{t}\not\to\infty\},\label{BClem_a}
\eeqn 
and using Theorem 2.1 in Gut \cite{gut2009stopped} by replacing $n$ by $g_{t}^{j}$ (as $g_{t}^{j}\stackrel{a.s}{\to}\infty$) in
\eqref{2be1}.
\vspace{0.5 cm}

\subsubsection{Proof of part (b)}

Since $\bigg(\frac{S_{g_{t}^{j}}^{\otimes j}}{\sqrt{t}}, \frac{ M_{g_{t}^{j}}^{\otimes j}}{\sqrt{t}} \bigg)=\bigg(\frac{S_{g_{t}^{j}}^{\otimes j}}{\sqrt{g_{t}^{j}}}, \frac{ M_{g_{t}^{j}}^{\otimes j}}{\sqrt{g_{t}^{j}}} \bigg)\sqrt{\frac{g_{t}^{j}}{t}},$ and under Assumption 1, $g_{t}^{j}/t\stackrel{a.s}{\to}\mu_{j},$ part (b) will follow if we prove following two steps,
\begin{itemize}
\item \textbf{Step 1 :} $\Big(\frac{S_{n}^{\otimes j}}{\sqrt{n}},\frac{M_{n}^{\otimes j}}{\sqrt{n}}\Big)\stackrel{d}{\to}\sigma_{j}(F_{1},F_{2})$ as $n\to\infty;$
\item \textbf{Step 2 :} $\Big(\frac{S_{n}^{\otimes j}}{\sqrt{n}},\frac{M_{n}^{\otimes j}}{\sqrt{n}}\Big)$ satisfy Anscombe's contiguity condition for $R_{n}$ in (2.11);
\end{itemize} 
then by replacing $n$ by $g_{t}^{j}$ (as $g_{t}^{j}\stackrel{a.s}{\to}\infty$) the result will follow.
\begin{enumerate}[(A)]
\item (\textbf{Proof of Step 1}) Let $W:=(W_{t}:0\le t\le 1)$ be a Wiener process in $[0,1]$. Let the random curve $x_{(n)}:=(x_{(n)}(t):0\le t\le 1)$ in $[0,1]$ be defined in the following manner
\beqn
x_{(n)}(t):=n^{-\frac{1}{2}}S_{[nt]}^{\otimes j}+(nt-[nt])\frac{\log L_{[nt]+1}^{j}}{\sqrt{n}},\s 0\le t\le 1,
\eeqn
where it represents the continuous time interpolated version of the discrete partial sum process $\{S_{n}^{\otimes j}:n\in\mathbb{N}\}.$ By $(\mathcal{C}[0,1],\mathcal{C})$ denote the space of continuous functions in $[0,1]$ with topology induced under supremum norm. Under conditions of Theorem 2(b) the weak convergence of $x_{(n)}$ in $\mathcal{C}[0,1]$ holds (a.k.a the functional central limit theorem (FCLT)),  
\beqn
x_{(n)}\stackrel{w}{\Longrightarrow}\sigma_{j} W\s \text{in}\s \mathcal{C}[0,1]\label{2bs1e1}
\eeqn 
as $n\to\infty.$ We will use the traditional invariance principle. Note that for any element $x\in \mathcal{C}[0,1],$ the functional $f:\mathcal{C}[0,1]\to (\mathcal{C}[0,1]\times \mathcal{C}[0,1])$ defined as $$f(x):=\Big(\big(x(t),\sup_{0\le s\le t}x(s)\big):t\in[0,1]\Big),$$ is a continuous under supremum norm. Using continuous mapping theorem as $n\to\infty,$ one has $f(x_{(n)})\stackrel{w}{\to}f(\sigma_{j}W)$ in $\mathcal{C}[0,1],$ and as a consequence
\beqn
f(x_{(n)})(1)\stackrel{d}{\to} f(\sigma_{j}W)(1)=\sigma_{j}\big(W_{1},\sup_{0\le s \le 1}W_{s}\big)\s .\label{contmapp}
\eeqn
Observe that $f(x_{(n)})(1)=n^{-\frac{1}{2}}\big(S^{\otimes j}_{n}, M^{\otimes j}_{n}\big).$ Hence \eqref{contmapp} implies that
$$n^{-\frac{1}{2}}\Big(S^{\otimes j}_{n}, M^{\otimes j}_{n}\Big)\stackrel{d}{\to}\sigma_{j}\big(W_{1},\sup_{0\le s \le 1}W_{s}\big),\s \text{as}\s n\to\infty.$$
Corollary 6.5.5 of \cite{resnick1992adventures} gives an explicit form of density of $\big(W_{1},\sup_{0\le s \le 1}W_{s}\big)$  that is identical to the joint density of $(F_{1},F_{2})$ in (4.5).
\item (\textbf{Proof of Step 2}) 

We prove \textbf{Step 2} by showing Anscombe's condition 
(uniform continuity in probability, condition (A) on page 16 in \cite{gut2009stopped}) 
which in our context requires showing the following: Given $\gamma>0,\eta>0$, there exist $\delta>0,n_{0}>0$ such that 
\beqn
P\Big[\max_{\{k: |k-n|<n\delta\}}\Big|\Big(\frac{S^{\otimes j}_{k}}{\sqrt{k}},\frac{M^{\otimes j}_{k}}{\sqrt{k}} \Big)-\Big(\frac{S^{\otimes j}_{n}}{\sqrt{n}},\frac{M^{\otimes j}_{n}}{\sqrt{n}}\Big)\Big|> \gamma \Big]<\eta\quad \text{for } n\ge n_{0}.\label{cond0.001}
\eeqn 
 For any two $a=(a_{1},a_{2}),b=(b_{1},b_{2}),$ we use the $L_{1}$ norm $|a-b|=|a_{1}-a_{2}|+|b_{1}-b_{2}|$ without loss of generality. \eqref{cond0.001} holds if we show that given $\gamma>0,\eta>0$, there exist $\delta>0,n_{1},n_{2}>0$ such that 
\beqn
P\Big[\max_{\{k: |k-n|<n\delta\}}\Big|\frac{S^{\otimes j}_{k}}{\sqrt{k}}-\frac{S^{\otimes j}_{n}}{\sqrt{n}}\Big|> \frac{\gamma}{2} \Big]<\frac{\eta}{2}\quad \text{for } n\ge n_{1}\label{cond0.001e1}\\
P\Big[\max_{\{k: |k-n|<n\delta\}}\Big|\frac{M^{\otimes j}_{k}}{\sqrt{k}} - \frac{M^{\otimes j}_{n}}{\sqrt{n}}\Big|> \frac{\gamma}{2} \Big]<\frac{\eta}{2}\quad \text{for } n\ge n_{2}\label{cond0.001e2}
\eeqn 
by setting $n_{0}=n_{1}\vee n_{2}.$ 

Since $\{k: |k-n|<n\delta\}=\{n\le k\le n(1+\delta)\}\cup\{n(1-\delta)\le k\le n\}$ an upper bound for the probability in \eqref{cond0.001e2}  is 
\beqn
&& P\Big[\max_{\{k: n\le k\le n+n\delta\}}\Big|\frac{M^{\otimes j}_{k}}{\sqrt{k}}-\frac{M^{\otimes j}_{n}}{\sqrt{n}}\Big|> \frac{\gamma}{2} \Big]\label{s2ub2}\\&&\s+  P\Big[\max_{\{k: n(1-\delta)\le k \le n\}}\Big|\frac{M^{\otimes j}_{k}}{\sqrt{k}}-\frac{M^{\otimes j}_{n}}{\sqrt{n}}\Big|>\frac{\gamma}{2} \Big].\non
\eeqn
We prove that the first term is smaller than $\eta/4$. Similar arguments show that second term is smaller than $\eta/4$ from which \eqref{cond0.001e2} follows. For notational simplicity we assume $n\delta$ to be an integer (even if it's not then abuse the notation $n\delta$ to denote $[n\delta]$).
Observe that
\begin{align}
&P\Big[\max_{\{k: n\le k\le n(1+\delta)\}}\bigg|\frac{M^{\otimes j}_{k}}{\sqrt{k}}-\frac{M^{\otimes j}_{n}}{\sqrt{n}}\Big|> \frac{\gamma}{2} \bigg]  \non \\
&\quad \le P\Big[\max_{\{k: n\le k\le n(1+\delta)\}}\Big(\Big|\frac{M^{\otimes j}_{k}}{\sqrt{k}}-\frac{M^{\otimes j}_{k}}{\sqrt{n}}\Big|+\Big|\frac{M^{\otimes j}_{k}}{\sqrt{n}}- \frac{M^{\otimes j}_{n}}{\sqrt{n}}\Big|\Big)> \frac{\gamma}{2} \Big]\non\\
&\quad\le P\Big[\frac{|M^{\otimes j}_{n(1+\delta)}|}{\sqrt{n(1+\delta)}}[(1+\delta)^{\frac{1}{2}}-1]>\frac{\gamma}{4}\Big]+P\Big[\frac{M^{\otimes j}_{n(1+\delta)}-M^{\otimes j}_{n}}{\sqrt{n}}>\frac{\gamma}{4}\Big].\label{cond0.003}
\end{align}
For all $\delta\ge 0$, $(1+\delta)^{\frac{1}{2}}\le (1+\frac{\delta}{2})$. Consequently, for all $\delta\ge 0$,
$$
P\Big[\frac{|M^{\otimes j}_{n(1+\delta)}|}{\sqrt{n(1+\delta)}}[(1+\delta)^{\frac{1}{2}}-1]>\frac{\gamma}{4}\Big]
\le P\Big[\frac{\delta}{2} \frac{|M^{\otimes j}_{n(1+\delta)}|}{\sqrt{n(1+\delta)}}>\frac{\gamma}{4}\Big].
$$
Note that the marginal density of $F_{2}$ will be identical as of a $|N|$ where $N$ is a Normal $(0,1)$ distributed random variable and this can be derived from the joint density of $(F_{1},F_{2})$ in step 2. This implies that, as $n\to\infty$ 
\beqn
\frac{M^{\otimes j}_{n}}{\sqrt{n}}\stackrel{d}{\to} \sigma_{j}|N|,\label{erdosKac}
\eeqn 
which is also celebrated as the Erdos-Kac theorem (see Theorem I in Erd\"os and Kac \cite{ErdosKac1946}). As a consequence of \eqref{erdosKac}, for given $\gamma>0,\eta>0$ one can choose a small $\delta_{\eta}>0$ and a large $n_{2}^{(1)}$ such that for all $\delta\le \delta^{(1)}_{\eta}$,  
$$
P\Big[\frac{\delta}{2} \frac{|M^{\otimes j}_{n(1+\delta)}|}{\sqrt{n(1+\delta)}}>\frac{\gamma}{4}\Big]<\frac{\eta}{8}, \quad \text{for } n\ge n_{2}^{(1)}.
$$
Second term in \eqref{cond0.003} can be shown to be smaller than $\eta/8$ as well. This can be shown by first writing, with  
$A_{1}:=\{(x,y): 0\le y<\infty,-\infty<x<y\}$, 
\begin{align}
&P\Big[M^{\otimes j}_{n(1+\delta)}-M^{\otimes j}_{n}> \frac{\gamma}{4} \sqrt{n}\Big] \non \\
&\quad=\int_{A_{1}}P\big[M^{\otimes j}_{n(1+\delta)}-M^{\otimes j}_{n}>\frac{\gamma}{4}\sqrt{n}\mid (S^{\otimes j}_{n},M^{\otimes j}_{n})=(s,t)\big]f_{(S^{\otimes j}_{n},M^{\otimes j}_{n})}(s,t)dsdt\non\\
&\quad=\int_{A_{1}}P\Big[\frac{M^{\otimes j}_{n\delta}}{\sqrt{n\delta}}>\frac{t-s}{\sqrt{n\delta}}+\frac{\gamma\sqrt{n}}{4\sqrt{n\delta}}\Big]f_{(S^{\otimes j}_{n},M^{\otimes j}_{n})}(s,t)dsdt, \label{cond0.0045}
\end{align}
where \eqref{cond0.0045} follows from the independent increment property of the mean $0,$ random walk $(S_{n}^{\otimes j})_{n\ge 1}$. 
For $(s,t)\in A_{1}$, one has $t-s>0,$ and as a consequence
$$
P\Big[\frac{M^{\otimes j}_{n\delta}}{\sqrt{n\delta}}>\frac{t-s}{\sqrt{n\delta}}+\frac{\gamma}{4\sqrt{\delta}}\Big]\le P\Big[\frac{M^{\otimes j}_{n\delta}}{\sqrt{n\delta}}> \frac{\gamma}{4\sqrt{\delta}}\Big].
$$ 
By choosing $\delta^{(2)}_{\eta,\gamma}>0$ and $\delta<\delta^{(2)}_{\eta,\gamma}$ such that $\gamma/(4\sqrt{\delta})$ is sufficiently large and choosing $n_{\eta,\gamma}^{(2)}$ large, 
as a consequence of \eqref{erdosKac}, 
$$
P\Big[\frac{M^{\otimes j}_{n\delta}}{\sqrt{n\delta}}> \frac{\gamma}{4\sqrt{\delta}}\Big]<\frac{\eta}{8}, 
\quad \text{for } n\ge n_{\eta,\gamma}^{(2)}.
$$
Putting this upper bound in \eqref{cond0.0045}, the first probability term in \eqref{s2ub2} can be shown to be less than $\eta/4,$ which follows by taking $n\ge n_{0}^{(1)} \vee n_{\eta,\gamma}^{(2)}$ and $\delta\le \delta^{(1)}_{\eta}\wedge \delta^{(2)}_{\eta,\gamma}$. Hence \eqref{cond0.001e2} holds by setting $n_{2}=n_{0}^{(1)} \vee n_{\eta,\gamma}^{(2)}$.

To show \eqref{cond0.001e1}, we can decompose its upper-bound like \eqref{s2ub2} and then show that given $\gamma,\eta>0$ there exists $\delta,n_{1}$ such that
\beqn
P\Big[\max_{\{n\le k\le n(1+\delta)\}}\Big|\frac{S^{\otimes j}_{k}}{\sqrt{k}}-\frac{S^{\otimes j}_{n}}{\sqrt{n}}\Big|> \frac{\gamma}{2} \Big]<\frac{\eta}{4}\quad \text{for } n\ge n_{1}\label{s2ubb1}
\eeqn
The LHS of \eqref{s2ubb1} can be shown to be less than
\begin{align}
&P\Big[\frac{|S^{\otimes j}_{n}|}{\sqrt{n(1+\delta)}}\big[(1+\delta)^{\frac{1}{2}}-1\big]> \frac{\gamma}{4} \Big]+
P\Big[\frac{\max_{\{n\le k\le n(1+\delta)\}}\Big|S^{\otimes j}_{k}-S^{\otimes j}_{n}\Big|}{\sqrt{n}}> \frac{\gamma}{4} \Big]\non\\
&\s\s\le P\Big[\frac{\delta}{2}\frac{|S^{\otimes j}_{n}|}{\sqrt{n}}> \frac{\gamma}{4} \Big]+
P\bigg[\frac{\max_{0\le k \le n\delta}|S_{k}^{\otimes j}|}{\sqrt{n}}> \frac{\gamma}{4} \bigg]\non\\
&\s\s= P\Big[\frac{|S^{\otimes j}_{n}|}{\sqrt{n}}> \frac{\gamma}{2\delta} \Big]+
P\bigg[\frac{\max_{0\le k \le n\delta}|S_{k}^{\otimes j}|}{\sqrt{n}}> \frac{\gamma}{4} \bigg]\label{s2ubb1f}
\end{align}
\end{enumerate}
Since marginally $\frac{S_{n}^{\otimes j}}{\sqrt{n}}\stackrel{d}{\to}\sigma_{j}N$, as $n\to\infty$ (follows from step $1$), given $\gamma,\eta>0$ we can always find $\delta_{0}>0$ small enough and $n_{1}$ large enough such that for $n\ge n_{1},$ the first term in RHS of \eqref{s2ubb1f} is less than $\eta/8.$ Regarding the second term Choose $\delta\le \frac{\eta \gamma^{2}}{128 \sup_{j\in S}\sigma_{j}^{2}},$ using Kolmogorov's maximal inequality it follows that
\beqn
P\bigg[\frac{\max_{0\le k \le n\delta}|S_{k}^{\otimes j}|}{\sqrt{n}}> \frac{\gamma}{4} \bigg]\le  \frac{16\sum_{i=1}^{n\delta}\text{Var}(\log L_{i}^{j})}{(\gamma\sqrt{n})^{2}}\le\frac{16 \delta \sup_{j\in S}\sigma^{2}_{j}}{\gamma^{2}}\le \frac{\eta}{8}.\label{maximalKM}
\eeqn
 Combining both terms \eqref{s2ubb1} is proved. \eqref{cond0.001e1} and \eqref{cond0.001e2} together show \eqref{cond0.001}, and the verification of step 2 is complete.

\end{proof}

\subsubsection{Proof of Lemma \ref{lemverify2b} (c)}
\begin{proof}
Part (c) follows if  both  $S^{\otimes j}_{g^j_{t}}/\sqrt{t}$ and $M^{\otimes j}_{g^j_{t}}/\sqrt{t}$ verify \eqref{everify2a} separately as $\log G_{t}^{(1)}$ (since $M^{\otimes j}_{g^j_{t}}/\sqrt{t}$ and $\log Z^{\otimes j}_{g^j_{t}}/\sqrt{t}$ are equivalent in probability using part (a)). 
Take $t\mapsto \varepsilon(t)$ to be an increasing function satisfying $\lim_{t\to\infty}\varepsilon(t)=\infty$ and $\lim_{t\to\infty}\varepsilon(t)/t=0$. 
Observe that
\beqn
\frac{M^{\otimes j}_{g^j_{t}}}{\sqrt{t}}-\frac{M^{\otimes j}_{g^j_{t-\varepsilon(t)}}}{\sqrt{t-\varepsilon(t)}}&=&\frac{M^{\otimes j}_{g^j_{t}}-M^{\otimes j}_{[\mu_{j}t]}}{\sqrt{t}} \non\\&&\s +\frac{M^{\otimes j}_{[\mu_{j}t]}}{\sqrt{t}}  - \frac{M^{\otimes j}_{[\mu_{j}(t-\varepsilon(t))]}}{\sqrt{t-\varepsilon(t)}}+ \frac{M^{\otimes j}_{[\mu_{j}(t-\varepsilon(t))]}-M^{\otimes j}_{g^j_{t-\varepsilon(t)}}}{\sqrt{t-\varepsilon(t)}}.\label{cond0.005}  
\eeqn
For $\delta>0,$ set $A^{}_{(\delta,j)}:=\{n\in\mathbb{N}:| n/t-\mu_{j}|\le \delta \}$ and observe that 
$\lim_{t\to\infty}P[g^{j}_{t}\in A^{}_{(\delta,j)}]=1$. For the first term on the right-hand side in \eqref{cond0.005}, 
\begin{align*}
&P\Big[\frac{|M^{\otimes j}_{g^j_{t}}-M^{\otimes j}_{[\mu_{j}t]}|}{\sqrt{t}}>\gamma\Big] \\
&\quad\le P\Big[\frac{|M^{\otimes j}_{g^j_{t}}-M^{\otimes j}_{[\mu_{j}t]}|}{\sqrt{t}} >\gamma, g^j_{t}\in A^{}_{(\delta,j)}\Big]+P\Big[\frac{|M^{\otimes j}_{g^j_{t}}-M^{\otimes j}_{[\mu_{j}t]}|}{\sqrt{t}} >\gamma, g^j_{t}\in A^{c}_{(\delta,j)}\Big]\\
&\quad\le P\Big[\frac{M^{\otimes j}_{[(\mu_{j}+\delta)t]}-M^{\otimes j}_{[\mu_{j}t]}}{\sqrt{t}} >\gamma\Big]
+P\Big[\frac{M^{\otimes j}_{[\mu_{j}t]}-M^{\otimes j}_{[(\mu_{j}-\delta)t]}}{\sqrt{t}} >\gamma\Big]
+P\Big[g^j_{t}\notin A^{c}_{(\delta,j)}\Big].
\end{align*}
The arguments used to prove part (b) can be applied to prove that the first two terms on the right-hand side above goes to $0,$ by making $\delta$ arbitrarily small, and taking $t\to\infty$. 
The third term goes to $0$ as $t\to\infty$ since $g^j_t/t\stackrel{\text{a.s.}}{\to}\mu_j$ as $t\to\infty$.
Consequently, the first term on the right-hand side in \eqref{cond0.005} goes to $0$ in probability as $t\to\infty$. 
Similarly, the third term on the right-hand side in \eqref{cond0.005} goes to $0$ in probability as $t\to\infty$ since $t-\varepsilon(t)\to\infty$ as $t\to\infty$. 
The second term on the right-hand side in \eqref{cond0.005} goes to $0$ in probability as $t\to\infty$ by using an argument similar to \eqref{cond0.003} with $\varepsilon(t)/t\to 0$ as $t\to\infty$. 

Regarding verification for the $S^{\otimes j}_{g^j_{t}}/\sqrt{t},$ observe that
\begin{align}
\frac{S^{\otimes j}_{g^j_{t}}}{\sqrt{t}}-\frac{S^{\otimes j}_{g^j_{t-\varepsilon(t)}}}{\sqrt{t-\varepsilon(t)}}=\frac{S^{\otimes j}_{g^j_{t-\varepsilon(t)}}}{\sqrt{t-\varepsilon(t)}}\bigg[\Big(1-\frac{\varepsilon(t)}{t}\Big)^{\frac{1}{2}}-1\bigg]+ \frac{\sum_{i=g^{j}_{t-\varepsilon(t)}+1}^{g^{j}_{t}}\log L_{i}^{j}}{\sqrt{t}}.\label{cond0.005p2}
\end{align}
Without loss of generality assume that $t$ is large enough so that $\varepsilon(t)\le t,$ for some $t\ge t_{0}.$ Observe that  $(1-\delta_{1})^{\frac{1}{2}}\le 1-\frac{\delta_{1}}{2}$ for all $0\le \delta_{1}\le 1.$ Hence, the upper bound can be expressed as
\begin{align}
\Big|\frac{S^{\otimes j}_{g^j_{t}}}{\sqrt{t}}-\frac{S^{\otimes j}_{g^j_{t-\varepsilon(t)}}}{\sqrt{t-\varepsilon(t)}}\Big|\le \frac{\varepsilon(t)}{2t} \frac{|S^{\otimes j}_{g^j_{t-\varepsilon(t)}}|}{\sqrt{t-\varepsilon(t)}}+ \sqrt{\frac{\varepsilon(t)}{t}}\frac{|\sum_{i=g^{j}_{t-\varepsilon(t)}+1}^{g^{j}_{t}}\log L_{i}^{j}|}{\sqrt{\varepsilon(t)}},\non
\end{align}
for $t\ge t_{0}$. Since $\frac{\varepsilon(t)}{t}\to 0$ as $t\to\infty$ and both quantities $ \frac{|S^{\otimes j}_{g^j_{t-\varepsilon(t)}}|}{\sqrt{t-\varepsilon(t)}}, \,\, \frac{|\sum_{i=g^{j}_{t-\varepsilon(t)}+1}^{g^{j}_{t}}\log L_{i}^{j}|}{\sqrt{\varepsilon(t)}}$ are $O_{P}(1)$ which follows from part (b) as $(t-\varepsilon(t))$ is large. Hence by Slutsky's theorem the LHS of \eqref{cond0.005p2} $\stackrel{P}{\to} 0$ as $t\to\infty,$ and this concludes the proof of part (c).
\end{proof}

\end{proof}

\section{Proof of supplementary results of Theorem 3}
\subsection{Proof of \textbf{Lemma 10.1}}

\begin{proof}
Denote $\widehat{a}(\cdot)=a(\cdot)-E_{\pi}a(\cdot).$ Conditioned on $\widetilde{A}_{1}^{j},$ observe that
\beqn
\int_{\tau_{0}^{j}}^{\tau_{1}^{j}}\widehat{a}(Y_{s})ds =\sum_{i\in A_{1}^{j}}^{}\widehat{a}(J_{i})T_{i+1}\stackrel{d}{=}\sum_{(i_{k},i_{k+1})\in \widetilde{A}_{1}^{j}} \widehat{a}(i_{k})\widetilde{F}_{i_{k}i_{k+1}}.\non
\eeqn
 Since for $(i_{k},i_{k+1})\in \widetilde{A}_{1}^{j}\cap \mathbb{S}_{\alpha}^{-},$ the  quantity $\widehat{a}(i_{k})<0,$ hence conditioned on $\widetilde{A}_{1}^{j},$ we have
\beqn
\int_{\tau_{0}^{j}}^{\tau_{1}^{j}}\widehat{a}(Y_{s})ds &\stackrel{d}{=}&\sum_{(i_{k},i_{k+1})\in \widetilde{A}_{1}^{j}} \widehat{a}(i_{k})\widetilde{F}_{i_{k}i_{k+1}}\non\\ &=&\sum_{(i_{k},i_{k+1})\in \widetilde{A}_{1}^{j}\cap \mathbb{S}_{\alpha}^{+}}\widehat{a}(i_{k})\widetilde{F}_{i_{k}i_{k+1}} -\sum_{(i_{k},i_{k+1})\in \widetilde{A}_{1}^{j}\cap \mathbb{S}_{\alpha}^{-}}|\widehat{a}(i_{k})|\widetilde{F}_{i_{k}i_{k+1}}\non\\&&\s\s +\sum_{(i_{k},i_{k+1})\in \widetilde{A}_{1}^{j}\cap \mathbb{S}^{^c}_{\alpha}}\widehat{a}(i_{k})\widetilde{F}_{i_{k}i_{k+1}}.\non
\eeqn
 For any event $D$
\begin{align}
&\s P\Big[\int_{\tau_{0}^{j}}^{\tau_{1}^{j}}\widehat{a}(Y_{s})ds\in D\Big]=  EP\Big[\sum_{(i_{k},i_{k+1})\in \widetilde{A}_{1}^{j}} \widehat{a}(i_{k})\widetilde{F}_{i_{k}i_{k+1}}\in D\mid \widetilde{A}_{1}^{j}\Big]\non\\
&\s=EP\bigg[\Big(\sum_{(i_{k},i_{k+1})\in \widetilde{A}_{1}^{j}\cap S_{\alpha}^{+}}\widehat{a}(i_{k})\widetilde{F}_{i_{k}i_{k+1}}+\sum_{(j_{k},j_{k+1})\in \widetilde{A}_{1}^{j}\cap S^{^{c,+}}_{\alpha}}\widehat{a}(j_{k})\widetilde{F}_{j_{k}j_{k+1}}\Big)\non\\&\s\s-\Big(\sum_{(i_{k},i_{k+1})\in \widetilde{A}_{1}^{j}\cap S_{\alpha}^{-}}|\widehat{a}(i_{k})|\widetilde{F}_{i_{k}i_{k+1}}+\sum_{(j_{k},j_{k+1})\in \widetilde{A}_{1}^{j}\cap S^{^{c,-}}_{\alpha}}|\widehat{a}(j_{k})|\widetilde{F}_{j_{k}j_{k+1}}\Big)\in D| \widetilde{A}_{1}^{j}\bigg]\label{e1RegV}\\
&\s=:EP\Big[\widetilde{X}^{(1)}_{1} - \widetilde{X}^{(2)}_{1}\in D|\widetilde{A}_{1}^{j}\Big]\non
\end{align}
where
\beqn
\widetilde{X}^{(1)}_{1}&:=& \sum_{(i_{k},i_{k+1})\in \widetilde{A}_{1}^{j}\cap S_{\alpha}^{+}}\widehat{a}(i_{k})\widetilde{F}_{i_{k}i_{k+1}}+\sum_{(j_{k},j_{k+1})\in \widetilde{A}_{1}^{j}\cap S^{^{c,+}}_{\alpha}}\widehat{a}(j_{k})\widetilde{F}_{j_{k}j_{k+1}},\non\\
\widetilde{X}^{(2)}_{1}&:=&\sum_{(i_{k},i_{k+1})\in \widetilde{A}_{1}^{j}\cap S_{\alpha}^{-}}|\widehat{a}(i_{k})|\widetilde{F}_{i_{k}i_{k+1}}+\sum_{(j_{k},j_{k+1})\in \widetilde{A}_{1}^{j}\cap S^{^{c,-}}_{\alpha}}|\widehat{a}(j_{k})|\widetilde{F}_{j_{k}j_{k+1}}.\non
\eeqn
We use the following Lemma \ref{regV21} few times in this proof.

\begin{lemma}\label{regV21}
For any arbitrary class of non-negative independent regularly varying random variables $\{\widetilde{X}_{i}:i=1,\ldots,n\}$ with index $\alpha\ge 0,$  one has the following asymptotic results as $x\to\infty$
\begin{enumerate}[(i)]
\item $P[\widetilde{X}_{1}-\widetilde{X}_{2}>x]\sim P[\widetilde{X}_{1}>x],\s$ and $\s P[\widetilde{X}_{1}-\widetilde{X}_{2}<-x]\sim P[\widetilde{X}_{2}>x].$
\item $P[\sum_{i=1}^{n}\widetilde{X}_{i}>x]\sim \sum_{i=1}^{n}P[\widetilde{X}_{i}>x].$ 
\item If $P[\widetilde{X}_{2}>x]\sim o(P[\widetilde{X}_{1}>x]),$ then $P[\widetilde{X}_{1}+\widetilde{X}_{2}>x]\sim P[\widetilde{X}_{1}>x].$
\end{enumerate}
\end{lemma}

Suppose both sets $\mathbb{S}_{\alpha}^{+}, \mathbb{S}_{\alpha}^{-}$ are non-empty. Choose $\widetilde{A}_{1}^{j}$ such that $\widetilde{A}_{1}^{j}\cap \mathbb{S}_{\alpha}^{+}\neq \emptyset$ and $\widetilde{A}_{1}^{j}\cap \mathbb{S}_{\alpha}^{-}\neq \emptyset$ (since if one of them doesn't hold then later we will see that the corresponding one of (10.1) or (10.2) will vacuously hold true in limit, as both sides will be $0$).  

Conditioned on $\widetilde{A}^{j}_{1}$ observe that $\{\widetilde{F}_{i_{k}i_{k+1}}: (i_{k},i_{k+1})\in \widetilde{A}_{1}^{j}\cap \mathbb{S}_{\alpha}^{+}\}$ are all independent and regularly varying at $+\infty$ with rate $x^{-\alpha}L(x).$ As a result of Lemma \ref{regV21}(ii) the following holds
\beqn
 P\Big[\sum_{(i_{k},i_{k+1})\in \widetilde{A}_{1}^{j}\cap \mathbb{S}_{\alpha}^{+}}\widehat{a}(i_{k})\widetilde{F}_{i_{k}i_{k+1}}>x|\widetilde{A}^{j}_{1}\Big]&\sim&\sum_{(i_{k},i_{k+1})\in \widetilde{A}_{1}^{j}\cap \mathbb{S}_{\alpha}^{+}}^{}P[\widehat{a}(i_{k})\widetilde{F}_{i_{k}i_{k+1}}>x]\non\\
&\sim& \bigg(\sum_{(i_{k},i_{k+1})\in \widetilde{A}_{1}^{j}\cap S_{\alpha}^{+}}^{} c_{i_{k},i_{k+1}}\widehat{a}^{\alpha}(i_{k})\bigg) x^{-\alpha}L(x).\label{preDCT}
\eeqn

For any $(j_{k},j_{k+1})\in\widetilde{A}^{j}_{1}\cap \mathbb{S}^{^{c,+}}_{\alpha},$ the quantity $P[\widehat{a}(j_{k})\widetilde{F}_{j_{k}j_{k+1}}>x|\widetilde{A}^{j}_{1}]=o(x^{-\alpha}L(x)),$ and as a result it follows that

 $$P[\widehat{a}(j_{k})\widetilde{F}_{j_{k}j_{k+1}}>x|\widetilde{A}^{j}_{1}]=o\Big(P\Big[\sum_{(i_{k},i_{k+1})\in \widetilde{A}_{1}^{j}\cap \mathbb{S}_{\alpha}^{+}}\widehat{a}(i_{k})\widetilde{F}_{i_{k}i_{k+1}}>x|\widetilde{A}^{j}_{1}\Big]\Big),$$ 

since for any $(j_{k},j_{k+1})\in \mathbb{S}^{^{c,+}}_{\alpha}$ the sojourn time will have tails lighter than any sojourn time corresponding to the element $(i_{k},i_{k+1})\in \mathbb{S}^{+}_{\alpha}.$ Hence as a consequence of Lemma \ref{regV21}(iii) we have as $x\to\infty$
\beqn
P\big[\widetilde{X}^{(1)}_{1}>x|\widetilde{A}^{j}_{1}\big]&\sim& P\Big[\sum_{(i_{k},i_{k+1})\in \widetilde{A}_{1}^{j}\cap \mathbb{S}_{\alpha}^{+}}\widehat{a}(i_{k})\widetilde{F}_{i_{k}i_{k+1}}>x|\widetilde{A}^{j}_{1}\Big],\non\\
\s\text{\& similarly}\s \s P\big[\widetilde{X}^{(2)}_{1}>x|\widetilde{A}^{j}_{1}\big]&\sim& P\Big[\sum_{(i_{k},i_{k+1})\in \widetilde{A}_{1}^{j}\cap \mathbb{S}_{\alpha}^{-}}|\widehat{a}(i_{k})|\widetilde{F}_{i_{k}i_{k+1}}>x|\widetilde{A}^{j}_{1}\Big].\label{reg2nd}
\eeqn

Clearly conditioned on $\widetilde{A}_{1}^{j}$ both $\widetilde{X}^{(1)}_{1},\widetilde{X}^{(2)}_{1}$ are independent and regularly varying at $+\infty$ with rate $x^{-\alpha}L(x).$ By virtue of Lemma \ref{regV21}(i) it follows that almost surely 
\beqn
P\big[\widetilde{X}^{(1)}_{1} - \widetilde{X}^{(2)}_{1} >x|\widetilde{A}^{j}_{1}\big]&\sim& P\big[\widetilde{X}^{(1)}_{1}>x|\widetilde{A}^{j}_{1}\big]\non\\&\sim& P\Big[\sum_{(i_{k},i_{k+1})\in \widetilde{A}_{1}^{j}\cap \mathbb{S}_{\alpha}^{+}}\widehat{a}(i_{k})\widetilde{F}_{i_{k}i_{k+1}}>x|\widetilde{A}^{j}_{1}\Big]\non\\&\sim&  \bigg(\sum_{(i_{k},i_{k+1})\in \widetilde{A}_{1}^{j}\cap \mathbb{S}_{\alpha}^{+}}^{} c_{i_{k},i_{k+1}}\widehat{a}^{\alpha}(i_{k})\bigg) x^{-\alpha}L(x).\label{limitreg}
\eeqn

We will need an upper bound of $\frac{P\big[\widetilde{X}^{(1)}_{1} - \widetilde{X}^{(2)}_{1} >x|\widetilde{A}^{j}_{1}\big]}{x^{-\alpha}L(x)}$ with finite expectation for large $x,$ in order to use the Dominated Convergence Theorem. Observe that
\beqn
&&P\big[\widetilde{X}^{(1)}_{1} - \widetilde{X}^{(2)}_{1} >x|\widetilde{A}^{j}_{1}\big]\le P\big[\widetilde{X}^{(1)}_{1}>x|\widetilde{A}^{j}_{1}\big]\non\\
&\le& P\Big[\sum_{(i_{k},i_{k+1})\in \widetilde{A}_{1}^{j}\cap S_{\alpha}^{+}}\widehat{a}(i_{k})\widetilde{F}_{i_{k}i_{k+1}}>\frac{x}{2}|\widetilde{A}^{j}_{1}\Big]+P\Big[\sum_{(j_{k},j_{k+1})\in \widetilde{A}_{1}^{j}\cap S^{^{c,+}}_{\alpha}}\widehat{a}(j_{k})\widetilde{F}_{j_{k}j_{k+1}}>\frac{x}{2}|\widetilde{A}^{j}_{1}\Big].\non
\eeqn

Working with the first-term yields
\beqn
P\Big[\sum_{(i_{k},i_{k+1})\in \widetilde{A}_{1}^{j}\cap \mathbb{S}_{\alpha}^{+}}\widehat{a}(i_{k})\widetilde{F}_{i_{k}i_{k+1}}>x|\widetilde{A}^{j}_{1}\Big]&\le& \sum_{(i_{k},i_{k+1})\in \widetilde{A}_{1}^{j}\cap \mathbb{S}_{\alpha}^{+}}P\Big[ \widehat{a}(i_{k})\widetilde{F}_{i_{k}i_{k+1}}>\frac{x}{|\widetilde{A}_{1}^{j}\cap S_{\alpha}^{+}|}|\widetilde{A}^{j}_{1}\Big],\non
\eeqn
and after scaling by $x^{-\alpha}L(x)$ we get
\beqn
\frac{P\Big[\sum_{(i_{k},i_{k+1})\in \widetilde{A}_{1}^{j}\cap \mathbb{S}_{\alpha}^{+}}\widehat{a}(i_{k})\widetilde{F}_{i_{k}i_{k+1}}>x|\widetilde{A}^{j}_{1}\Big]}{x^{-\alpha}L(x)} &\le& \sum_{(i_{k},i_{k+1})\in \widetilde{A}_{1}^{j}\cap \mathbb{S}_{\alpha}^{+}}\frac{P\Big[ \widehat{a}(i_{k})\widetilde{F}_{i_{k}i_{k+1}}>\frac{x}{|\widetilde{A}_{1}^{j}\cap \mathbb{S}_{\alpha}^{+}|}|\widetilde{A}^{j}_{1}\Big]}{x^{-\alpha}L(x)}\non\\  &\stackrel{(a)}{\le}&\widetilde{c}_{\epsilon}\sum_{(i_{k},i_{k+1})\in \widetilde{A}_{1}^{j}\cap \mathbb{S}_{\alpha}^{+}}  c_{i_{k}i_{k+1}}\Big(1\vee \widehat{a}(i_{k})|\widetilde{A}_{1}^{j}\cap \mathbb{S}_{\alpha}^{+}|\Big)^{\alpha+\epsilon}\non
\eeqn
where (a) holds due to (4.6) and the expectation of the final upper bound in the above RHS is finite due to (4.7). 

The second term after scaled by $p(x):=x^{-\alpha}L(x)$ will give the upper bound for any $x\ge 2a^{*}|\widetilde{A}_{1}^{j}|x^{*},\& \,\,\epsilon_{1}\in (0,\alpha)$
\begin{align}
&\frac{P\Big[\sum_{(j_{k},j_{k+1})\in \widetilde{A}_{1}^{j}\cap \mathbb{S}^{^{c,+}}_{\alpha}}\widehat{a}(j_{k})\widetilde{F}_{j_{k}j_{k+1}}>\frac{x}{2}|\widetilde{A}^{j}_{1}\Big]}{x^{-\alpha}L(x)}\non\\&\s\le \sum_{(j_{k},j_{k+1})\in \widetilde{A}_{1}^{j}\cap \mathbb{S}^{^{c,+}}_{\alpha}}\frac{P\Big[\widehat{a}(j_{k})\widetilde{F}_{j_{k}j_{k+1}}>\frac{x}{2|\widetilde{A}_{1}^{j}\cap \mathbb{S}^{^{c,+}}_{\alpha}|}|\widetilde{A}^{j}_{1}\Big] }{x^{-\alpha}L(x)}\non\\&\s=\sum_{(j_{k},j_{k+1})\in \widetilde{A}_{1}^{j}\cap \mathbb{S}^{^{c,+}}_{\alpha}}\frac{\overline{F}_{j_{k}j_{k+1}}\Big(\frac{x}{2\widehat{a}(j_{k}) |\widetilde{A}_{1}^{j}\cap \mathbb{S}^{^{c,+}}_{\alpha}|}\Big)}{p\Big(\frac{x}{2\widehat{a}(j_{k}) |\widetilde{A}_{1}^{j}\cap \mathbb{S}^{^{c,+}}_{\alpha}|}\Big)}\frac{p\Big(\frac{x}{2\widehat{a}(j_{k}) |\widetilde{A}_{1}^{j}\cap \mathbb{S}^{^{c,+}}_{\alpha}|}\Big)}{p(x)}\non\\&\s\stackrel{(a)}{\le} \sum_{(j_{k},j_{k+1})\in \widetilde{A}_{1}^{j}\cap \mathbb{S}^{^{c,+}}_{\alpha}}\frac{p\Big(\frac{x}{2\widehat{a}(j_{k}) |\widetilde{A}_{1}^{j}\cap \mathbb{S}^{^{c,+}}_{\alpha}|}\Big)}{p(x)}\non\\
&\s\stackrel{(b)}{\le} \sum_{(j_{k},j_{k+1})\in \widetilde{A}_{1}^{j}\cap \mathbb{S}^{^{c,+}}_{\alpha}} (1+\epsilon_{1})\Big(2\widehat{a}(j_{k}) |\widetilde{A}_{1}^{j}\cap \mathbb{S}^{^{c,+}}_{\alpha}|\Big)^{\alpha -\epsilon_{1}}\non\\
&\s\stackrel{(c)}{\le} (1+\epsilon_{1}) (4a^{*})^{\alpha -\epsilon_{1}} \sum_{(j_{k},j_{k+1})\in \widetilde{A}_{1}^{j}\cap \mathbb{S}^{^{c,+}}_{\alpha}} |\widetilde{A}_{1}^{j}\cap \mathbb{S}^{^{c,+}}_{\alpha}|^{\alpha}\non\\
&\s\stackrel{(d)}{=} (1+\epsilon_{1}) (4a^{*})^{\alpha -\epsilon_{1}} |\widetilde{A}_{1}^{j}\cap \mathbb{S}^{c}_{\alpha}|^{1+\alpha}\non
\end{align}
where (a) follows by using the first condition of (4.8) with $x\ge 2a^{*}|\widetilde{A}_{1}^{j}|x^{*},$ (b) follows by using the Potter's bound using any arbitrary $\epsilon_{1}\in (0,\alpha).$ The inequality at (c) holds by taking supremum of $\widehat{a}(\cdot)$ over the summand, and taking it outside of the sum and having the upper bound by ignoring $\epsilon_{1};$ and finally, the result follows by observing that the final upper bound in the RHS of (d) has finite expectation due to the third condition of (4.8).

Hence from \eqref{limitreg} the assertion holds for the right tail (10.1)  as a consequence of Dominated Convergence Theorem, i.e, as $x\to\infty,$
\beqn
\frac{P\big[\widetilde{X}^{(1)}_{1} - \widetilde{X}^{(2)}_{1} >x\big]}{x^{-\alpha}L(x)}&=&\frac{EP\big[\widetilde{X}^{(1)}_{1} - \widetilde{X}^{(2)}_{1} >x|\widetilde{A}^{j}_{1}\big]}{x^{-\alpha}L(x)}\non\\&\to& E\bigg[\sum_{(i_{k},i_{k+1})\in \widetilde{A}_{1}^{j}\cap \mathbb{S}_{\alpha}^{+}}^{} c_{i_{k},i_{k+1}}\widehat{a}^{\alpha}(i_{k})\bigg]\non\\
&=&\widetilde{\alpha}_{j}^{(+)}=\frac{1+\beta_{j}}{2}\sigma^{\alpha}_{(j,\alpha)}.\non
\eeqn
Similarly (10.2) holds working with the $\{\widetilde{X}^{(1)}_{1} - \widetilde{X}^{(2)}_{1}<-x\}$ type of events.

If $\widetilde{A}_{1}^{j}\cap \mathbb{S}_{\alpha}^{+}$ is empty then (10.1) holds vacuously true as both sides are $0$ in limit, and the same holds for (10.2) corresponding to the case $\widetilde{A}_{1}^{j}\cap \mathbb{S}_{\alpha}^{-}=\emptyset$ as well.

Hence the assertion is proved for the case when both $\mathbb{S}_{\alpha}^{+}, \mathbb{S}_{\alpha}^{-}$ non-empty. If one of $\mathbb{S}_{\alpha}^{+}, \mathbb{S}_{\alpha}^{-}$ is empty, then the RHS of one of (10.1) and (10.2) will be trivially $0,$ that can be easily verified after scaling the probability in the corresponding LHS by $x^{-\alpha}L(x)$ (using similar arguments as above).
\end{proof}

\subsubsection{Proof of \textbf{Lemma 10.2}}
\begin{proof}
Observe that $$\widetilde{G}^{*(1)}_{t}=\frac{\sum_{i=1}^{g_{t}^{j}}Y_{i}^{**}}{t^{\frac{1}{\alpha}}L_{0}(t)}+\frac{\big[\tau_{0}^{j}+(t-\tau_{g_{t}^{j}}^{j})\big]E_{\pi}a(\cdot)}{t^{\frac{1}{\alpha}}L_{0}(t)},$$
where $Y_{i}^{**}=-\int_{\tau_{i-1}^{j}}^{\tau_{i}^{j}}\Big(a(Y_{s})-E_{\pi}a(\cdot)\Big)ds.$ Since the quantity $\tau_{0}^{j}+(t-\tau_{g_{t}^{j}}^{j})$ is $O_{P}(1),$ the second term of $\widetilde{G}^{*(1)}_{t}$ is $o_{P}(1).$ Denoting $t^{\frac{1}{\alpha}}L_{0}(t),(g_{t}^{j})^{1/\alpha}L_{0}(g_{t}^{j})$ by $c_{t},c_{g_{t}^{j}}$ respectively,  it is enough to show that as $t\to\infty,$
\beqn
\frac{\sum_{i=1}^{g_{t}^{j}}Y_{i}^{**}}{c_{t}}-\frac{\sum_{i=1}^{g_{t-\varepsilon(t)}^{j}}Y_{i}^{**}}{c_{t-\varepsilon(t)}} \stackrel{P}{\to}0 \label{tplem3a}
\eeqn
for some increasing $t\to\varepsilon(t)$, such that $\varepsilon(t)\to\infty$ and $\frac{\varepsilon(t)}{t}\to 0$ as $t\to\infty.$

Observe that 
\begin{align}
&\frac{\sum_{i=1}^{g_{t}^{j}}Y_{i}^{**}}{c_{t}}-\frac{\sum_{i=1}^{g_{t-\varepsilon(t)}^{j}}Y_{i}^{**}}{c_{t-\varepsilon(t)}}\non\\ &\s =\frac{\sum_{i=g_{t-\varepsilon(t)}^{j}+1}^{g_{t}^{j}}Y_{i}^{**}}{c_{t}}+\sum_{i=1}^{g_{t-\varepsilon(t)}^{j}}Y^{**}_{i}\Big[\frac{1}{c_{t}}-\frac{1}{c_{t-\varepsilon(t)}}\Big]\non\\
&\s =\frac{c_{\varepsilon(t)}}{c_{t}}\frac{c_{g_{t}^{j}-g_{t-\varepsilon(t)}^{j}}}{c_{\varepsilon(t)}}\frac{\sum_{i=g_{t-\varepsilon(t)}^{j}+1}^{g_{t}^{j}}Y_{i}^{**}}{c_{g_{t}^{j}-g_{t-\varepsilon(t)}^{j}}}+\frac{c_{g_{t-\varepsilon(t)}^{j}}}{c_{t-\varepsilon(t)}}\frac{\sum_{i=1}^{g_{t-\varepsilon(t)}^{j}}Y^{**}_{i}}{c_{g_{t-\varepsilon(t)}^{j}}}\Big[\frac{c_{t-\varepsilon(t)}}{c_{t}}-1\Big].\label{3ver1}
\end{align}
Using ideas of Lemma 12.3(a) it follows that both $\frac{c_{g_{t}^{j}-g_{t-\varepsilon(t)}^{j}}}{c_{\varepsilon(t)}}$ and $\frac{c_{g_{t-\varepsilon(t)}^{j}}}{c_{t-\varepsilon(t)}}$ converges almost surely to $(E|\mathfrak{I}^{j}_{1}|)^{-1/\alpha},$ as $t\to\infty.$ Since $\{Y_{i}^{**}:i\ge 1\}$ are a sequence of regularly varying random variables with mean $0,$ and hence a version of stable central limit theorem can be obtained in line of $c^{-1}_{n}\sum_{i=1}^{n}Y_{i}^{**}\stackrel{d}{\to}\sigma_{(j,\alpha)}\mathcal{S}_{\alpha}(1,\beta_{j},0)$ for $n\to\infty.$ Hence by Theorem 3.2 of \cite{gut2009stopped}, using $g_{t-\varepsilon(t)}^{j}$ and $g_{t}^{j}-g_{t-\varepsilon(t)}^{j}$ as the stopped random times it follows that as $t\to\infty$

\beqn
\frac{\sum_{i=1}^{g_{t-\varepsilon(t)}^{j}}Y^{**}_{i}}{c_{g_{t-\varepsilon(t)}^{j}}}&\stackrel{d}{\to}&\sigma_{(j,\alpha)}\mathcal{S}_{\alpha}(1,\beta_{j},0),\non\\
\frac{\sum_{i=g_{t-\varepsilon(t)}^{j}+1}^{g_{t}^{j}}Y_{i}^{**}}{c_{g_{t}^{j}-g_{t-\varepsilon(t)}^{j}}}&\stackrel{d}{=}&\frac{\sum_{i=1}^{g_{t}^{j}-g_{t-\varepsilon(t)}^{j}}Y_{i}^{**}}{c_{g_{t}^{j}-g_{t-\varepsilon(t)}^{j}}}\stackrel{d}{\to}\sigma_{(j,\alpha)}\mathcal{S}_{\alpha}(1,\beta_{j},0).\non
\eeqn

Hence by Slutsky's theorem our assertion will follow in \eqref{3ver1} if we prove that both $\frac{c_{\varepsilon(t)}}{c_{t}},\Big[\frac{c_{t-\varepsilon(t)}}{c_{t}}-1\Big]$ go to $0$ as $t\to\infty.$ 

Observe that $\frac{c_{\varepsilon(t)}}{c_{t}}=\big(\frac{\varepsilon(t)}{t}\big)^{1/\alpha} \frac{L_{0}(\varepsilon(t))}{L_{0}(t)},$ and using Lemma12.3(b) (Potter's bound) it follows that for any $0<\epsilon<\frac{1}{\alpha},$ there exists $t_{0,\epsilon}$ such that for all $t\ge t_{0,\epsilon},$
$$\frac{1}{1+\epsilon}\Big(\frac{\varepsilon(t)}{t}\Big)^{\frac{1}{\alpha} +\epsilon} \le \frac{c_{\varepsilon(t)}}{c_{t}}\le \frac{1}{1-\epsilon}\Big(\frac{\varepsilon(t)}{t}\Big)^{\frac{1}{\alpha} -\epsilon},$$ 
and similarly using $\frac{c_{t-\varepsilon(t)}}{c_{t}}=\big(1-\frac{\varepsilon(t)}{t}\big)^{\frac{1}{\alpha}}\frac{L_{0}(t-\varepsilon(t))}{L_{0}(t)}$ one can get the following
$$(1-\epsilon)\big(1-\frac{\varepsilon(t)}{t}\big)^{\frac{1}{\alpha}+\epsilon}-1\le \frac{c_{t-\varepsilon(t)}}{c_{t}}-1 \le (1+\epsilon)\big(1-\frac{\varepsilon(t)}{t}\big)^{\frac{1}{\alpha}-\epsilon} -1,$$
for any $0<\epsilon<\frac{1}{\alpha},$ and $t\ge t_{0,\epsilon}.$
Now taking  $t\to\infty, \frac{\varepsilon(t)}{t}\to 0,$ and taking $\epsilon$ arbitrarily small 
\beqn
\Big(\frac{c_{\varepsilon(t)}}{c_{t}},\Big[\frac{c_{t-\varepsilon(t)}}{c_{t}}-1\Big]\Big)\to(0,0)\label{convc_t}
\eeqn
 is proved as $t\to\infty,$ showing the assertion.
\end{proof}

\subsection{Proof of \textbf{Lemma 10.3 (a) \& (b)}}
\begin{proof}
The proof follows similar ideas (and notations) used in proving Lemma \ref{lemverify2b} with some differences as this is a case when variance of the increments (that is $\text{Var}(\log L_{i}^{j})$) does not exist. 

\vspace{0.5 cm}

Part (a) will follow if we show that $\frac{\max_{1\le k\le n}\log|Q_{k}^{j}|}{c_{n}}\stackrel{a.s}{\to}0,$ as $n\to\infty.$ Condition (4.14) suggests that $E\big[(\log|Q_{1}^{j}|)^{\alpha+\eps}\big]=E\big[(\log|Q_{n}^{j}|)^{\alpha+\eps}\big]<\infty,$ it follows that 
\begin{align}
&\sum_{n=1}^{\infty}P[(\log|Q_{n}^{j}|)^{\alpha+\eps} > n\eps_{1}^{\alpha+\eps}]<\infty,\s \forall \eps_{1}>0,\non\\
&\Longleftrightarrow  P\big[\log|Q_{n}^{j}|> \eps_{1} n^{\frac{1}{\alpha+\eps}} \,\,\,\text{infinitely often}\big]=0,\s \forall \eps_{1}>0.
\end{align}

Hence for the set $\big\{\frac{\log|Q_{k}^{j}|}{k^{\frac{1}{\alpha+\eps}}}, 1\le k \le\infty\big\}$ only finitely many exceed $\eps_{1}>0,$ for any arbitrarily small $\eps_{1}>0.$ Hence we have $\frac{\max_{1\le k\le n}\log|Q_{k}^{j}|}{n^{\frac{1}{\alpha+\eps}}}\stackrel{a.s}{\to}0$ as $n\to\infty$. Since $\eps>0,$ and $$\frac{\max_{1\le k\le n}\log|Q_{k}^{j}|}{c_{n}}=\frac{\max_{1\le k\le n}\log|Q_{k}^{j}|}{n^{\frac{1}{\alpha+\eps}}} \frac{1}{n^{\frac{1}{\alpha}-\frac{1}{\alpha+\eps}}L_{0}(n)},$$
using $n^{\frac{1}{\alpha}-\frac{1}{\alpha+\eps}}L_{0}(n)=n^{\Big[(\frac{1}{\alpha}-\frac{1}{\alpha+\epsilon})+\frac{\log L_{0}(n)}{\log n}\Big]},$ and Lemma 12.3(c) suggests that $\frac{\log L_{0}(n)}{\log n}\stackrel{P}{\to} 0.$ Hence $\frac{\max_{1\le k\le n}\log|Q_{k}^{j}|}{c_{n}}\stackrel{a.s}{\to}0,$ as $n\to\infty.$ Using argument similar to \eqref{BClem_a} the assertion of part (a) follows.

\subsubsection{Proof of part (b)}

Since $\bigg(\frac{S_{g_{t}^{j}}^{\otimes j}}{c_{t}}, \frac{ M_{g_{t}^{j}}^{\otimes j}}{c_{t}} \bigg)=\bigg(\frac{S_{g_{t}^{j}}^{\otimes j}}{c_{g_{t}^{j}}}, \frac{ M_{g_{t}^{j}}^{\otimes j}}{c_{g_{t}^{j}}} \bigg)\frac{c_{g_{t}^{j}}}{c_{t}},$ and under Assumption 1 and Lemma 12.3(a), $c_{g_{t}^{j}}/c_{t}\stackrel{a.s}{\to} \Big(\frac{1}{E|\mathfrak{I}^{j}_{1}|}\Big)^{\frac{1}{\alpha}},$ part (b) will follow if we prove the following two steps,
\begin{itemize}
\item \textbf{Step 1 :} $\Big(\frac{S_{n}^{\otimes j}}{c_{n}},\frac{M_{n}^{\otimes j}}{c_{n}}\Big)\stackrel{d}{\to}\sigma_{(j,\alpha)}\big(\mathcal{\bf{S}}_{\alpha,\beta_{j}}(1),\sup_{0\le u \le 1}\mathcal{\bf{S}}_{\alpha,\beta_{j}}(u)\big)$ as $n\to\infty;$
\item \textbf{Step 2 :} $\Big(\frac{S_{n}^{\otimes j}}{c_{n}},\frac{M_{n}^{\otimes j}}{c_{n}}\Big)$ satisfies the Anscombe's contiguity condition for $R_{n}$ in (2.11);
\end{itemize} 
then by replacing $n$ by $g_{t}^{j}$ (as $g_{t}^{j}\stackrel{a.s}{\to}\infty$) the result will follow.
\begin{enumerate}[(A)]
\item (\textbf{Proof of Step 1}) Let $\mathcal{\bf{S}}_{\alpha,\beta}$ denote the $\alpha$-stable Levy process $\mathcal{\bf{S}}_{\alpha,\beta}:=\big(\mathcal{S}_{\alpha,\beta}(t):t\ge 0\big)$ in $[0,\infty)$. Recall that $S_{n}^{\otimes j}:=\sum_{i=1}^{n}\log L_{i}^{j},\s M_{n}^{\otimes j}:=\max_{1\le k \le n}\Big\{\sum_{i=1}^{k}\log L_{i}^{j}\Big\}$ 
where under condition of Theorem 3(b), $E\log L_{i}^{j}=0$ and  $\{\log L_{i}^{j}: i\ge 1\}$ is a sequence of regularly varying random variables with index $\alpha>0$ along with (10.1) and (10.2). Let the random curve $y_{(n)}:=(y_{(n)}(t):t\ge 0)$ be defined in the following manner
\beqn
y_{(n)}(t):=c_{n}^{-1}S_{[nt]}^{\otimes j},\s  t\ge 0,
\eeqn
where it represents a stepwise functions having jumps at integers in $\{nt: t\ge 0\},$ which is also the continuous time analogue of the discrete partial sum process $\{S_{n}^{\otimes j}:n\in\mathbb{N}\}.$  We express a FCLT motivated by Stable CLT in  Theorem 4.5.3 in \cite{whitt2002stochastic}. Under conditions of Theorem \ref{T3}(b) one has
\beqn
\Big(c_{n}^{-1}S_{[nt]}^{\otimes j}:t\ge 0\Big)\stackrel{w}{\Longrightarrow}\sigma_{(j,\alpha)}\mathcal{\bf{S}}_{\alpha,\beta_{j}}\s \text{in}\s (D,J_{1})
\eeqn 
as $n\to\infty,$ where $(D,J_{1})$ denotes the Skorohod space under $J_{1}$ topology. Note that for any element $y\in D,$ define the functional $f:D\to(D\times D)$ such that $$f(y):=\Big(\big(y(t),\sup_{0\le s\le t}y(s)\big): t\ge 0\Big),$$ and it can be shown that, $f$ is a continuous functional that is continuous almost surely at every continuity point of the Levy process $\mathcal{\bf{S}}_{\alpha,\beta_{j}}$. When $E\log L_{1}^{j}=0,$ using continuous mapping theorem one has as $n\to\infty,$ $f(y_{(n)})\stackrel{w}{\to} f(\mathcal{\bf{S}}_{\alpha,\beta_{j}})$ in $(D\times D,J_{1})$, and as a consequence  one has
\beqn
f(y_{(n)})(1)\stackrel{d}{\to} f(\mathcal{\bf{S}}_{\alpha,\beta_{j}})(1)=\sigma_{(j,\alpha)}\big(\mathcal{\bf{S}}_{\alpha,\beta_{j}}(1),\sup_{0\le u \le 1}\mathcal{\bf{S}}_{\alpha,\beta_{j}}(u)\big)\s .\label{dcontmapp}
\eeqn
Observe that $f(y_{(n)})(1)=c_{n}^{-1}\big(S^{\otimes j}_{n}, M^{\otimes j}_{n}\big).$ Hence \eqref{dcontmapp} implies that
$$c_{n}^{-1}\Big(S^{\otimes j}_{n}, M^{\otimes j}_{n}\Big)\stackrel{d}{\to}\sigma_{(j,\alpha)}\big(\mathcal{\bf{S}}_{\alpha,\beta_{j}}(1),\sup_{0\le u \le 1}\mathcal{\bf{S}}_{\alpha,\beta_{j}}(u)\big),\s \text{as}\s n\to\infty,$$
which proves the assertion.

\item (\textbf{Proof of Step 2}) We proceed similar to Step $2$ of \eqref{lemverify2b}. In order to show  -- ``Given $\gamma>0,\eta>0$, there exist $\delta>0,n_{0}>0$ such that 
\beqn
P\Big[\max_{\{k: |k-n|<n\delta\}}\Big|\Big(\frac{S^{\otimes j}_{k}}{c_{k}},\frac{M^{\otimes j}_{k}}{c_{k}} \Big)-\Big(\frac{S^{\otimes j}_{n}}{c_{n}},\frac{M^{\otimes j}_{n}}{c_{n}}\Big)\Big|> \gamma \Big]<\eta\quad \text{for } n\ge n_{0}."\label{3Bcond0.001}
\eeqn 
it is enough if we show that given $\gamma>0,\eta>0$, there exist $\delta>0,n_{1},n_{2}>0$ such that 
\beqn
P\Big[\max_{\{k: |k-n|<n\delta\}}\Big|\frac{S^{\otimes j}_{k}}{c_{k}}-\frac{S^{\otimes j}_{n}}{c_{n}}\Big|> \frac{\gamma}{2} \Big]<\frac{\eta}{2}\quad \text{for } n\ge n_{1}\label{3Bcond0.001e1}\\
P\Big[\max_{\{k: |k-n|<n\delta\}}\Big|\frac{M^{\otimes j}_{k}}{c_{k}} - \frac{M^{\otimes j}_{n}}{c_{n}}\Big|> \frac{\gamma}{2} \Big]<\frac{\eta}{2}\quad \text{for } n\ge n_{2}\label{3Bcond0.001e2}
\eeqn 
by setting $n_{0}=n_{1}\vee n_{2}.$

To prove \eqref{3Bcond0.001e2} it is enough to show that there exists $\eta,\gamma>0,$ and large $n_{2}$  such that
\beqn
a_{n}:=P\Big[\max_{k: n\le k\le n(1+\delta)}\Big|\frac{M^{\otimes j}_{k}}{c_{k}} - \frac{M^{\otimes j}_{n}}{c_{n}}\Big|>\frac{\gamma}{2}\Big]<\frac{\eta}{4}\quad \text{for } n\ge n_{2},\label{3Bcond1a}
\eeqn
then using similar argument the other quantity $P\big[\max_{\{k: n(1-\delta)\le k\le n\}}\big|\cdot\big|>\frac{\gamma}{2}\big]$ can be shown to be less than $\frac{\eta}{4}$ proving  \eqref{3Bcond0.001e2}. It is important to note that $c_{n}$ is increasing in $n\ge 0.$ Observe that
\begin{align}
a_{n}&= P\Big[\max_{k: n\le k\le n(1+\delta)}\Big|\frac{M^{\otimes j}_{k}}{c_{k}} - \frac{M^{\otimes j}_{n}}{c_{n}}\Big|> \frac{\gamma}{2} \Big]\nonumber\\
&\le P\Big[\max_{k: n\le k\le n(1+\delta)}\bigg(\Big|\frac{M^{\otimes j}_{k}}{c_{k}} - \frac{M^{\otimes j}_{k}}{c_{n}}\Big|+\Big|\frac{M^{\otimes j}_{k}}{c_{n}} - \frac{M^{\otimes j}_{n}}{c_{n}}\Big|\bigg)> \frac{\gamma}{2} \Big]\nonumber\\
&\le P\Big[\max_{k: n\le k\le n(1+\delta)} \Big|\frac{M^{\otimes j}_{k}}{c_{k}} - \frac{M^{\otimes j}_{k}}{c_{n}}\Big|>\frac{\gamma}{4}\Big]+ P\Big[\frac{M^{\otimes j}_{n(1+\delta)} - M^{\otimes j}_{n}}{c_{n}}>\frac{\gamma}{4}\Big]\nonumber\\
&\le P\Big[\frac{|M^{\otimes j}_{n(1+\delta)}|}{c_{n(1+\delta)}}\Big(\frac{c_{n(1+\delta)}}{c_{n}}-1\Big)>\frac{\gamma}{4}\Big]+ P\Big[\frac{M^{\otimes j}_{n(1+\delta)} - M^{\otimes j}_{n}}{c_{n}}>\frac{\gamma}{4}\Big],\label{3Beak}
\end{align}
and denote the two probabilities in \eqref{3Beak} respectively by $a_{n}^{(1)}$ and $a_{n}^{(2)}$. Using the Potter's bound, for any $\epsilon_{1}\in (0,1-\frac{1}{\alpha}),$ there exists $n_{\epsilon_{1}}$ such that for all $n\ge n_{\epsilon_{1}},$ 
$(1-\epsilon_{1})(1+\delta)^{-\epsilon_{1}}\le\frac{L_{0}(n(1+\delta))}{L_{0}(n)}\le (1+\epsilon_{1})(1+\delta)^{\epsilon_{1}},$  and as a consequence
\beqn
\frac{c_{n(1+\delta)}}{c_{n}} =(1+\delta)^{\frac{1}{\alpha}}\frac{L_{0}(n(1+\delta))}{L_{0}(n)}\le (1+\epsilon_{1})(1+\delta)^{\frac{1}{\alpha}+\epsilon_{1}}.\label{c_nd1}
\eeqn
Hence the first term $a_{n}^{(1)}\le P\Big[\frac{|M^{\otimes j}_{n(1+\delta)}|}{c_{n(1+\delta)}}\Big((1+\epsilon_{1})(1+\delta)^{\frac{1}{\alpha}+\epsilon_{1}}-1\Big)>\frac{\gamma}{4}\Big].$ Since $\frac{1}{2}<\frac{1}{\alpha}<1,$ and the chosen $\epsilon_{1}\in (0,1-\frac{1}{\alpha}),$ hence $\epsilon_{1}+\frac{1}{\alpha}<1.$ As a consequence of that for any $\delta>0$, we have $(1+\delta)^{\frac{1}{\alpha}+\epsilon_{1}}\le 1+\frac{\delta}{\alpha}+\delta\epsilon_{1}$ and 
\beqn
(1+\epsilon_{1})(1+\delta)^{\frac{1}{\alpha}+\epsilon_{1}}\le (1+\epsilon_{1})\Big(1+\frac{\delta}{\alpha}+\delta\epsilon_{1}\Big)=: 1+g(\epsilon_{1},\delta),\label{c_nd2}
\eeqn
where $g(\epsilon_{1},\delta):=\frac{\delta}{\alpha}+(1+\delta)\epsilon_{1}+\epsilon_{1}\big(\frac{\delta}{\alpha}+\delta\epsilon_{1}\big).$ Observe that $g(\epsilon_{1},\delta)$ can be made arbitrarily small if we take both $\epsilon_{1},\delta$ small towards $0.$ It follows that for the first term $$a_{n}^{(1)}\le P\Big[\frac{|M^{\otimes j}_{n(1+\delta)}|}{c_{n(1+\delta)}}>\frac{\gamma}{4g(\epsilon_{1},\delta)}\Big].$$

Now given $\gamma>0,\eta>0,$ we can choose $\epsilon_{1},\delta$ arbitrarily small, such that $\frac{\gamma}{4g(\epsilon_{1},\delta)}$ will be large. Since $\frac{M^{\otimes j}_{n}}{c_{n}}\stackrel{d}{\to}\sigma_{(j,\alpha)}\sup_{0\le u\le 1}\mathcal{\bf{S}}_{\alpha,\beta_{j}}(u)$ as $n\to\infty,$ there exists $n(\epsilon_{1},\delta)>0$ such that for all $n\ge n(\epsilon_{1},\delta)\vee n_{\epsilon_{1}}$ 
\beqn
a_{n}^{(1)}\le P\Big[\frac{|M^{\otimes j}_{n(1+\delta)}|}{c_{n(1+\delta)}}>\frac{\gamma}{4g(\epsilon_{1},\delta)}\Big]\le \frac{\eta}{8}.\label{a_n1}
\eeqn

For the second term $a_{n}^{(2)},$ with the notation $A_{1}:=\{(x,y): 0\le y<\infty,-\infty<x<y\}$, it follows that
\begin{align}
a_{n}^{(2)}&\quad=P\Big[M^{\otimes j}_{n(1+\delta)}-M^{\otimes j}_{n}> \frac{\gamma}{4} c_{n}\Big] \non \\
&\quad=\int_{A_{1}}P\big[M^{\otimes j}_{n(1+\delta)}-M^{\otimes j}_{n}>\frac{\gamma}{4}c_{n}\mid (S^{\otimes j}_{n},M^{\otimes j}_{n})=(s,t)\big]f_{(S^{\otimes j}_{n},M^{\otimes j}_{n})}(s,t)dsdt\non\\
&\quad\stackrel{(a)}{=}\int_{A_{1}}P\Big[\frac{M^{\otimes j}_{n\delta}}{c_{n\delta}}>\frac{t-s}{c_{n\delta}}+\frac{\gamma c_{n}}{4c_{n\delta}}\Big]f_{(S^{\otimes j}_{n},M^{\otimes j}_{n})}(s,t)dsdt \non\\
&\quad\stackrel{(b)}{\le}P\Big[\frac{M^{\otimes j}_{n\delta}}{c_{n\delta}}> \frac{\gamma c_{n}}{4c_{n\delta}}\Big]\label{cond0.004}
\end{align}
where $(a)$ follows from the independent increment property of the mean $0,$ random walk $(S_{n}^{\otimes j})_{n\ge 1},$ and the inequality $(b)$ follows by observing that for any $(s,t)\in A_{1}$, one has $t-s>0.$ 

Using Potter's bound again for any $\epsilon_{2}>0,$ there exists $n_{\epsilon_{2}}>0,$ such that for all $n\ge n_{\epsilon_{2}},$ one has
$$\frac{c_{n}}{c_{n\delta}}=\Big(\frac{1}{\delta}\Big)^{\frac{1}{\alpha}}\frac{L_{0}(n)}{L_{0}(n\delta)}\ge (1-\epsilon_{2})\Big(\frac{1}{\delta}\Big)^{\frac{1}{\alpha} -\epsilon_{2}},$$
which implies that $a_{n}^{(2)}\le P\big[\frac{M^{\otimes j}_{n\delta}}{c_{n\delta}}> (1-\epsilon_{2})\frac{\gamma}{4}\big(\frac{1}{\delta}\big)^{\frac{1}{\alpha} -\epsilon_{2}}\big]$ for all $n\ge n_{\epsilon_{2}}.$ 

Select $\epsilon_{2}\in\big(0,\frac{1}{\alpha}\big)$ and fix any $\eta,\gamma>0.$ There exists $\delta^{(2)}_{\eta,\gamma}>0,$ small enough so that for $\delta<\delta^{(2)}_{\eta,\gamma},$ the quantity $\frac{\gamma}{4}\big(\frac{1}{\delta}\big)^{\frac{1}{\alpha}-\epsilon_{2}}$ is sufficiently large and choosing $n^{(2)}_{\eta,\gamma,\epsilon_{2}}$ large, so that for all $n\ge n^{(2)}_{\eta,\gamma,\epsilon_{2}},$  we will have 
\beqn
a_{n}^{(2)}\le P\Big[\frac{M^{\otimes j}_{n\delta}}{c_{n\delta}}> (1-\epsilon_{2})\frac{\gamma}{4}\Big(\frac{1}{\delta}\Big)^{\frac{1}{\alpha} -\epsilon_{2}}\Big]<\frac{\eta}{8}.
\label{a_n2}
\eeqn
Hence putting all bounds together, combining \eqref{a_n1},\eqref{a_n2} it follows that $a_{n}\le \frac{\eta}{4},$ for all $n\ge n_{2}:= \big(n(\epsilon_{1},\delta)\vee n_{\epsilon_{1}}\big)\vee\big(n_{\epsilon_{2}}\vee n^{(2)}_{\eta,\gamma,\epsilon_{2}}\big)$ proving \eqref{3Bcond1a}.

In order to show \eqref{3Bcond0.001e1} it is enough to show that there exists $\eta,\gamma>0$ and some large $n_{1}$ such that
\beqn
b_{n}:=P\Big[\max_{k: n\le k\le n(1+\delta)}\Big|\frac{S^{\otimes j}_{k}}{c_{k}} - \frac{S^{\otimes j}_{n}}{c_{n}}\Big|>\frac{\gamma}{2}\Big]<\frac{\eta}{4}\quad \text{for } n\ge n_{1}.\label{3Bcond2a}
\eeqn
Using the following inequalities,
\beqn
b_{n}&\le&P\Big[\max_{k: n\le k\le n(1+\delta)}\Big(\Big|\frac{S^{\otimes j}_{n}}{c_{k}} - \frac{S^{\otimes j}_{n}}{c_{n}}\Big|+\Big|\frac{S^{\otimes j}_{k}}{c_{k}} - \frac{S^{\otimes j}_{n}}{c_{k}}\Big|\Big)>\frac{\gamma}{2}\Big]\non\\
&\le&P\Big[\max_{k: n\le k\le n(1+\delta)}\frac{|S^{\otimes j}_{n}|}{c_{n}}\Big|\frac{c_{n}}{c_{k}}-1\Big|>\frac{\gamma}{4}\Big]+P\Big[\max_{k: n\le k\le n(1+\delta)}\frac{|S^{\otimes j}_{k} - S^{\otimes j}_{n}|}{c_{n}}>\frac{\gamma}{4}\Big]\non\\
&\le& P\Big[\frac{|S^{\otimes j}_{n}|}{c_{n}}\Big|\frac{c_{n}}{c_{n(1+\delta)}}-1\Big|>\frac{\gamma}{4}\Big]+P\Big[\frac{\max_{\{k: 0\le k\le n\delta\}}|S^{\otimes j}_{k}|}{c_{n}}>\frac{\gamma}{4}\Big]\non
\eeqn
and denote the above two terms in the RHS by $b_{n}^{(1)},b_{n}^{(2)}$ respectively.  Using \eqref{c_nd1} and \eqref{c_nd2} the first term $b_{n}^{(1)}$ can be upper bounded by $P\big[\frac{|S^{\otimes j}_{n}|}{c_{n}}>\frac{\gamma}{4g(\epsilon_{1},\delta)}\big]$ for any $\epsilon_{1}\in (0,1-\frac{1}{\alpha}).$ Since $\frac{S_{n}^{\otimes j}}{c_{n}}\stackrel{d}{\to}\sigma_{(j,\alpha)}\mathcal{\bf{S}}_{\alpha,\beta_{j}}(1),$ for a fixed $\gamma,\eta>0$ we can always find a small $\epsilon^{*}_{1}$ and $\delta^{*},$ such that  for $(\epsilon_{1},\delta)\in (0,\epsilon^{*}_{1})\times(0,\delta^{*})$, the quantity $\frac{\gamma}{4g(\epsilon_{1},\delta)}$ is large and  there exists $n_{2}(\epsilon_{1},\delta)>0$ such that for all $n\ge n_{2}(\epsilon_{1},\delta)\vee n_{\epsilon_{1}}$ 
\beqn
b_{n}^{(1)}\le P\Big[\frac{|S^{\otimes j}_{n}|}{c_{n}}>\frac{\gamma}{4g(\epsilon_{1},\delta)}\Big]\le \frac{\eta}{8}.\label{b_n1}
\eeqn

For bounding the second term $b_{n}^{(2)},$ we cannot use the same \textbf{Kolmogorov-Maximal-inequality} based technique since the random walk $(S_{n}^{\otimes j})_{n\ge 1}$ has regularly varying increment with $1<\alpha<2$, for non-existence of the  variance of its increments. Instead we use \textbf{Montgomery-Smith inequality} (see Theorem 3.2 (Page 263 of \cite{bertail2019Bernstein}) also) that we state here. 

\bt \textbf{(Montgomery-Smith's inequality)} If $\{\widetilde{X}_{i}:i\in \mathbb{N}\}$ are independent and identically distributed random variables, then for $1\le k\le n<\infty$ and all $t>0$ we have
$$P\Big(\max_{1\le k\le n}\|\sum_{i=1}^{k}\widetilde{X}_{i}\|>t\Big)\le 9 P\Big(\|\sum_{i=1}^{n}\widetilde{X}_{i}\|>t/30\Big).$$
\et
Since $\{\log L_{i}^{j}: i\ge 1\}$ are regularly varying with $1<\alpha<2,$ using Jensen's inequality there exists $\alpha'\in(1,\alpha)$ such that $\sup_{j\in S}E |\log L_{i}^{j}|^{\alpha'}<\infty$ and for any $\delta$
$$\Big(\frac{|S_{n\delta}^{\otimes j}|}{n\delta}\Big)^{\alpha'}\le \frac{\sum_{i=1}^{n\delta}|\log L_{i}^{j}|^{\alpha'}}{n\delta}\quad\Longrightarrow\quad E\Big[|S_{n\delta}^{\otimes j}|^{\alpha'}\Big]\le (n\delta)^{\alpha'}E |\log L_{i}^{j}|^{\alpha'}.$$
It follows that
\beqn
P\bigg[\frac{\max_{0\le k \le n\delta}|S_{k}^{\otimes j}|}{c_{n}}> \frac{\gamma}{4} \bigg]&\stackrel{(a)}{\le}&  9 P\bigg[|S_{n\delta}^{\otimes j}| > \frac{\gamma c_{n}}{120} \bigg]\nonumber\\
&\stackrel{(b)}{\le}& 9\times 120^{\alpha'}\frac{E\Big[|S_{n\delta}^{\otimes j}|^{\alpha'}\Big]}{\gamma^{\alpha'} c^{\alpha'}_{n}}\nonumber\\
&\stackrel{(c)}{\le}& 9\times 120^{\alpha'} \Big(\frac{n\delta}{\gamma c_{n}}\Big)^{\alpha'}E |\log L_{i}^{j}|^{\alpha'}\nonumber\\
&=& 9\times 120^{\alpha'} \Big(\frac{n^{1-\frac{1}{\alpha}}\delta}{\gamma L_{0}(n)}\Big)^{\alpha'}E |\log L_{i}^{j}|^{\alpha'},
\eeqn
where $(a), (b), (c)$ follow respectively by using Montgomery-Smith's inequality, Chebyshev and Jensen's inequality. Now given $\gamma,\eta$ we can choose $\delta_{0}$ such that \newline $\delta_{0}=\frac{\gamma}{120}\Big[\frac{\eta}{36\sup_{j\in S}E |\log L_{i}^{j}|^{\alpha'}}\Big]^{1/\alpha'}\frac{L_{0}(n)}{n^{1-\frac{1}{\alpha'}}}$, and then for any $\delta<\min\{\delta_{0},\delta^{*}\},$
$$P\bigg[\frac{\max_{0\le k \le n\delta}|S_{k}^{\otimes j}|}{c_{n}}> \frac{\gamma}{4} \bigg]\le 9\times 120^{\alpha'} \Big(\frac{n^{1-\frac{1}{\alpha}}\delta}{\gamma L_{0}(n)}\Big)^{\alpha'}E |\log L_{i}^{j}|^{\alpha'}\le  \frac{\eta}{8}.$$

Putting all bounds together, combining the lower limit for $\delta,$ it follows that $b_{n}\le \frac{\eta}{4},$ for all $n\ge n_{1}:=n_{2}(\epsilon_{1},\delta)\vee n_{\epsilon_{1}}$ proving \eqref{3Bcond2a}.
\end{enumerate}
\end{proof}

\subsection{Proof of \textbf{Lemma 10.3(c)}}

\begin{proof}

This follows if  both  $S^{\otimes j}_{g^j_{t}}/c_{t}$ and $M^{\otimes j}_{g^j_{t}}/c_{t}$ verify (6.1) separately as $H_{t}$ (since $M^{\otimes j}_{g^j_{t}}/c_{t}$ and $\log Z^{\otimes j}_{g^j_{t}}/c_{t}$ are equivalent in probability using part (a)).  For $\delta>0,$ use the same notation $A^{}_{(\delta,j)}:=\{n\in\mathbb{N}:| n/t-\mu_{j}|\le \delta \}$ as in Lemma \ref{lemverify2b}(c) and observe that $\lim_{t\to\infty}P[g^{j}_{t}\in A^{}_{(\delta,j)}]=1.$ Use the equality
\beqn
\frac{M^{\otimes j}_{g^j_{t}}}{c_{t}}-\frac{M^{\otimes j}_{g^j_{t-\varepsilon(t)}}}{c_{t-\varepsilon(t)}}&=&\frac{M^{\otimes j}_{g^j_{t}}-M^{\otimes j}_{[\mu_{j}t]}}{c_{t}} \non\\&&\s +\frac{M^{\otimes j}_{[\mu_{j}t]}}{c_{t}}  - \frac{M^{\otimes j}_{[\mu_{j}(t-\varepsilon(t))]}}{c_{t-\varepsilon(t)}}+ \frac{M^{\otimes j}_{[\mu_{j}(t-\varepsilon(t))]}-M^{\otimes j}_{g^j_{t-\varepsilon(t)}}}{c_{t-\varepsilon(t)}}.\label{3Bcond0.005}  
\eeqn
 For the first term in the right-hand side of \eqref{3Bcond0.005}, 
\begin{align*}
&P\Big[\frac{|M^{\otimes j}_{g^j_{t}}-M^{\otimes j}_{[\mu_{j}t]}|}{c_{t}}>\gamma\Big] \\
&\quad\le P\Big[\frac{|M^{\otimes j}_{g^j_{t}}-M^{\otimes j}_{[\mu_{j}t]}|}{c_{t}} >\gamma, g^j_{t}\in A^{}_{(\delta,j)}\Big]+P\Big[\frac{|M^{\otimes j}_{g^j_{t}}-M^{\otimes j}_{[\mu_{j}t]}|}{c_{t}} >\gamma, g^j_{t}\in A^{c}_{(\delta,j)}\Big]\\
&\quad\le P\Big[\frac{M^{\otimes j}_{[(\mu_{j}+\delta)t]}-M^{\otimes j}_{[\mu_{j}t]}}{c_{t}} >\gamma\Big]
+P\Big[\frac{M^{\otimes j}_{[\mu_{j}t]}-M^{\otimes j}_{[(\mu_{j}-\delta)t]}}{c_{t}} >\gamma\Big]
+P\Big[g^j_{t}\notin A^{c}_{(\delta,j)}\Big].
\end{align*}
The arguments used to prove part (b) can be applied to prove that the first two terms on the right-hand side above goes to $0,$ by making $\delta$ arbitrarily small, and taking $t\to\infty$. 
The third term goes to $0$ as $t\to\infty$ since $g^j_t/t\stackrel{\text{a.s.}}{\to}\mu_j$ as $t\to\infty$.
Consequently, the first term on the right-hand side in \eqref{3Bcond0.005} goes to $0$ in probability as $t\to\infty$. 
Similarly, the third term on the right-hand side in \eqref{3Bcond0.005} goes to $0$ in probability as $t\to\infty$ since $t-\varepsilon(t)\to\infty$ as $t\to\infty$.

Regarding verification for the $S^{\otimes j}_{g^j_{t}}/c_{t},$ observe that
\begin{align}
\frac{S^{\otimes j}_{g^j_{t}}}{c_{t}}-\frac{S^{\otimes j}_{g^j_{t-\varepsilon(t)}}}{c_{t-\varepsilon(t)}}=\frac{S^{\otimes j}_{g^j_{t-\varepsilon(t)}}}{c_{t-\varepsilon(t)}}\Big[\frac{c_{t-\varepsilon(t)}}{c_{t}} -1\Big]+ \frac{\sum_{i=g^{j}_{t-\varepsilon(t)}+1}^{g^{j}_{t}}\log L_{i}^{j}}{c_{t}}.\label{3Bcond0.005p2}
\end{align}

The upper bound can be expressed as
\begin{align}
\Big|\frac{S^{\otimes j}_{g^j_{t}}}{c_{t}}-\frac{S^{\otimes j}_{g^j_{t-\varepsilon(t)}}}{c_{t-\varepsilon(t)}}\Big|\le \Big|\frac{S^{\otimes j}_{g^j_{t-\varepsilon(t)}}}{c_{t-\varepsilon(t)}}\Big| \Big|\frac{c_{t-\varepsilon(t)}}{c_{t}} -1\Big|+\frac{c_{\varepsilon(t)}}{c_{t}}\frac{\sum_{i=g^{j}_{t-\varepsilon(t)}+1}^{g^{j}_{t}}\log L_{i}^{j}}{c_{\varepsilon(t)}},\non
\end{align}
for $t\ge t_{0}$. Since $\frac{\varepsilon(t)}{t}\to 0$ as $t\to\infty$ and both quantities $ \frac{|S^{\otimes j}_{g^j_{t-\varepsilon(t)}}|}{c_{t-\varepsilon(t)}}, \,\, \frac{|\sum_{i=g^{j}_{t-\varepsilon(t)}+1}^{g^{j}_{t}}\log L_{i}^{j}|}{c_{\varepsilon(t)}}$ are $O_{P}(1)$ which follows from part (b) as $(t-\varepsilon(t)), \varepsilon(t)$ are large. As a consequence of \eqref{convc_t} $\Big(\frac{c_{\varepsilon(t)}}{c_{t}},\Big[\frac{c_{t-\varepsilon(t)}}{c_{t}}-1\Big]\Big)\stackrel{}{\to}(0,0),$ and by using  Slutsky's theorem the LHS of \eqref{3Bcond0.005p2} $\stackrel{P}{\to} 0$ as $t\to\infty,$ and this concludes the proof of part (c).

\end{proof}

\subsection{Proof of results in \textbf{Remark 4.2}}
\begin{proof}
Under Assumption 1(a) the $S$-valued Markov chain $J$ is irreducible, aperiodic, and positive recurrent. Hence using that for any $f:S\times S\to \R,$ for any subset $A\subset S\times S,$ using Proposition 2.12.2 of \cite{resnick1992adventures} we have 
\beqn
\s\,\frac{E_{j}\big[\sum_{(i_{1},i_{2})\in A}\sum_{n=0}^{|\widetilde{A}_{1}^{j}|-1}1_{\{J_{n}=i_{1},J_{n+1}=i_{2}\}} f(J_{n},J_{n+1})\big] }{E_{j}|\widetilde{A}_{1}^{j}|}=\sum_{(i_{1},i_{2})\in A}\mu_{i_{1}}P_{i_{1}i_{2}}f(i_{1}.i_{2}),\label{regenE}\eeqn
as under stationarity $P_{\mu}\big[J_{n}=i_{1},J_{n+1}=i_{2}\big]=\mu_{i_{1}}P_{i_{1}i_{2}}.$ The right tail-index expression can be simplified under Assumptions 1 and 4. It follows from Lemma 10.1 that
\beqn
\widetilde{\alpha}_{j}^{(+)}&:=&E\Big[\sum_{(i_{l},i_{l+1})\in \widetilde{A}_{1}^{j}\cap \mathbb{S}^{+}_{\alpha}}c_{i_{l}i_{l+1}}\big[a(i_{l})-E_{\pi}a(\cdot)\big]^{\alpha}\Big]\non\\
&=& E|\widetilde{A}_{1}^{j}|\frac{E\Big[\sum_{(i_{l},i_{l+1})\in \widetilde{A}_{1}^{j}\cap \mathbb{S}^{+}_{\alpha}}c_{i_{l}i_{l+1}}\widehat{a}^{\alpha}(i_{l})\Big]}{E|\widetilde{A}_{1}^{j}|}\non\\
&\stackrel{(a)}{=}& E|A_{1}^{j}|\frac{E\sum_{(i_{l},i_{l+1})\in  \mathbb{S}^{+}_{\alpha}}\Big[\sum_{n=0}^{|\widetilde{A}_{1}^{j}|-1}1_{\{J_{n}=i_{1},J_{n+1}=i_{2}\}}c_{J_{n}J_{n+1}}\widehat{a}^{\alpha}(J_{n})\Big]}{E|\widetilde{A}_{1}^{j}|}\non\\
&=& E|A_{1}^{j}| \sum_{(i_{l},i_{l+1})\in  \mathbb{S}^{+}_{\alpha}} \mu_{i_{1}}P_{i_{1}i_{2}} c_{i_{1}i_{2}}\widehat{a}^{\alpha}(i_{1})\non
\eeqn
where (a) follows by observing $|A_{1}^{j}|=|\widetilde{A}_{1}^{j}|$ and the last equality follows by using \eqref{regenE}. The left tail-index expression and $\sigma_{(j,\alpha)}$ can be derived similarly.

Second part of the remark holds as
\beqn
E\Big[\sum_{(i_{l},i_{l+1})\in \widetilde{A}_{1}^{j}\cap \mathbb{S}_{\alpha}}\widetilde{c}_{\epsilon}(|a(i_{l})||\widetilde{A}_{1}^{j}\cap \mathbb{S}_{\alpha}|)^{\alpha+\epsilon}\Big]&\le& \widetilde{c}_{\epsilon}\sup_{x\in S}|a(x)|E[|\widetilde{A}_{1}^{j}\cap \mathbb{S}_{\alpha}|^{1+\alpha+\epsilon}]\non\\
&\le& \widetilde{c}_{\epsilon}\sup_{x\in S}|a(x)| E[|\widetilde{A}_{1}^{j}|^{1+\alpha+\epsilon}]\non\\
&=&\widetilde{c}_{\epsilon}\sup_{x\in S}|a(x)| E[|A_{1}^{j}|^{1+\alpha+\epsilon}]<\infty.\non
\eeqn
\end{proof}

\subsection{Proof of \textbf{Proposition 4.4}}
\begin{proof} Note that $\sigma_{j}^{2}=\Var\Big(\int_{\tau_{0}^{j}}^{\tau_{1}^{j}}\big(a(Y_{s}) - E_{\pi}a(\cdot)\big)ds\Big),$ and
$$\sigma_{j}^{2}=0 \Longleftrightarrow P\Big[\int_{\tau_{0}^{j}}^{\tau_{1}^{j}}\big(a(Y_{s}) - E_{\pi}a(\cdot)\big)ds=c\Big]=1$$
for some constant $c$. Taking expectation of $\int_{\tau_{0}^{j}}^{\tau_{1}^{j}}\big(a(Y_{s}) - E_{\pi}a(\cdot)\big)ds$ observe that $c=0,$ which implies that 
\beqn
P\bigg[\frac{\int_{\tau_{0}^{j}}^{\tau_{1}^{j}}a(Y_{s})ds}{\tau_{1}^{j} - \tau_{0}^{j}}=E_{\pi}a(\cdot)\bigg]=1,\label{aconstante1}
\eeqn
 where recall that $E_{\pi}a(\cdot)=\frac{E \big[\int_{\tau_{0}^{j}}^{\tau_{1}^{j}}a(Y_{s})ds\big]}{E[\tau_{1}^{j} - \tau_{0}^{j}]}.$ Denote the set $\Big\{\omega\in \Omega:\frac{\int_{\tau_{0}^{j}}^{\tau_{1}^{j}}a(Y_{s})ds}{\tau_{1}^{j} - \tau_{0}^{j}}=E_{\pi}a(\cdot)\Big\}$ by $A.$

For each $i\in S,$ and for each $\omega\in A$ with $P[A]=1,$ define $$\Tilde{T}_{i}(\omega)= \text{total time spent at state} \,\,\,\, i\,\,\, \text{in}\,\,\, [\tau_{0}^{j},\tau_{1}^{j}).$$
Note that $\int_{\tau_{0}^{j}}^{\tau_{1}^{j}}a(Y_{s})ds= \sum_{i\in S} a(i)\Tilde{T}_{i}(\omega)$ by definition and for each $i\in S$ define $\widehat{a}(i):=a(i)-E_{\pi}a(\cdot)$, and for each $\omega\in A,$ $\Tilde{T}_{i}(\omega)>0$ with probability $1$ (since under Assumption 1 sojourn time laws do not have atom at $0$). \eqref{aconstante1} can be rewritten as 
\beqn
P\Big[ \frac{\sum_{i\in S} a(i)\Tilde{T}_{i}}{\tau_{1}^{j} -\tau_{0}^{j}}=E_{\pi}a(\cdot)\Big]=1\s \Longleftrightarrow\s \forall \,\,\omega\in A,\,\,\, \sum_{i\in S}\widetilde{a}_{i}\Tilde{T}_{i}(\omega)=0.
\label{aconstante2}
\eeqn 
Since there are uncountably infinitely many $\omega\in A,$ and $S$ can be only at most countably finite. Then one can always choose a finite collection $B=\{\omega_{j}:j\in S\},$ such that the matrix $\Big\{\Tilde{T}_{i}(\omega_{j}): i\in S, \omega_j\in B\Big\}$ will be almost surely non-singular. \eqref{aconstante2} implies that
$$\text{for all}\s \omega_{j}\in B,\s \sum_{i\in S}\widehat{a}(i)\Tilde{T}_{i}(\omega_{j})=0.$$
Hence the only solution for $\{\widehat{a}(i):i\in S\}$ holds iff $\widehat{a}(i)=0 \,\,\forall i\in S,$ proving the assertion.
\end{proof}

\begin{remark}\label{countableremark}
We conjecture that Proposition 4.4 holds even if $S$ is countably infinite. This can be shown if we can prove that one can always choose a countably infinite collection $B=\{\omega_{j}:j\in S\},$ such that the matrix $\Big\{\Tilde{T}_{i}(\omega_{j}): i\in S, \omega_j\in B\Big\}$ will be almost surely non-singular. 
\end{remark}

\subsection{Proof of \textbf{Theorem 4}}
\begin{proof}
 For part (a) with $a<0,$ we use following notations
\beqn
(L^{j}_{k},Q^{j}_{k})&\stackrel{}{:=}&\Big(e^{-a(\tau_{k}^{j} - \tau_{k-1}^{j})},\int_{\tau_{k-1}^{j}}^{\tau_{k}^{j}}b^{}(Y_{s})e^{-a(\tau_{k}^{j} - s)}ds\Big), \quad k\geq 1,\non\\ 
(L^{j}_{0},Q^{j}_{0})&\stackrel{}{:=}&\Big(e^{-a\tau_{0}^{j}},\int_{0}^{\tau_{0}^{j}}b^{}(Y_{s})e^{-a(\tau_{0}^{j}-s)}ds\Big),\non
\eeqn
Recalling (8.4), denote the following
$S^{**}_{n}:=P_{n}\big(\frac{1}{L_{1}^{j}},\frac{Q_{1}^{j}}{L_{1}^{j}}\big)= \sum_{k=1}^{n}\Big(\prod_{l=1}^{k-1}\frac{1}{L_{l}^{j}}\Big)\frac{Q_{k}^{j}}{L_{k}^{j}},\,\, n\ge 1.$

We begin with for any function $f(\cdot,\cdot):\R_{\ge 0}\times \R\to\R$
\beqn
\s\s\s P\big[f(t,I_{t}^{(a,b)})\in A | Y_{0}=i\big]=\sum_{j\in S}P[Y_{t}=j | Y_{0}=i] P\big[f(t,I_{t}^{(a,b)})\in A | Y_{0}=i, Y_{t}=j\big].\label{eqa1}
\eeqn

Observe that on the event $\{Y_{0}=i,\,\, Y_{t}=j\}$ for $f(t_{1},x)=e^{at_{1}}x$ one has
\begin{align}
&e^{at}I_{t}^{(a,b)}\stackrel{d}{=}\frac{\int_{0}^{t}b^{}(Y_{s})e^{a(s-t)}ds}{e^{-at}}\non\\
&\s=\frac{G^{a,b^{}}_{j}(t-\tau_{g_{t}^{j}}^{j}) +e^{-a(t-\tau_{g_{t}^{j}}^{j})}\Big[\Big(\prod_{k=1}^{g_{t}^{j}}L_{k}^{j}\Big)Q_{0}^{j}+\sum_{k=1}^{g_{t}^{j}}\Big(\prod_{l=k+1}^{g_{t}^{j}}L_{l}^{j}\Big)Q_{k}^{j}       \Big] }{e^{-a(t-\tau_{g_{t}^{j}}^{j})}\Big(\prod_{k=1}^{g_{t}^{j}}L_{k}^{j}\Big) e^{-a\tau^{j}_{0}}} \non\\
&\s=\frac{G^{a,b^{}}_{j}(t-\tau_{g_{t}^{j}}^{j})}{e^{-at}} +\frac{Q_{0}^{j}+S^{**}_{g_{t}^{j}}}{L_{0}^{j}}\non\\
&\s\stackrel{P}{\to}\frac{Q_{0}^{j}+S^{**}_{\infty}}{L_{0}^{j}}\label{R4e1}
\end{align}
as $t\to\infty.$
Last step follows by using $a<0$ and proceeding similarly to the derivation done in the display \eqref{SREdiv}. Notice that using similar steps of Lemma 6.2 with relevant assumptions $S^{**}_{\infty}$ satisfies the unique solution of the SRE 
$$S^{**}_{\infty}\stackrel{d}{=} \frac{1}{L_{1}^{j}}S^{**}_{\infty}+ \frac{Q_{1}^{j}}{L_{1}^{j}},\s S^{**}_{\infty}\perp \Big(\frac{1}{L_{1}^{j}},\frac{Q_{1}^{j}}{L_{1}^{j}}\Big).$$ 
Using \eqref{R4e1}, proceeding similar to the way we proved Theorem 2(a) from \eqref{beg}, the assertion of part (a) follows by taking limit $t\to\infty$ in \eqref{eqa1}, denoting $S^{**}_{\infty}$ by $\widetilde{V}_{j}.$

Under the conditions of part (b) (that is when $a=0$) observe that $$I^{(0,b)}_{t}=\int_{0}^{t}b(Y_{s})ds $$ and then applying the law of large number for renewal process (under condition $E_{\pi}|b^{}(\cdot)|<\infty$) yields 
$$\frac{I^{(0,b)}_{t}}{t}=\frac{\int_{0}^{t}b(Y_{s})ds}{t}\,\,\stackrel{a.s}{\to} \,\, E_{\pi}b^{}(\cdot),$$
as $t\to\infty$ as required.

\vspace{0.3 cm}

For the rest parts, we sketch the proofs omitting the details, as they are very similar to steps of Theorem 2(a) and Theorem 3(a) and follow finally by using \eqref{eqa1} by taking $f(t,x)=x$. Observe that on the set $\{Y_{0}=i, Y_{t}=j\}$ one has $$\int_{0}^{t}b(Y_{s})ds = \sum_{i=1}^{g_{t}^{j}}\int_{\tau_{i-1}^{j}}^{\tau_{i}^{j}}b(Y_{s})ds+\int_{0}^{\tau_{0}^{j}}b(Y_{s})ds+b(j)(t-\tau_{g_{t}^{j}}),$$
where $\sum_{i=1}^{g_{t}^{j}}\int_{\tau_{i-1}^{j}}^{\tau_{i}^{j}}b(Y_{s})ds$ is the key term determining the asymptotics result while $\int_{0}^{\tau_{0}^{j}}b(Y_{s})ds+b(j)(t-\tau_{g_{t}^{j}})$ is just a $O_{p}(1)$ which goes to $0$ asymptotically after division with any exponent of $t.$ Since $\Big(\sum_{i=1}^{g_{t}^{j}}\int_{\tau_{i-1}^{j}}^{\tau_{i}^{j}}b(Y_{s})ds: t\ge 0\Big)$ can be seen as a renewal reward process with $i$-th reward being $\int_{\tau_{i-1}^{j}}^{\tau_{i}^{j}}b(Y_{s})ds,$ using results derived in \eqref{clt} of Theorem 2(a) and Theorem 3(a), one can the subsequently derive results based on conditions if $\sigma^{2}_{j,(b)}$ is finite, or $=\infty$ or $=0.$

We prove the result for (ii). Observe that if $\widehat{b}(\cdot):=b(\cdot)-E_{\pi}b(\cdot)$, then on $\{Y_{0}=i,Y_{t}=j\}$ $$I_{t}^{(0,b)}-tE_{\pi}b(\cdot)=\sum_{i=1}^{g_{t}^{j}}\int_{\tau_{i-1}^{j}}^{\tau_{i}^{j}}\widehat{b}(Y_{s})ds+\int_{0}^{\tau_{0}^{j}}\widehat{b}(Y_{s})ds+\widehat{b}(j)(t-\tau_{g_{t}^{j}}).$$
Under Assumptions 4(a),(b) with $a(\cdot)$ replaced by $b(\cdot),$ Lemma 10.1 suggests that $\int_{\tau_{0}^{j}}^{\tau_{1}^{j}}\widehat{b}(Y_{s})ds=\int_{\tau_{0}^{j}}^{\tau_{1}^{j}}(b(Y_{s})-E_{\pi}b(\cdot))ds$ is regularly varying with index $\alpha_{b},$ along with the slowly varying function $L_{1}(\cdot),$ with following right and left tail asymptotic limits
\beqn
\widetilde{\alpha}_{j,(b)}^{(+)}&=& E\Big[\sum_{(i_{l},i_{l+1})\in \widetilde{A}_{1}^{j}\cap \mathbb{S}^{+}_{\alpha_{b}}}c_{i_{l}i_{l+1}}\big[b(i_{l})-E_{\pi}b(\cdot)\big]^{\alpha_{b}}\Big],\non\\
\widetilde{\alpha}_{j,(b)}^{(-)}&=&E\Big[\sum_{(i_{l},i_{l+1})\in \widetilde{A}_{1}^{j}\cap  \mathbb{S}^{-}_{\alpha_{b}}}c_{i_{l}i_{l+1}}|b(i_{l})-E_{\pi}b(\cdot)|^{\alpha_{b}}\Big],\non\\
\widetilde{\beta}_{j,(b)}&=&\frac{\widetilde{\alpha}_{j,(b)}^{(-)} - \widetilde{\alpha}_{j,(b)}^{(+)}}{\widetilde{\alpha}_{j,(b)}^{(+)}+\widetilde{\alpha}_{j,(b)}^{(-)}}.\non
\eeqn
Only for this part denote $t^{1/\alpha_{b}}L_{1}(t), n^{1/\alpha_{b}}L_{1}(n)$ by $\tilde{c}_{t},\tilde{c}_{n}$ respectively.
For each $j\in S$ using the quantity $\widetilde{\beta}_{j,(b)}$ above, applying the stable central limit Theorem 4.5.1 of \cite{whitt2002stochastic}, as sample size $n\to\infty$ one has
\beqn
\tilde{c}^{-1}_{n}\Big(\sum_{i=1}^{n}\int_{\tau_{i-1}^{j}}^{\tau_{i}^{j}}\widehat{b}(Y_{s})ds\Big)\stackrel{d}{\to}\sigma_{(j,\alpha_{b},(b))}\mathcal{S}_{\alpha_{b}}(1,-\beta_{j,(b)},0).\label{stable3aa}
\eeqn 
Using \eqref{stable3aa} and Theorem 3.2 of \cite{gut2009stopped} for the stopped random time $g_{t}^{j}$ as $t\to\infty,$ 
$$\frac{\sum_{i=1}^{g_{t}^{j}}\int_{\tau_{i-1}^{j}}^{\tau_{i}^{j}}\widehat{b}(Y_{s})ds}{(g_{t}^{j})^{\frac{1}{\alpha_{b}}}L_{1}(g_{t}^{j})}\stackrel{d}{\to}\sigma_{(j,\alpha_{b},(b))}\mathcal{S}_{\alpha_{b}}(1,-\beta_{j,(b)},0).$$
Furthermore observe $\tau_{0}^{j}+(t-\tau_{g_{t}^{j}}^{j})$ is $O_{p}(1)$ and  $\lim_{t\to\infty}\frac{\tilde{c}_{g_{t}^{j}}}{\tilde{c}_{t}}=\Big(\frac{1}{E|\mathfrak{I}^{j}_{1}|}\Big)^{\frac{1}{\alpha_{b}}}$ (see Lemma 12.3(a)). Hence as $t\to\infty$ we have  
\beqn
\mathcal{L}\Big(\tilde{c}_{t}^{-1}\big[I_{t}^{(0,b)}-tE_{\pi}b(\cdot)\big] | Y_{0}=i, Y_{t}=j\Big)\stackrel{w}{\to} \mathcal{L}\bigg(\big(E|\mathfrak{I}^{j}_{1}|\big)^{-\frac{1}{\alpha_{b}}}\sigma_{(j,\alpha_{b},(b))}\mathcal{S}_{\alpha_{b}}(1,-\beta_{j,(b)},0)\bigg).\nonumber
\eeqn
Putting this estimate in \eqref{eqa1} using $f(t,x)=x$, and taking $t\to\infty,$ the second part of Theorem 4(b) will follow. Third part is trivial and the proof is omitted.
\end{proof}

\section{Proof of Supplementary parts of Corollary 4.5}

\subsubsection{Proof of \textbf{Lemma 11.1(a)}}
\begin{proof}
Proceeding similar to the proof of Theorem 2(a) (and using same notation) observe that,
\beqn
|I_{t}^{(ma,b)}|^{\frac{1}{m}}&=& \big|\int_{0}^{t}b(Y_{s})e^{-\int_{s}^{t}ma(Y_{r})dr}ds\big|^{\frac{1}{m}}\non\\
&=&\Phi_{t}^{(a)}|\widetilde{I}_{t}^{(ma,b)}|^{\frac{1}{m}}.\non
\eeqn
Suppose $\{(L_{i,m}^{j},Q_{i,m}^{j}): i\ge 1\},$ $(L_{0,m}^{j},Q_{0,m}^{j})$  are defined similarly as $\{(L_{i}^{j},Q_{i}^{j}): i\ge 1\}$ and $(L_{0}^{j},Q_{0}^{j})$ by replacing $a(\cdot)$ by $ma(\cdot)$ respectively. $\{S_{n,m}^{**}:n\ge 1\}$ are defined as $$S^{**}_{n,m}:=P_{n}\Big(\frac{1}{L_{1,m}^{j}},\frac{Q_{1,m}^{j}}{L_{1,m}^{j}}\Big)=\sum_{i=1}^{n}\frac{Q_{i,m}^{j}}{L_{i,m}^{j}}\prod_{k=1}^{i-1}\Big(\frac{1}{L_{k,m}^{j}}\Big).$$

Observe that on $\{Y_{0}=i,Y_{t}=j\},$
\beqn
\widetilde{I}^{(ma,b)}_{t}=\bigg[\frac{Q_{0,m}^{j}+S^{**}_{g_{t}^{j},m}}{L_{0,m}^{j}}+\frac{G_{j}^{(ma,b)}(t-\tau_{g_{t}^{j}})}{\Phi_{t}^{(ma)}}\bigg].\non
\eeqn
Using same notations as in the proof of Theorem 2(a), let $\widetilde{G}^{(3,m)}_{t}$ be defined as 
\beqn
&& \widetilde{G}^{(3,m)}_{t}:=\bigg[\frac{Q_{0,m}^{j}+S^{**}_{g_{t}^{j},m}}{L_{0,m}^{j}}+\frac{G_{j}^{(ma,b)}(t-\tau_{g_{t}^{j}})}{\Phi_{t}^{(ma)}}\bigg]^{\frac{1}{m\sqrt{t}}}.\non
\eeqn

Observe that 
\beqn
\mathcal{L}\Big(|\widetilde{I}^{(ma,b)}_{t}|^{\frac{1}{m\sqrt{t}}}\mid Y_{0}=i,Y_{t}=j\Big)=\mathcal{L}\Big(\widetilde{G}^{(3,m)}_{t}\mid Y_{0}=i,Y_{t}=j\Big)
\eeqn
and furthermore
\beqn
&& \mathcal{L}\bigg(\bigg(\frac{[\Phi^{(a)}_{t}]^{\frac{1}{\sqrt{t}}}}{e^{-\sqrt{t}E_{\pi}a(\cdot)}}, \frac{|I^{(ma,b)}_{t}|^{\frac{1}{m\sqrt{t}}}}{e^{-\sqrt{t}E_{\pi}a(\cdot)}}\bigg)\mid Y_{0}=i,Y_{t}=j\bigg)\non\\ &&\,\,\,\,\,\,\,\,\,\,=\mathcal{L}\Big(\big(G^{(1)}_{t}G^{(2)}_{t},G^{(1)}_{t}G^{(2)}_{t}\widetilde{G}^{(3,m)}_{t}\big)\mid Y_{0}=i,Y_{t}=j\Big).\non
\eeqn
Since $S^{**}_{g_{t}^{j},m}\stackrel{d}{\to}S_{m}^{**}$    where $S_{m}^{**}$ satisfies
$$S_{m}^{**}\stackrel{d}{=}\frac{1}{L_{1,m}^{j}}S_{m}^{**}+\frac{Q_{1,m}^{j}}{L_{1,m}^{j}}, \,\,\text{with}\,\,E\log\Big[\frac{1}{L_{1,m}^{j}}\Big]<0,\,\, E \log^{+}\Big|\frac{Q_{1,m}^{j}}{L_{1,m}^{j}}\Big|<\infty,$$
under Assumption 3(c) with $(a,b)$ replaced by $(ma,b)$.

Hence $S^{**}_{g_{t}^{j},m}$ is $O_{p}(1),$ along with $(L^{j}_{0,m},Q^{j}_{0,m},t-\tau_{g_{t}^{j}})$. Since $\Phi_{t}^{(ma)}\to\infty,$ as $t\to\infty$
$$\mathcal{L}\Big(\widetilde{G}^{(3,m)}_{t}\mid Y_{0}=i,Y_{t}=j\Big)=\exp\bigg[\frac{\log\Big|\frac{Q_{0,m}^{j}+S^{**}_{g_{t}^{j},m}}{L_{0,m}^{j}}+\frac{G_{j}^{(ma,b)}(t-\tau_{g_{t}^{j}})}{\Phi_{t}^{(ma)}}\Big|}{m\sqrt{t}}\bigg]\stackrel{w}{\to}\delta_{\{1\}}.$$
Using similar analysis on the asymptotic results of $G_{t}^{(1)},G_{t}^{(2)},$ proceeding as the rest of the proof of Theorem 2(a), the assertion of  Lemma 11.1(a) will follow.
\end{proof}

\subsubsection{Proof of \textbf{Lemma 11.1(b)}}

Let $\{(L_{i,m}^{j},Q_{i,m}^{j}): i\ge 0\}$ be defined as in Lemma 11.1(a) and for any $n\ge 1,$ denote $Z_{n,m}^{\otimes j}:=\max_{1\le k\le n}\Big\{\Big(\prod_{i=1}^{k-1}L_{i,m}^{j}\Big)|Q_{k,m}^{j}|\Big\},$ and 
$$P^{(j)}_{n,m}:=\sum_{k=1}^{n}\bigg(\prod_{i=1}^{k-1}L_{i,m}^{j}\bigg)Q_{k,m}^{j},\s\widetilde{P}^{(j)}_{n,m}:=\sum_{k=1}^{n}\bigg(\prod_{i=k+1}^{n}L_{i,m}^{j}\bigg)Q_{k,m}^{j}.$$

\begin{proof}
Using a similar approach and notations of the proof of Theorem 2(b), we observe that 
\begin{align}
&\bigg(\log\Phi_{t}^{(a)},\frac{\log |I_{t}^{(ma,b)}|}{m}\bigg)\stackrel{\mathcal{L}(\cdot|Y_{0}=i, Y_{t}=j)}{=}
\bigg(\log L^{j}_{0}+\log\bigg(\prod_{k=1}^{g^{j}_{t}}L^{j}_{k}\bigg)+\log(e^{ -a(j)(t-\tau_{g_{t}^{j}}^{j})}),\non\\ &\s\s\s\s\s\s\frac{1}{m}\log\bigg|e^{-ma(j)(t-\tau_{g_{t}^{j}}^{j})}\bigg[\bigg(\prod_{k=1}^{g_{t}^{j}}L^{j}_{k,m}\bigg)^{} Q^{j}_{0,m}+\widetilde{P}^{(j)}_{g_{t}^{j},m}\bigg]+G_{j}^{(ma,b^{})}\big(t-\tau_{g_{t}^{j}}^{j}\big)\bigg|\bigg)\non\\
&\s\s\s\stackrel{d}{=}\bigg(S_{g_{t}^{j}}^{\otimes j}+[\log L^{j}_{0}-a(j)(t-\tau_{g_{t}^{j}}^{j})]\,\, , \,\,\frac{\log Z_{g_{t}^{j},m}^{\otimes j}}{m}+ \frac{1}{m}\log\bigg|B_{t,m}^{(1)}+\frac{B_{t,m}^{(2)}}{Z_{g_{t}^{j},m}^{\otimes j}}P_{g_{t}^{j},m}^{(j)}\bigg|\bigg)\label{T2be1}
\end{align}

where 
\beqn
B_{t,m}^{(1)}&=&\frac{e^{-ma(j)(t-\tau_{g_{t}^{j}}^{j})}\bigg(\prod_{k=1}^{g_{t}^{j}}L^{j}_{k,m}\bigg)^{} Q^{j}_{0,m}+G_{j}^{(ma,b^{})}\big(t-\tau_{g_{t}^{j}}^{j}\big) }{Z_{g_{t}^{j},m}^{\otimes j}},\non\\
B_{t,m}^{(2)}&=&e^{-ma(j)(t-\tau_{g_{t}^{j}}^{j})},\non
\eeqn
and denote $\bigg(B_{t,m}^{(1)}+B_{t,m}^{(2)}\frac{P_{g_{t}^{j},m}^{(j)}}{Z_{g_{t}^{j},m}^{\otimes j}}\bigg)$ by $B^{(3)}_{t,m}.$ We now consider the scaled limit of $\Big(\frac{\log\Phi_{t}^{(a)}}{\sqrt{t}},\frac{\log |I_{t}^{(ma,b)}|}{m\sqrt{t}}\Big)$ as $t\to\infty.$ Observe that
$$B_{t,m}^{(1)}=\frac{e^{-ma(j)(t-\tau_{g_{t}^{j}}^{j})}\bigg(\prod_{k=1}^{g_{t}^{j}}L^{j}_{k,m}\bigg)^{} Q^{j}_{0,m}+G_{j}^{(ma,b^{})}\big(t-\tau_{g_{t}^{j}}^{j}\big)}{\max_{1\le k\le g_{t}^{j}}\Big(\prod_{i=1}^{k-1}L_{i,m}^{j}\Big)|Q_{k,m}^{j}|},$$
and $B_{t,m}^{(2)}$ are both $O_{p}(1)$ by construction, and for each $1\le k\le n,$ 
$$- \max_{1\le k\le g_{t}^{j}}\Big(\prod_{i=1}^{k-1}L_{i,m}^{j}\Big)|Q_{k,m}^{j}|\le Q_{k,m}^{j}\Big(\prod_{i=1}^{k-1}L_{i,m}^{j}\Big)\le \max_{1\le k\le g_{t}^{j}}\Big(\prod_{i=1}^{k-1}L_{i,m}^{j}\Big)|Q_{k,m}^{j}|.$$ 

 Hence for all $1\le k\le n,$ $Q_{k,m}^{j}\Big(\prod_{i=1}^{k-1}L_{i,m}^{j}\Big)\in\Big(-Z_{n,m}^{\otimes j},Z_{n,m}^{\otimes j}\Big).$
This implies that for all $n\ge 1,$ $$-n Z_{n,m}^{\otimes j}\le P_{n,m}^{(j)} \le n Z_{n,m}^{\otimes j},$$  which further implies that $\Big|\frac{P_{g_{t}^{j},m}^{(j)}}{Z_{g_{t}^{j},m}^{\otimes j}}\Big|\le g_{t}^{j}$; hence $|B^{(3)}_{t,m}|=O_{p}(t)$ and strictly positive. As $t\to\infty,$
$$\frac{\log|B_{t,m}^{(3)}|}{m\sqrt{t}}=\frac{\log \Big|B_{t,m}^{(1)}+B_{t,m}^{(2)}\frac{P_{g_{t}^{j},m}^{(j)}}{Z_{g_{t}^{j},m}^{\otimes j}}\Big|}{m\sqrt{t}}\stackrel{P}{\to} 0.$$
Now we proceed identically as done in Theorem 2(b) and observe that 
\beqn
&&P\bigg[\Big(\frac{\log \Phi_{t}^{(a)}}{\sqrt{t}},\frac{\log |I_{t}^{(ma,b)}|}{m\sqrt{t}}\Big)\in \cdot\s\mid Y_{0}=i\bigg]\non\\ \s\s\s\s &&=\sum_{j\in S} P[Y_{t}=j|Y_{0}=i]P\Big[\Big(\frac{\log \Phi_{t}^{(a)}}{\sqrt{t}},\frac{\log |I_{t}^{(ma,b)}|}{m\sqrt{t}}\Big)\in \cdot\s \mid Y_{0}=i, Y_{t}=j\Big]\non\\
&&=\sum_{j\in S}\pi_{j}\lim_{t\to\infty} P\bigg[\bigg(\frac{S_{g_{t}^{j}}^{\otimes j}}{\sqrt{t}},\frac{\log Z_{g_{t}^{j},m}^{\otimes j}}{m\sqrt{t}}\bigg)\in \cdot\s\mid Y_{0}=i\bigg],\non
\eeqn
by applying Lemma 6.1, with 
$$ L_{t}^{(1)}:=\bigg(\frac{[\log L^{j}_{0}-a(j)(t-\tau_{g_{t}^{j}}^{j})]}{\sqrt{t}} , \frac{\log|B_{t,m}^{(3)}|}{m\sqrt{t}}\bigg)',H_{t}:=\bigg(\frac{S_{g_{t}^{j}}^{\otimes j}}{\sqrt{t}},\frac{\log Z_{g_{t}^{j},m}^{\otimes j}}{m\sqrt{t}}\bigg)',$$
and $L_{t}^{(2)}$ defined as in Theorem 2(b) (with the verification lemma that $H_{t}$ satisfies (6.1), using ideas  similar to Lemma \ref{lemverify2b}(c)). 

For any two sequences $\{u_{k}\}_{k\ge 1},\{v_{k}\}_{k\ge 1},$ one has $$\max\{u_{k}\} -  \max\{v_{k}\}\le \max\{u_{k}+v_{k}\}\le \max\{u_{k}\} +  \max\{v_{k}\}.$$ As a result we can bound 
\beqn
\log Z_{g_{t}^{j},m}^{\otimes j}&=& \max_{1\le k\le g_{t}^{j}}\Big\{\log \Big(\prod_{i=1}^{K-1}L_{i,m}^{j}\Big)|Q_{k,m}^{j}|\Big\}\non\\&=&\max_{1\le k\le g_{t}^{j}}\{\sum_{i=1}^{k-1}\log L_{i,m}^{j}+\log|Q_{k,m}^{j}|\}\non\\&=&\max_{1\le k\le g_{t}^{j}}\{m\sum_{i=1}^{k-1}\log L_{i}^{j}+\log|Q_{k,m}^{j}|\},\non
\eeqn
 in the following manner

\beqn
\,\,\,\,\frac{M_{g_{t}^{j}}^{\otimes j}}{\sqrt{t}} -\frac{\max_{1\le k\le g_{t}^{j}}\log|Q_{k,m}^{j}|}{m\sqrt{t}}\le \frac{\log Z_{g_{t}^{j},m}^{\otimes j}}{m\sqrt{t}}\le \frac{M_{g_{t}^{j}}^{\otimes j}}{\sqrt{t}}+\frac{\max_{1\le k\le g_{t}^{j}}\log|Q_{k,m}^{j}|}{m\sqrt{t}}.\label{11.11}
\eeqn

Since (4.3) holds with $(a,b)$ replaced by $(ma,b),$ it follows that $\frac{\max_{1\le k\le g_{t}^{j}}\log|Q_{k,m}^{j}|}{m\sqrt{t}}\stackrel{a.s}{\to}0$
 (following arguments similar to Lemma \ref{lemverify2b}(a)). We get the final result by applying the weak limit of $\Big(\frac{S_{g_{t}^{j}}^{\otimes j}}{\sqrt{t}},\frac{M_{g_{t}^{j}}^{\otimes j}}{\sqrt{t}}\Big)$ by using Lemma \ref{lemverify2b}(b).
\end{proof}

\subsection{Proof of \textbf{Corollary 5.2}}
\begin{proof} Proceeding similar to the proof of Theorem 5.9, integrating in the range $[0,t]$ yields,
\beqn
X_{t}&=&X_{0}e^{-\int_{0}^{t}a(Y_{s})ds}+\int_{0}^{t}b(Y_{s})e^{-\int_{s}^{t}a(Y_{r})dr}d\mathcal{\bf{S}}_{\alpha^{*},0}(s)\non\\
&\stackrel{d}{=}& X_{0}\Phi^{(a)}_{t}+|I^{(\alpha^{*} a,b^{^{\alpha^{*}}})}_{t}|^{1/\alpha^{*}} \mathcal{S}_{\alpha^{*}}(1,0,0)\label{CC1e}
\eeqn

Using the continuous mapping theorem on Theorem 1, part (A) will follow.

For the rest of part (B), (C), (D) we observe that marginal evolution of the process at \eqref{CC1e} resembles with $Z_{t}^{(2)}$ in Corollary 4.5 of \cite{Majumder2024semi} with $a_{2}=X_{0},b_{2}= \mathcal{S}_{\alpha^{*}}(1,0,0), m=\alpha^{*};$ hence all the results will follow by replacing $(a,b^{})$ by $(a,b^{^{\alpha^{*}}})$ in Corollary 4.5 for $|Z_{t}^{(2)}|$. 
\end{proof}

\section{Proofs of the lemmas in the Appendix with some additional results}

Proofs of Lemma 12.1, Proposition 2.2, Lemma 12.2, Lemma 12.3(a),(b),(c), and one supplementary result are given here.

\subsection{Proof of \textbf{Lemma 12.1}}
\begin{proof}
 We use the notations similar to section 8 of \cite{cinlar1969markov}. Define the matrix of functions $A(t):=(A_{ij}(t): i,j\in S)$ such that $A_{ij}(t):=Q_{ij}(t)=P\big[J_{1}=j, T_{1}\le t\mid J_{0}=j\big]=P_{ij}F_{ij}(t).$ $A$ represents the semi-Markov matrix. For any two matrix valued functions $B(t):=(B_{ij}(t): i,j\in S), C(t):=(C_{ij}(t): i,j\in S),$ by the convolution $B\star C$ we denote the matrix valued functions where for each $i,j\in S,\s (B\star C)_{ij}(t):=\int_{0}^{t}\sum_{k\in S}B_{ik}(du)C_{kj}(t-u)=\int_{0}^{t}\sum_{k\in S}B_{ik}(t-u)C_{kj}(du).$ This manner recuresively for $n\ge 1,$ we have $A^{(n)}=A\star A^{(n-1)}$ and the interpretation of $A_{ij}^{(n)}(t)= P[Y_{n}=j, T_{n} \le t\mid Y_{0}=i].$ Define the renewal function at state $j$ starting from $i$ as $R_{ij}(t):=E \big[(g_{t}^{j}+1)1_{\{t\ge \tau_{0}^{j}\}}\mid Y_{0}=i\big]=\sum_{n\ge 0} A_{ij}^{(n)}(t)$ as the expected number of hitting at state $j$ before time $t.$  Using similar notations of \cite{cinlar1969markov} by $V^{-}(t)$ we denote $t-S_{N_{t}}.$ Note that on $\{Y_{t}=j\},$ $\tau_{g_{t}^{j}}=S_{N_{t}},$ hence we can safely write $P_{_{i}}[t-\tau_{g_{t}^{j}}^{j}>x, Y_{t}=j]=P_{_{i}}[t-S_{N_{t}}>x, Y_{t}=j]=P_{_{i}}[V^{-}(t)>x, Y_{t}=j].$ 

Observe that
\beqn
&P_{_{i}}[V^{-}(t)>x, Y_{t}=j]= P_{_{i}}[V^{-}(t)>x, Y_{t}=j, T_{1}>t]+P_{_{i}}[V^{-}(t)>x, Y_{t}=j, T_{1}<t]\non\\
&\s\s=\delta_{\{i=j\}}1_{\{t>x\}}P[T_{1}>t\mid Y_{0}=j]+\sum_{k\in S}\int_{0}^{t}P_{ik}F_{ik}(du)P_{k}\big[V^{-}(t-u)>x, Y_{t-u}=j\big]\non\\
&\,\,=\delta_{\{i=j\}}1_{\{t>x\}}\sum_{k\in S}P_{jk}\bar{F}_{jk}(t)+\sum_{k\in S}\int_{0}^{t}A_{ik}(du)P_{_{k}}\big[V^{-}(t-u)>x, Y_{t-u}=j\big].\non\label{reneq1}
\eeqn
Now for fixed $i,x,$ writing $U_{j}(t):=P_{_{i}}[V^{-}(t)>x, Y_{t}=j],$ and $$G_{j}(t):=\delta_{\{i=j\}}1_{\{t>x\}}\sum_{k\in S}P_{jk}\bar{F}_{jk}(t),$$ from aforementioned relation we have the renewal equation 
\beqn
U=G+A\star U.\label{Renewaleq2}
\eeqn Observe that $G_{j}(t)=\delta_{\{i=j\}}1_{\{t>x\}}\sum_{k\in S}P_{jk}\bar{F}_{jk}(t) = \delta_{\{i=j\}}1_{\{t>x\}}(1-F_{j}(t))\le (1-F_{j}(t)) =1-A_{j}(t)1,$ which implies that $G\le 1- A.1$ co-ordinate wise, and as a consequence of Theorem 3.17 of \cite{cinlar1969markov} there exists a solution of the equation \eqref{Renewaleq2}, and from Theorem 3.11 and display 4.9 of \cite{cinlar1969markov} the solution is unique if all states are conservative (which is granted under Assumption 1). The minimal solution satisfies
\beqn
P_{_{i}}[V^{-}(t)>x, Y_{t}=j]&=& \int_{0}^{t}R_{ij}(du)1_{\{t-u>x\}}(1-F_{j}(t-u))\non\\
&=& \int_{0}^{t}R_{ij}(du)1_{\{t-u>x\}}\sum_{k\in S}P_{jk}\bar{F}_{jk}(t-u).\non
\eeqn

Now using analog of Key-Renewal Theorem 6.3 of \cite{cinlar1969markov} gives
\beqn
\lim_{t\to\infty}P_{_{i}}[V^{-}(t)>x, Y_{t}=j]&=&\lim_{t\to\infty} \int_{0}^{t}R_{ij}(du)1_{\{t-u>x\}}\big(1-F_{j}(t-u)\big)\non\\
&=& \frac{\mu_{j}\int_{x}^{\infty}(1-F_{j}(y))dy}{\sum_{k\in S}\mu_{k}m_{k}},\label{reneweq3}
\eeqn

given the fact that state $j$ is recurrent (follows from Assumption 1) and the function $t\to 1_{\{t>x\}}(1-F_{j}(t))$ is directly Riemann integrable (in short d.r.i).  Definition and the necessary-sufficient conditions of `d.r.i functions' can be found in Section 4 (Page 153 of chapter V (Renewal theory) of \cite{asmussen2008applied}) and Proposition 4.1 of \cite{asmussen2008applied} respectively. Since the distribution function $F_{j}(t)=\sum_{k\in S}P_{jk}F_{jk}(t)$ has at most countable number of jumps so it is continuous almost everywhere w.r.t Lebesgue measure and the function $1_{\{t>x\}}(1-F_{j}(t))$ is bounded as well. So using Proposition 4.1(i) of \cite{asmussen2008applied} the function $1_{\{t>x\}}(1-F_{j}(t))$ is d.r.i hence the above assertion \eqref{reneweq3} holds.
\end{proof}

\subsection{Proof of \textbf{Proposition 2.2}}
\begin{proof}
Note that the distribution of $\tau_{g^{j}_{t}+2}^{j}-\tau_{g^{j}_{t}+1}^{j}$ is identical as $\tau^{j}_{2}-\tau_{1}^{j},$  since 
\beqn
P[\tau_{g^{j}_{t}+2}^{j}-\tau_{g^{j}_{t}+1}^{j}>x]&=&\sum_{n=0}^{\infty}P[\tau_{g^{j}_{t}+2}^{j}-\tau_{g^{j}_{t}+1}^{j}>x\mid g^{n}_{t}=n]P[g^{n}_{t}=n]\non\\
&=&\sum_{n=0}^{\infty}P[\tau_{n+2}^{j}-\tau_{n+1}^{j}>x\mid g^{j}_{t}=n]P[g^{j}_{t}=n]\non\\
&\stackrel{(a)}{=}&\sum_{n=0}^{\infty}P[\tau_{1}^{j}-\tau_{0}^{j}>x]P[g^{j}_{t}=n]=P[\tau_{1}^{j}-\tau_{0}^{j}>x]\label{eres2}
\eeqn
where (a) holds from the fact that the event $\{g_{t}^{j}=n\}=\{\tau_{n}^{j}\le t <\tau_{n+1}^{j}\}$ is independent with the event $\{\tau_{n+2}^{j}-\tau_{n+1}^{j}>x\}$  due to the regeneration property, and the latter is identical in probability with $\{\tau_{1}^{j}-\tau_{0}^{j}>x\},$ proving the assertion of part (i). 

Similar to the steps done in \eqref{eres2} one can show that any functional of $Y$ in $\{s: s\ge \tau^{j}_{g_{t}^{j}+1}\}$ is independent of $g_{t}^{j}$ and identically distributed as that functional on $\{s:s\ge \tau_{0}^{j}\}$ for any $t\ge \tau_{0}^{j}$ for part (ii).

For part (iii), observe that $\{g^{j}_{t}=n, Y_{t}=j'\}=\{\tau^{j}_{n}\le t <\tau^{j}_{n+1},Y_{t}=j'\}.$ For any two sets $A,B$
\beqn
&&P[Y^{(1)}_{t}\in A ,Y^{(2)}_{t} \in B\mid K_{t},Y_{t}=j']\non\\&\s\s=&\sum_{n=0}^{\infty}P\big[Y^{(1)}_{t}\in A ,Y^{(2)}_{t} \in B\mid K_{t}, Y_{t}=j',g^{j}_{t}=n\big]P[g^{j}_{t}=n\mid K_{t},Y_{t}=j']\non\\
&\s\s\stackrel{(a)}{=}& \sum_{n=0}^{\infty}P\big[Y^{(1)}_{t}\in A\mid  K_{t},g^{j}_{t}=n, Y_{t}=j'\big] P\big[Y^{(2)}_{t} \in B\mid  K_{t}, g^{j}_{t}=n,  Y_{t}=j'\big]\non\\&&\s\times P[g^{j}_{t}=n\mid  K_{t}, Y_{t}=j']\non\\
&\s\s\stackrel{(b)}{=}& \sum_{n=0}^{\infty}P\big[Y^{(1)}_{t}\in A\mid  K_{t}, g^{j}_{t}=n, Y_{t}=j'\big] P\big[Y^{(2)}_{t} \in B\mid  K_{t}\big]P[g^{j}_{t}=n\mid  K_{t}, Y_{t}=j']\non\\
&\s\s=& P[Y^{(1)}_{t}\in A\mid  K_{t}, Y_{t}=j'\big]P\big[Y^{(2)}_{t} \in B \mid K_{t}].\non
\eeqn
Equality in $(a)$ follows from the fact that conditioned on $\{K_{t},g^{j}_{t}=n,Y_{t}=j'\}$ the processes $Y^{(1)}_{t},Y^{(2)}_t$ are respectively $\sigma\{\mathcal{H}_{i}:i\le n\}$ and $\sigma\{\mathcal{H}_{i}: i\ge n+2\}$ measurable which are independent sigma fields. Equality in $(b)$ follows as conditioned on $\{K_{t}\}$ and $\{g_{t}^{j}=n,Y_{t}=j'\}$ $Y_{t}^{(2)}$ is $\sigma\{\mathcal{H}_{i}:i\ge n+2\}$ measurable and independent with $\{g_{t}^{j}=n,Y_{t}=j'\}=\{\tau_{n}^{j}\le t<\tau_{n+1}^{j},Y_{t}=j'\}$ as the latter is $\sigma\{\mathcal{H}_{i}: i\le n+1\}$ measurable.

\end{proof}

\subsection{Proof of \textbf{Lemma 12.2}}

\begin{proof}
We prove that $\log^{}|L^{j}_{0}|=O_{P}(1)$ and $|Q^{j}_{0}|=O_{P}(1)$. Let, for arbitrary $i,j\in S$,  $A^{(k)}_{ij}$ denote the event that $Y$ visits the state $j$ in $k$-th excursion from state $i$ to itself (i.e in time interval $[\tau_{k-1}^{i},\tau_{k}^{i})$). Clearly $0<P\big(A^{(k)}_{ij}\big)<1$ for any $k\in\mathbb{N}$ since $Y$ is irreducible in $S$. Observe that on $\{Y_{0}=i\}$ for any arbitrary $i\in S$, 
$$
\log^{}|L^{j}_{0}|=-\int_{0}^{\tau_{0}^{j}}a(Y_{s})ds\le \sum_{l=1}^{k^{*}-1}\log L_{k}^{i}+ \int_{\tau_{k^{*}-1}^{i}}^{\tau_{k^{*}}^{i}}|a(Y_{s})|ds,
\quad k^{*}=\inf\Big\{k\ge 1: 1_{A^{(k)}_{ij}}=1\Big\}.
$$ 
Observe that above upper bound is a geometric sum of random variables having finite expectation which proves the first assertion that $\log^{}|L^{j}_{0}|=O_{P}(1)$.
Similarly, on $\{Y_{0}=i\}$,  
\begin{align*}
Q^{j}_{0}&=\int_{0}^{\tau_{0}^{j}}b^{}(Y_{s})e^{-\int_{s}^{\tau_{0}^{j}}a(Y_{r})dr}ds\\
&=e^{(\tau_{0}^{j}- \tau_{k^{*}}^{i})}\sum_{l=1}^{k^{*}-1}\Big(\prod_{n=l+1}^{k^{*}}L_{n}^{i}\Big)^{} Q_{l}^{i}+\int_{\tau_{k*-1}^{i}}^{\tau_{0}^{j}}b^{}(Y_{s})e^{-\int_{s}^{\tau_{0}^{j}}a(Y_{r})dr}ds.
\end{align*}
Note that $[\tau_{k*-1}^{i},\tau_{0}^{j})\subset [\tau_{k*-1}^{i},\tau_{k*}^{i})$ and that 
$$
\sum_{l=1}^{\infty}\Big(\prod_{n=l+1}^{\infty}L_{n}^{i}\Big)^{}Q_{l}^{i}
$$
is absolutely convergent due to Assumption 2. Hence, $|Q^{j}_{0}|=O_{P}(1)$. 
\end{proof}

\subsection{Statement of \textbf{Lemma 12.3}}
\begin{enumerate}[(a)]
\item Suppose Assumption 1 holds and $c_{\cdot}$ is a function defined on $\mathbb{N}$ as (2.15) and its definition is extended to the whole $\R_{\ge 0}$ for any arbitrary slowly varying function $L(\cdot)$. Let $g^{j}_{t}$ be defined as the number of total renewals to the state $j,$ before time $t.$ Then for any fixed $s>0,$ $$\lim_{t\to\infty}c^{-1}_{g_{st}^{j}}/c^{-1}_{t}\stackrel{a.s}{\to} \Big(\frac{s}{E|\mathfrak{I}^{j}_{1}|}\Big)^{-\frac{1}{\alpha}}.$$
\item For any function $L(\cdot)$ slowly varying at $+\infty,$ and any two functions $a_{1}(\cdot),b_{1}(\cdot):\R_{\ge 0}\to \R_{\ge 0}$ such that both $a_{1}(t),b_{1}(t)\to \infty$ as $t\to\infty$ and $a_{1}\le b_{1};$ for any $\eps>0$ there exists  $t_{0,\eps}\ge 0,$ such that for all $t\ge t_{0,\eps},$ 
\beqn
(1-\eps)\Big(\frac{b_{1}(t)}{a_{1}(t)}\Big)^{-\eps}\le\frac{L(b_{1}(t))}{L(a_{1}(t))}\le (1+\eps)\Big(\frac{b_{1}(t)}{a_{1}(t)}\Big)^{\eps}\label{regVp3}
\eeqn
\item For any function $L(\cdot)$ slowly varying at $t\to\infty,$  $\lim_{t\to\infty}\frac{\log L(t)}{\log t}\to 0$. 
\end{enumerate}
\subsubsection{Proof of \textbf{Lemma 12.3(a)}}
\begin{proof}
Observe that,
$$c^{-1}_{g_{st}^{j}}/c^{-1}_{t}= \frac{(g_{st}^{j})^{-\frac{1}{\alpha}}L_{0}^{-1}(g_{st})}{t^{-\frac{1}{\alpha}}L^{-1}_{0}(t)},$$
where $L_{0}(\cdot)$ depends on $L(\cdot)$ through (2.15) as it is defined as $c_{n}\sim n^{-1/\alpha}L_{0}(n).$

The result follows from the fact $\lim_{t\to\infty}\frac{g_{st}}{\frac{s}{E|\mathfrak{I}^{j}_{1}|}t}\stackrel{a.s}{\to}1,$ and $\lim_{t\to\infty}\frac{L_{0}(t)}{L_{0}(g^{j}_{st})} \stackrel{a.s}{\to} 1$ as an application of Potter's bound. To see the second result observe that given a small $\epsilon_{1}\in \Big(0,\frac{s}{E|\mathfrak{I}_{1}^{j}|}\Big),$ for each $\omega\in A$ with $P[A]=1$  there exists $t_{1}(\omega,\epsilon_{1})$ such that $\frac{g_{st}}{ \frac{s}{E|\mathfrak{I}^{j}_{1}|} t}\ge 1-\frac{\epsilon_{1}}{\frac{s}{E|\mathfrak{I}^{j}_{1}|}},$ for all $t\ge t_{1}(\omega,\epsilon_{1}).$ Since $$\frac{L_{0}(t)}{L_{0}(g^{}_{st})}=\frac{L_{0}(t)}{L_{0}\big(t\big(\frac{s}{E|\mathfrak{I}^{j}_{1}|}-\epsilon_{1}\big)\big)}\frac{L_{0}\big(t\big(\frac{s}{E|\mathfrak{I}^{j}_{1}|}-\epsilon_{1}\big)\big)}{L_{0}(g^{}_{st})},$$
the first term goes to $1$ as $t\to\infty$ by definition. The second term can be bounded by using Potter's bound for any slowly varying function $L_{0}(\cdot)$ which suggests that for all $\epsilon>0,$ there exists $t_{0}>0,$ such that 
$$(1-\epsilon)\bigg(\frac{\frac{g_{st}^{j}}{ t} }{\frac{s}{E|\mathfrak{I}^{j}_{1}|}-\epsilon_{1}}\bigg)^{-\epsilon}\le\frac{L_{0}(g^{j}_{st})}{L_{0}\big(t\big(\frac{s}{E|\mathfrak{I}^{j}_{1}|}-\epsilon_{1}\big)\big)}\le (1+\epsilon)\bigg(\frac{\frac{g_{st}^{j}}{\frac{s}{E|\mathfrak{I}^{j}_{1}|}t} }{1-\frac{\epsilon_{1}}{\frac{s}{E|\mathfrak{I}^{j}_{1}|}}}\bigg)^{\epsilon}$$
for all $t\ge t_{0}\vee t_{1}(\omega,\epsilon_{1}).$ As both $\epsilon_{1},\epsilon$ can be chosen arbitrarily small, both sides have an almost sure limit to $1,$ for large $t$ proving the part (b) of the Lemma.
\end{proof}
\subsubsection{Proof of \textbf{Lemma 12.3(b)}}
\begin{proof}
Using Karamata's representation theorem for a slowly varying function $L(\cdot)$ at $\infty,$ there exist $x_{0}\ge 0$, functions $c_{1}(\cdot),\eta(\cdot)$ such that there exists a constant $c_{0},$ for which $\lim_{x\to\infty}c_{1}(x)=c_{0},\lim_{x\to\infty}\eta(x)=0,$ and $L(\cdot)$ can be represented as $L(x)=c_{1}(x)e^{\int_{x_{0}}^{x}\frac{\eta(s)}{s}ds}.$ 

Using this representation for any $t\ge 0,$
$$\frac{L(b_{1}(t))}{L(a_{1}(t))}=\frac{c_{1}(b_{1}(t))}{c_{1}(a_{1}(t))} \frac{e^{\int_{x_{0}}^{b_{1}(t)}\frac{\eta(s)}{s}ds}}{e^{\int_{x_{0}}^{a_{1}(t)}\frac{\eta(s)}{s}ds}}=\frac{c_{1}(b_{1}(t))}{c_{1}(a_{1}(t))}e^{\int_{1}^{\frac{b_{1}(t)}{a_{1}(t)}}\frac{\eta(a_{1}(t) s)}{s}ds}.$$
Since both $(a_{1}(t),b_{1}(t))\to(\infty,\infty)$ as $t\to\infty,$ for any $\epsilon>0,$ there exists $t_{0,\epsilon}\ge 0$ such that 
$$\Big|\frac{c_{1}(b_{1}(t))}{c_{1}(a_{1}(t))}-1\Big|<\epsilon\s\& \s |\eta(a_{1}(t)s)|\le\epsilon\s \forall \,\, t\ge t_{0,\epsilon}.$$
If $|\eta(a_{1}(t)s)|\le\epsilon,$ it implies that 
$$-\epsilon\log\Big(\frac{b_{1}(t)}{a_{1}(t)}\Big)\le \int_{1}^{\frac{b_{1}(t)}{a_{1}(t)}}\frac{\eta(a_{1}(t) s)}{s}ds \le  \epsilon\log\Big(\frac{b_{1}(t)}{a_{1}(t)}\Big),$$
which together with $\Big|\frac{c_{1}(b_{1}(t))}{c_{1}(a_{1}(t))}-1\Big|<\epsilon,$ prove the assertion in \eqref{regVp3}.
\end{proof}
\subsubsection{Proof of \textbf{Lemma 12.3(c)}}
\begin{proof}
For any $L(t)$ slowly varying at $t\to\infty,$ we show that $\lim_{t\to\infty}\frac{\log L(t)}{\log t}\to 0.$ Suppose given any small $\delta>0,$ there exists $t_{\delta}$ such that for all $t\ge t_{\delta},$ $|\eta(t)|\le\delta$ and $|c_{1}(t)-c_{0}|\le \delta.$

Using Karamata's presentation $$\frac{\log L(t)}{\log t}= \frac{\log c_{1}(t)}{\log(t)}+\frac{1}{\log(t)}\Big[\int_{x_{0}}^{t_{\delta}}\frac{\eta(s)}{s}ds+\int_{t_{\delta}}^{t}\frac{\eta(s)}{s}ds\Big].$$

It follows that for large $t,$ the first and second term in the RHS go to $0,$ and the third term
$$\frac{1}{\log(t)}\int_{t_{\delta}}^{t}\frac{\eta(s)}{s}ds\le \delta \frac{\log(t) - \log(t_{\delta})}{\log(t)}\le \delta,$$
for any arbitrary $\delta>0.$ Hence the assertion follows as $t\to\infty.$
\end{proof}

\subsection{Proof of Lemma \ref{regV21}}
\begin{proof}
To prove first assertion observe that for any $\delta>0,$ $$\{\widetilde{X}_{1}>(1+\delta)x, \widetilde{X}_{2}<\delta x\}\subset\{\widetilde{X}_{1}-\widetilde{X}_{2}>x\}\subset \{\widetilde{X}_{1}>x\}.$$ These inequalities together imply that 
\beqn
1&\ge& \frac{P\big[\widetilde{X}_{1}-\widetilde{X}_{2}>x\big]}{P[\widetilde{X}_{1}>x]}\non\\
&\ge& \frac{P\big[\widetilde{X}_{1}>(1+\delta)x, \widetilde{X}_{2}<\delta x\big]}{P[\widetilde{X}_{1}>x]}\non\\
&=& \frac{P\big[\widetilde{X}_{1}>(1+\delta)x\big]}{P[\widetilde{X}_{1}>x]}\Big[1-P[\widetilde{X}_{2}\ge \delta x]\Big]=(1+\delta)^{-\alpha}[1-\delta^{-\alpha}O(x^{-\alpha})]+o(1),\label{r1}
\eeqn
where equality in \eqref{r1} follows from the independence of $\{\widetilde{X}_{i}:i=1,\ldots,n\}$. Taking limit as $x\to\infty$ and having $\delta>0$ arbitrarily small proves the first assertion $$P[\widetilde{X}_{1}-\widetilde{X}_{2}>x]\sim P[\widetilde{X}_{1}>x].$$ The assertion $P[\widetilde{X}_{1}-\widetilde{X}_{2}<-x]\sim P[\widetilde{X}_{2}>x]$ follows similarly by changing signs. Second and third assertion follow as a corollary of Lemma 1.3.4 and Remark 1.3.5 of \cite{mikosch1999regular} respectively.
\end{proof}

\subsection{Proof of \textbf{Lemma 12.4}}
\begin{proof}
Define two functions $G_{1},G_{2}:\R_{\ge 0}^{2}\to\R,$ such that $$G_{1}(x,y)=\max\{\log x,\log y\},G_{2}(x,y)=\log \max\{x,y\}.$$ It is easy to see that 
\begin{align}
&P[G_{1}(X,Y)=G_{2}(X,Y)]=P[G_{1}(X,Y)=G_{2}(X,Y), X>Y]\non\\&\s +P[G_{1}(X,Y)=G_{2}(X,Y), X=Y]+P[G_{1}(X,Y)=G_{2}(X,Y), X<Y]\non
\end{align}
 and each of these probabilities in the RHS are 
\beqn
P[G_{1}(X,Y)=G_{2}(X,Y), X>Y]&=&P[X>Y],\non\\
 P[G_{1}(X,Y)=G_{2}(X,Y), X<Y]&=&P[X<Y],\non\\ P[G_{1}(X,Y)=G_{2}(X,Y), X=Y]&=&P[X=Y]\non
\eeqn
 and after summing one can see that $P[G_{1}(X,Y)=G_{2}(X,Y)]=1.$
\end{proof}


\end{document}